\documentclass[reqno]{amsart}
\usepackage{amsmath}
\usepackage{amsthm}
\usepackage{amssymb}
\usepackage{stmaryrd}
\usepackage{afterpage}
\usepackage{rotating}    
\usepackage{adjustbox}   
\usepackage{booktabs}
\usepackage[utf8]{inputenc}
\usepackage[T1]{fontenc}
\usepackage{lmodern}
\usepackage{eucal}
\pdfoutput=1 
\usepackage{microtype}
\usepackage{url}
\usepackage{enumitem}
\usepackage{tikz}
\usetikzlibrary{arrows.meta}
\usepackage{tikz-cd}
\usepackage{bbm}
\usepackage{longtable,booktabs}

\usepackage[letterpaper, left = 1.5in, right = 1.5in, top = 1.5in, bottom = 1.5in]{geometry}

\usepackage{pdflscape}

\usepackage{hyperref}

\usepackage{xcolor}
\hypersetup{
    colorlinks,
    linkcolor={red!50!black},
    citecolor={blue!50!black},
    urlcolor={blue!80!black}
}
\definecolor{chartgray}{gray}{0.4}
\definecolor{darkcyan}{rgb}{0, 0.7, 0.7}
\definecolor{darkgreen}{rgb}{0, 0.65, 0}
\definecolor{truemagenta}{rgb}{1, 0, 1}

\usepackage[capitalise,nameinlink]{cleveref}
\crefname{subsection}{Subsection}{Subsections}
\crefname{subsubsection}{Subsubsection}{Subsubsections}

\usepackage{graphicx}

\usepackage{lipsum}     
\usepackage{booktabs} 
\usepackage{rotating}
\usepackage{tabularx}

\theoremstyle{definition}
\newtheorem{theorem}{Theorem}[section]

\newtheorem{defn}[theorem]{Definition}

\newtheorem{cor}[theorem]{Corollary}
\newtheorem{lemma}[theorem]{Lemma}
\newtheorem{prop}[theorem]{Proposition}
\newtheorem{rmk}[theorem]{Remark}

\newtheorem*{rmk*}{Remark}
\newtheorem*{ex*}{Example}
\newtheorem*{theorem*}{Theorem}
\newtheorem*{defn*}{Definition}

\newcommand{\bbZ}{\mathbb{Z}}

\newcommand{\bbF}{\mathbb{F}}
\newcommand{\bbN}{\mathbb{N}}

\newcommand{\bbR}{\mathbb{R}}

\newcommand{\calA}{\mathcal{A}}

\newcommand{\Sp}{\mathcal{S}\mathrm{p}}

\newcommand{\Ext}{\operatorname{Ext}}

\newcommand{\im}{\operatorname{Im}}
\renewcommand{\ker}{\operatorname{Ker}}

\newcommand{\id}{\operatorname{id}}

\newcommand{\tors}{\mathrm{tors}}

\newcommand{\floor}[1]{\left\lfloor#1\right\rfloor}

\makeatletter
\DeclareRobustCommand{\tvdots}{%
  \vbox{\baselineskip4\p@\lineskiplimit\bbZ@\kern0\p@\hbox{.}\hbox{.}\hbox{.}}}
\makeatother

\makeatletter
\@namedef{subjclassname@2020}{\textup{2020} Mathematics Subject Classification}
\makeatother

\makeatletter
\newcommand{\raisemath}[1]{\mathpalette{\raisem@th{#1}}}
\newcommand{\raisem@th}[3]{\raisebox{#1}{$#2#3$}}
\makeatother

\begin{document}

\title{$C_3$-equivariant stable stems}
\author{Yueshi Hou}
\address{Department of Mathematics, University of California San Diego, La Jolla, CA 92093, USA}
\email{yuh091@ucsd.edu}

\author{Shangjie Zhang}
\address{Department of Mathematics, University of California San Diego, La Jolla, CA 92093, USA}
\email{shz046@ucsd.edu}

\begin{abstract}
 We compute the spoke-graded $C_3$-equivariant stable homotopy groups of spheres $\pi_{i, j}^{C_3}$, for stems less than 25 (i.e. $i\leq 25$) and for weights between -16 and 16 (i.e. $-16\leq j\leq 16$). In particular, for $j=2k$, this corresponds to the usual $RO(C_3)$-graded homotopy groups of spheres $\pi^{C_3}_{i-j+k\lambda}$ for some fixed 2-dimensional $C_3$-faithful representation $\lambda$. We also describe the geometric fixed point map $\Phi^{C_3}: \pi_{i, j}^{C_3}\to \pi_{i-j}^{cl}$ and the underlying map $Res: \pi_{i, j}^{C_3}\to \pi_{i}^{cl}$.
\end{abstract}

\maketitle

\tableofcontents

\section{Introduction}
Computations of the stable homotopy groups of spheres is one of the most important yet challenging problems in algebraic topology. A long line of work---from the foundational results of \cite{Serre1951fibres, adams1960nonexistence, toda1962composition, may1965steenrod, barratt1970differentials, Miller1977periodic,ravenel1986complexcobordism} to recent advances by \cite{isaksen2019stable, isaksenwangxu2023stable}---has greatly deepened our understanding of these groups. This fundamental question admits a natural generalization to the equivariant context. Naturally, equivariant stable stems were first explored for the simplest case, $C_2$, with foundational work by \cite{bredon1967equivariant, landweber1969equivariant} and more extensive calculations by \cite{arakiiriye1982equivariant}. In recent years, such computations have seen a renewed surge. Especially, with the aid from $\bbR$-motivic homotopy theory \cite{belmontisaksen2020rmotivic, belmontguillouisaksen2021c2,behrensshah2020c2}, Guillou--Isaksen \cite{guillouisaksen2024c2} have computed the $C_2$-equivariant stable stems to a greater range. However, despite the progress in the $C_2$-equivariant computations, much less is known about the equivariant stable stems for other groups, even for cyclic groups of odd prime orders. 

This paper aims to provide a first answer by studying the $C_3$-equivariant stable homotopy groups of spheres. Recall that the $C_3$-real representation ring $RO(C_3)$ has a basis consisting of the trivial representation and the 2-dimensional faithful real representation $\lambda$, characterized by rotating by $\frac{\pi}{3}$. Moreover, let $S^\Yright$ denote the cofiber of the collapsing map ${C_3}_+\to S^0$, called the spoke sphere \cite{hahnsengerwilson2023odd, angeliniknoll2023deformationborelequivarianthomotopy, balderrama2025cpnequivariantmahowaldinvariants}. We will consider not only the $RO(C_3)$-graded homotopy groups, but also the additional ``spoke-grading''. 

\begin{defn}\label{defofspokegradedhomotopygroup}
For any $C_3$-equivariant spectrum $X$, define the spoke-graded homotopy groups of $X$ to be 
\[\pi_{i,j}^{C_3}(X):= \begin{cases}[S^{i-j+k\lambda} \wedge S^\Yright,X]^{C_{3}}&\text{for} \hspace{3pt} j=2k+1\\ [S^{i-j+k\lambda},X]^{C_{3}}&\text{for} \hspace{3pt} j=2k.
\end{cases}
\]
We call the index $i$ stem and the index $j$ weight in resemblance to the motivic grading. In particular, for $j=2k$, this is exactly the $RO(C_3)$-graded homotopy groups $\pi^{C_3}_{i-j+k\lambda}(X)$. When $X = S^{0}$ we abbreviate these groups by $\pi^{C_3}_{i,j}$, which we will refer to hereafter as the $C_3$-equivariant stable stems. 
\end{defn}
This enhancement has the advantage of detecting the $C_3$-equivariant weak equivalences, as will be proved in \cref{spokedetectequivalence}, whereas the usual $RO(C_3)$-graded homotopy groups do not. Our main theorem can be summarized as follows.

\begin{theorem}\label{mainthm}
    For $0\leq i\leq 25$ and $-16\leq j\leq 16$, the 3-primary group structures of $\pi_{i, j}^{C_3}$ are summarized in \cref{ChartC_3stablestem}. 

    For $i<0$, $\pi^{C_3}_{i,j} \cong \pi^{cl}_{i-j}$, where the latter term denotes the classical stable stems. The torsion-free and $p$-primary ($p\neq 3$) information of $\pi^{C_3}_{i, j}$ is summarized in \cref{otherpcomplete} and \cref{integral}.
\end{theorem}

\subsection{Strategy}\label{strategy}
We will treat the $p$-primary part of $\pi_{i, j}^{C_3}$ separately for different $p$'s. The isotropy separation sequence
\[{EC_3}_+\to S^0\to \widetilde{EC_3}\]
splits if we invert 3, so the most interesting calculations happen $3$-adically. In this case, the above cofiber sequence induces the following long exact sequence in homotopy groups
\[\pi_{i-j+1}^{cl}\overset{M}{\to} \pi_{i, j}^{C_3}({EC_3}_+)\to \pi_{i, j}^{C_3}\overset{\Phi^{C_3}}{\to} \pi_{i-j}^{cl}\]
where $\pi^{cl}_*$ is the classical (3-primary) stable stems and $\Phi^{C_3}$ is the geometric fixed point map. As will be proved in \cref{isotropyLES}, $\pi_{i, j}^{C_3}({EC_3}_+)$ can be identified with another non-equivariant object $\pi_{i-j}^{cl}(BC_3)_{-j}^\infty$. Taking $\pi^{cl}_*$ as the main input, we will compute the groups $\pi_{i-j}^{cl}(BC_3)_{-j}^\infty$ using the Atiyah--Hirzebruch spectral sequences and the map $M$ which encodes the information of the ($3$-primary) Mahowald invariants \cite{mahowald1967metastable, mahowaldravenel1993root}. We also obtain the data for the fixed point map $\Phi^{C_3}:\pi_{i, j}^{C_3}\to \pi_{i-j}^{cl}$ and the underlying map $Res: \pi_{i, j}^{C_3} \to \pi_i^{cl}$ in \cref{aspokemodule}.

In \cite{behrens2006root}, Behrens computed the 3-primary Mahowald invariants by computing the $BP_*$-based algebraic Atiyah--Hirzebruch spectral sequences. In \cref{3}, we study the Atiyah--Hirzebruch spectral sequence under the perspective of filtered spectra. We not only recover the corresponding differentials of Behrens using different methods, but also prove some new differentials (\cref{leibnizrule}, \cref{alpha2tobeta1^3}, \cref{beta1tobeta2}).

\begin{rmk}\label{Explanation of methodology}
     In \cref{S^0borel}, $\pi_{i, j}^{C_3}$ can be identified with the homotopy group of another non-equivariant object $\pi_{i-j-1}^{cl}(BC_3)_{-\infty}^{-j-1}$. Although it does not appear so at a first glance, $(BC_3)_{-\infty}^{-j-1}$ is in fact a bounded below spectrum. It fits into the cofiber sequence
    \[(BC_3)_{-\infty}^{-j-1}\to (BC_3)_{-\infty}^{\infty}\to (BC_3)_{-j}^{\infty}\]
    where 3-adically the Segal conjecture \cite{gunawardena1980segal, ravenel1984segalconjecture, carlsson1983equivariant} gives the identification
    \[(BC_3)_{-\infty}^{\infty} \overset{\underset{\mathrm{def}}{}}{=} \displaystyle \lim_{j \to \infty} (BC_3)^{\infty}_{-j} \simeq S^{-1}. \]
    As shown in \cref{isotropyLES}, this cofiber sequence is exactly the non-equivariant counterpart of the $C_3$-isotropy separation sequence.
\end{rmk}

The above methodology was motivated by the discussion of the $C_2$-equivariant stable stems in \cite[Section 7]{behrensshah2020c2}. Notably, the first systematic computation of the $C_2$-equivariant stable stems in \cite{arakiiriye1982equivariant} used similar ideas. On the other hand, Guillou-Isaksen recently studied the genuine equivariant Adams spectral sequence and extended the computation of the $C_2$-equivariant stable stems to a much larger range \cite{guillouisaksen2024c2}. Despite the success in the $C_2$ case, it is hard to generalize the approach of genuine equivariant Adams spectral sequence to cyclic groups of prime orders, since the structure of the genuine equivariant (dual) Steenrod algebra is not flat over the cohomology of a point \cite{sankarwilson2021cpequivariantdualsteenrodalgebra, hukrizsombergzou2023zpequivariantdualsteenrodalgebra}. In particular, there is no immediate identification of the $E_2$-page of the genuine $C_3$-equivariant Adams spectral sequence with some $\Ext$-groups. 

Prior to our work, Szymik \cite{szymik2007equivariant} has computed the $C_3$-equivariant stable stems $\pi_{i, j}^{C_3}$ for $j=0, 2$ and $i\leq 6$ using the Borel equivariant Adams spectral sequence \cite{greenlees1988borelhomology}. Recent work of Angelini-Knoll--Behrens--Belmont--Kong \cite{angeliniknoll2023deformationborelequivarianthomotopy} has developed a deformation of the $C_3$-equivariant homotopy up to Borel completion, in analogy with the Artin--Tate $\bbR$-motivic homotopy category as a deformation of the $C_2$-equivariant homotopy category \cite{burklundhahnsenger2022galois}. According to Ma \cite{ma2024borel}, the $E_2$-page of the Borel Adams spectral sequence is isomorphic to the algebraic Atiyah--Hirzebruch spectral sequence of stunted lens spectra, whose data are mostly encoded in our outlined strategy. In particular, we extend Szymik's work to a much larger range.

\subsection{Applications and future work}
Our motivation for computing the $C_3$-equivariant stable stems comes from the interest in the generalized Tate spectral sequence \cite{greenleesmay1995generalized} for groups with mixed prime orders. The result of this paper serves as an input to the study of $\Sigma_3$-equivariant Tate-type spectral sequences, whose patterns are currently under investigation.

Recently, our methods and results have been applied by Kuhn--Quigley--Wang \cite{kuhnquigleywang2026exotic} to discover new families of exotic spheres that admit smooth $C_3$-actions. In particular, since the underlying maps
\begin{align*}
    \pi^{C_3}_{10, 6} &\to \pi_{10}^{cl}=\bbZ/3\{\beta_1\}\\
    \pi^{C_3}_{13, 8} &\to \pi_{13}^{cl}=\bbZ/3\{\alpha_1\beta_1\}\\
    \pi^{C_3}_{20,12} &\to \pi_{20}^{cl}=\bbZ/3\{\beta_1^2\}
\end{align*}
all have non-trivial images by \cref{charofres}, their results imply that there are exotic spheres corresponding to $\beta_1, \alpha_1\beta_1$ and $\beta_1^2$ that admit smooth $C_3$-actions, which makes them equivariantly homeomorphic to $S^{4+2\lambda}, S^{5+4\lambda}$ and $S^{8+6\lambda}$, respectively. 

\subsection{Organizations}
In \cref{1.5}, we review and prove some basic facts about the 3-primary classical stable stems and the stunted lens spaces. In \cref{2}, we review the necessary background for the $C_3$-equivariant homotopy. In \cref{3}, we carry out the computation of the Atiyah--Hirzebruch spectral sequence of the stunted lens spaces. In \cref{4}, we resolve the extension problems in the long exact sequence of the isotropy separation sequence, and also describe the image of the geometric fixed point and the underlying map. Everything up to this section will be 3-completed. In \cref{5}, we deal with the other primary information. In \cref{6}, we include the necessary tables and charts for reference.

\subsection*{Acknolwedgement}
The authors are deeply grateful to Mark Behrens for suggesting this methodology, for sharing his own computations and for many helpful conversations. The authors would like to thank Eva Belmont, Sihao Ma, Yunze Lu, Danny Xiaolin Shi, Yuchen Wu, Yu Zhang and Foling Zou for helpful conversations. Special thanks to Hana Jia Kong and Dan Isaksen for sharing the code for drawing the spectral sequence charts. Finally, our heartfelt thanks go to Zhouli Xu for his encouragement, support, and helpful insights throughout this project.

\section{Some non-equivariant background}\label{1.5}
\subsection{The classical 3-primary stable stems}
As suggested in Section \ref{strategy}, the classical 3-primary stable stems serve as the main input of this paper. The stem-wise computation was first carried out by Serre \cite{Serre1951fibres} up to stem 8. Building on the Serre spectral sequence and further developing his brackets, Toda managed to compute the 3-primary stable stems up to stem 32 \cite{toda1958pprimarycomponents}. Subsequently, Oka extended these computations to stem 76 following a similar approach \cite{oka1971stablegroupI,oka1972stablegroupII}.

The Adams spectral sequence has been extensively utilized in the study of stable stems ever since its first appearance in \cite{adams1958AdamsSS}. In the $p=3$ case, using his spectral sequence, May and computed the Adams $E_2$-page up to stem 88 \cite{may1964thesis}. Later with the assistance from the computer, Tangora computed the Adams $E_2$-page up to total degree of 100 using the lambda algebra in \cite{tangora1985computing}. 

Starting from the $E_2$-page, the first Adams differentials studied were the family of $d_2$-differentials supported on $h_j$ for $j \geq 1$ \cite{liulevicius1962factor, shimada1961hopf}, leading to the resolution of the Hopf invariant one problem. May later determined some $d_3$-differentials, pushing computations further to stem 77. Building upon these results, Nakamura extended the Adams spectral sequence computations up to stem 104 \cite{nakamura1975somedifferentials}.

For our purpose, we summarize the data of the 3-primary Adams spectral sequence of $S^0$ up to stem 35 in \cref{3primaryAdamsSS}. The relevant Adams differentials can be found in \cite{may1964thesis}.

Meanwhile, the Adams–Novikov spectral sequence introduced by Novikov \cite{Novikov1967cobordism} has also been widely exploited. It was first applied by Zahler \cite{Zahler1972ANSS} to compute up to stem 45. Using the algebraic Novikov spectral sequence \cite{Novikov1967cobordism, miller1975thesis}, Miller managed to further compute the Adams–Novikov $E_2$-page up to stem 79. The computations up to stem 108 are summarized in \cite{ravenel1986complexcobordism}. More recently, Wang \cite{wang2021ANSSE2} has introduced a computer-assisted method based on $BP_{\ast}BP$-minimal resolution, hence successfully extending the computations of Adams-Novikov $E_2$-page up to stem 158 \cite{belmontcode}. 

Besides stem-wise computation, infinite families in the stable stems is also of great interest. Cohen \cite{cohen1981infinite} has shown the existence of an infinite family on the Adams 3-line, motivated by the 2-primary result of Mahowald \cite{mahowald1977infinite}. On the other hand, the Adams-Novikov spectral sequence has also been widely used in detecting the Greek-letter families ever since \cite{Miller1977periodic}. The examples that are fully understood are the $v_1$-periodic elements in the image of the $J$ homomorphism \cite{adams1966on}, also known as the $\alpha$-family \cite{ravenel1986complexcobordism} (which we summarize below in \cref{homotopyofj}). Later on, families such as the $v_2$-periodic elements has been studied by \cite{oka1975beta,shimomura1996beta,behrens2004v2selfmap,shimomura2010beta,belmont2023beta}.

\begin{theorem}\label{homotopyofj}
    Let $j$ denote the $3$-primary image-of-$J$-spectrum \cite{mahowald1982imageofj}. Then 
\[\pi_*^{cl}(j)=\begin{cases}
    \bbZ_{p}^\wedge \quad &\text{if}\ i=0\\
    \bbZ/p^{j+1} &\text{if}\ i=4k-1, \text{ord}_3(k)=j\\
    0 & \text{otherwise},
\end{cases}\]
where $\mathrm{ord}_3(k)$ denote the $3$-adic valuation of $k$. Moreover, the Hurewicz homomorphism $e: S^0\to j$ is a split surjection onto $\pi_*^{cl}$.
\end{theorem}
\begin{proof}
    See \cite[Theorem 1.5.19]{ravenel1986complexcobordism}.
\end{proof}

We will specify the generators of the homotopy groups of $j$ in \ref{sometodabrackets}, which will be viewed as the generators of $\pi_*^{cl}$, since $\pi_*^{cl}(j)$ is a split summand of $\pi_*^{cl}$.

\subsubsection{Some backgrounds on the Adams $E_2$-page}
Recall from \cite{milnor1958steenrodalgebra} that the 3 primary Steenrod algebra $\calA$ is the $\bbF_3$-algebra generated by \[\beta: H\bbF_3\to \Sigma H\bbF_3,\ \text{and} \ P^i: H\bbF_3\to \Sigma^{4i} H\bbF_3, \quad i\geq 1,\]
subject to the Adem relations. Let $\calA(1)$ denote the subalgebra of $\calA$ generated by $\beta$ and $P^1$.

The 3-primary dual steenrod algebra is 
\[\calA_*=P[\xi_i| i\geq 1]\otimes E(\tau_j|j\geq 0),\]
where the degree of the generators are
\[|\xi_i|=2\cdot (3^i-1), \quad |\tau_j|=2\cdot 3^j-1.\]
The coproduct on $\calA_*$ is determined by
\[\psi(\xi_n)=\sum_{i=0}^n \xi^{3^i}_{n-i}\otimes \xi_i, \quad \psi(\tau_n)=\tau_n\otimes 1 +\sum_{i=0}^n\xi^{3^i}_{n-i}\otimes \tau_i.\]

The $E_2$-page of the Adams spectral sequence $\Ext_\calA^{s, t}$ can be computed via the cohomology of (reduced) cobar complex $C^*_{\calA_*}$ of the dual Steenrod algebra \cite{adams1960nonexistence}, where
    \[C^{s, t}_{\calA_*}\cong \bar{\calA}_*^{\otimes s}\]
    and for $\psi(a_i)=\sum_u a_{i, u}'\otimes a_{i, u}''$,
    \[d_s(a_1|...|a_{s}):=\sum_{1\leq i\leq s; u} (-1)^{\epsilon(i, u)} a_1|\cdots |a_{i, u}'|a_{i, u}''|\cdots|a_s,\]
    where $\epsilon(i, u):=i+|a_{i,u}'|+|a_1|+|a_2|+\cdots+|a_{i-1}|$. 

In \cite{may1964thesis}, May developed a spectral sequence by filtering the cobar complex, which makes large-range computations doable by hands. Following \cite{ravenel1986complexcobordism}, the 3-primary May spectral sequence has $E_1$-page :
\[E_1^{s, t, u}=E(h_{ij}|i\geq1, j\geq0)\otimes P[b_{ij}| i\geq1, j\geq 0]\otimes P[a_i|i\geq 0],\]
where tri-degrees of the generators are
\[|h_{ij}|=(1, 2(3^i-1)\cdot3^j, 2i-1), |b_{ij}|=(2, 2(3^i-1)\cdot3^{j+1}, 3\cdot(2i-1)), |a_i|=(1, 2\cdot3^i-1, 2i+1).\]
The May spectral sequence has differentials 
\[d_r: E^{s,t,u}_r \rightarrow E^{s+1,t,u-r}_r.\]
and converges to the $E_2$-page of the Adams spectral sequence. In particular, the $d_1$-differentials are given by
\[d_1(h_{ij})=\sum_{0<k<j}h_{kj}h_{i-k, k+j},\quad  d_1(b_{ij})=0, \quad d_1(a_i)=-\sum_{0<k<i}a_kh_{i-k, k}.\]

The following higher May differentials will be used later.
\begin{prop}\label{maydiff}
    Around stem 11, up to a sign, there are higher May differentials \[\begin{split}
            d_5(a_1^3)&= a_0^3h_{11},\\
            d_2(a_1^2h_{10})&=a_0^2b_{10}.
        \end{split}\]
\end{prop}
\begin{proof}
$a_0^2b_{10}$, $a_1^2h_{10}$, $a_1^3$ and $a_0^3h_{11}$ all support an $a_0$-tower and there are no other $a_0$-tower around. Using Adams's vanishing line result \cite{adams1966periodicitytheorem}, these $a_0$-towers have to cancel each other. For degree reasons, the only possibility is \[\begin{split}
            d_5(a_1^3)&= a_0^3h_{11},\\
            d_2(a_1^2h_{10})&=a_0^2b_{10}.
        \end{split}\]\end{proof}

There is also the Lambda algebra $\Lambda^{s, t}$ developed by \cite{bousfieldcurtiskanquillenrectorschlesinger1966modp}, which is a smaller differential graded algebra quasi-isomorphic to the cobar complex. Following \cite{tangora1985computing}, the $3$-primary Lambda algebra $\Lambda^{s, t}$ has generators
\[\lambda_i\in \Lambda^{1, 4n}\ \  (i\geq 1), \quad \mu_i\in \Lambda^{1, 4n+1}\ \ (i\geq 0)\]
subject to relations: for any $k\geq 0$,
\begin{align*}
    \lambda_{i}\lambda_{3i+k}&=\sum_{j\geq 0}A(k ,j)\lambda_{i+k-j}\lambda_{3i+j},\\
    \lambda_i\mu_{3i+k}&=\sum_{j\geq 0}A(k ,j)\lambda_{i+k-j}\mu_{3i+j}+\sum_{j\geq 0}B(k, j)\mu_{i+k-j}\lambda_{3i+j},\\
    \mu_{i}\lambda_{3i+k+1}&=\sum_{j\geq 0}A(k ,j)\mu_{i+k-j}\lambda_{3i+j+1},\\
    \mu_{i}\mu_{3i+k+1}&=\sum_{j\geq 0}A(k ,j)\mu_{i+k-j}\mu_{3i+j+1},
\end{align*}
where \[A(k ,j):=(-1)^{j+1} {2(k-j)-1\choose j}, \quad B(k ,j):=(-1)^{j} {2(k-j)\choose j}.\] 
The differentials on the generators are
\begin{align*}
    d(\lambda_k)&=\sum_{j\geq 1}A(k ,j)\lambda_{k-j}\lambda_j,\\
    d(\mu_k)&=\sum_{j\geq 0}A(k, j)\lambda_{k-j}\mu_j+\sum_{j\geq 1}B(k, j)\mu_{k-j}\lambda_j.
\end{align*}

Using the above methods, we can obtain the $E_2$-page of the Adams spectral sequence of $S^0$ as summarized in \cref{3primaryAdamsSS}. 

The following proposition will be helpful for the analysis of the $\alpha$-family Toda brackets later on. 

\begin{prop}\label{imageofj}\hspace{5pt}
\begin{enumerate}
    \item The 3-primary $j$ fits into the (co)fiber sequence
\[j \rightarrow \ell \xrightarrow{\psi} \Sigma^4\ell \rightarrow \Sigma j\]
where $\ell$ is the Adams summand of the $ku$ and $\psi$ is a lift of the Adams operation $\psi^2-1$. Denote by $C = \text{coker}(\psi^\ast) \simeq \calA \mathbin{/ \!/} \calA(1)$ and $K = \text{ker}(\psi^\ast) \simeq \Sigma^{12}\calA \mathbin{/\!/}\calA(1)$ where $\psi^\ast: H^\ast(\Sigma^4\ell) \rightarrow H^\ast(\ell)$ is the induced map. 
    \item The Adams $E_{2}$-page of $j$ fits into the following long exact sequence
\[ \cdots \text{Ext}_{\calA}^{s-1,t}(C,\bbF_{3}) \overset{\delta_{X}}\rightarrow \text{Ext}_{\calA}^{s-1,t}(\Sigma^{-1}K,\bbF_{3}) \rightarrow E^{s,t}_{2}(j) \rightarrow \text{Ext}_{\calA}^{s,t}(C,\bbF_{3})\cdots.\]
The first term $ \text{Ext}_{\calA}^{\ast,\ast}(C,\bbF_{3})$ is isomorphic to 
\[\text{Ext}_{\calA(1)}^{\ast,\ast}(\bbF_{3},\bbF_{3}) \simeq \bbF_{3}[v_{0},b, w_{1}] \otimes E[a_{i}|i=1,2]/(\sim),\]
where $v_{0},b,a_{1}, a_{2}$ are the images of $a_{0},b_{10},h_{0},a_{1}h_{0}$ under $e$ with the same bidegree. The bidgree of $w_{1}$ is (15,3). The relations $(\sim)$ are generated by $v_{0}a_{i} = 0$ and 
\[a_{i}a_{j} = \left\{\begin{matrix}
 
 (-1)^{i-1} v_{0}b & i+j=3,\\
 0 & \text{else}. \\
\end{matrix}\right.\]
The second term $\text{Ext}_{\calA}^{s-1,t}(\Sigma^{-1}K,\bbF_{3})$ is a free $\text{Ext}_{\calA(1)}^{\ast,\ast}(\bbF_{3},\bbF_{3})$-module with a single basis element $\bar{\bar{b}}$ of bidegree (11,0).

    \item Up to a sign, there are $d_2$-differentials 
\[d_{2}(w_{1}^{n}v_{0}\bar{\bar{b}}) = w_{1}^{n}a_{1}a_{2}\]
as well as higher differentials
\[d_r(w_1^k) = v_0^{r+1}a_3w_1^{k-1}\]
where $r = 1+ \text{ord}_3(k)$ for $k \equiv 0$ (mod 3) in the Adams spectral sequence of $j$. Here $a_3 =  v_0^2\bar{\bar{b}}\in E_2(j)$. The Hurewicz map
\[e: \pi^{cl}_{4k-1} \rightarrow \pi_{4k-1}(j)\]
preserves Adams filtration when $\text{ord}_3(k)=0$. Otherwise it increase Adams filtration by $\text{ord}_3(k)-1$.
\end{enumerate}
\end{prop}
\begin{proof}
    See \cite[Section 5]{bruner2022adams}.
\end{proof}

\subsubsection{Some Toda brackets}\label{sometodabrackets}
In this section, we compute certain Toda brackets which will be useful in the computation of the Atiyah--Hirzebruch spectral sequence later on. Let ``$\doteq$'' denote equality up to a sign. We will specify our choice of the homotopy elements inductively based on relations of the Toda brackets. In particular, we will try to avoid using $\doteq$ in homotopy relations, but leaving inevitable sign ambiguities on the spectral sequence side-- especially those arising as differentials and Massey product relations.

We first recall the following shuffling formulas of Toda brackets.
\begin{theorem}\label{shuffingformula}
Let $\alpha, \beta, \gamma, \delta, \epsilon\in \pi_*^{cl}$. Assuming all Toda brackets are well-defined, then,
\begin{enumerate}
        \item \begin{align*}
            \langle \alpha, \beta, \gamma\rangle\cdot \delta&\subseteq \langle \alpha, \beta, \gamma\delta\rangle;\\
             \langle \alpha, \beta, \gamma\delta\rangle&\subseteq \langle \alpha, \beta\gamma, \delta\rangle;\\
             \langle\alpha\beta, \gamma, \delta\rangle & \subseteq  \langle \alpha, \beta\gamma, \delta\rangle ;\\
             \alpha\langle\beta, \gamma, \delta\rangle &\subseteq (-1)^{|\alpha|}\langle\alpha\beta, \gamma, \delta\rangle;\\
             \alpha\langle\beta, \gamma, \delta\rangle &= (-1)^{|\alpha|+1}\langle\alpha,\beta, \gamma\rangle \delta.
        \end{align*}
        \item \[\langle \alpha, \beta, \gamma\rangle=(-1)^{|\alpha||\beta|+|\beta||\gamma|+|\gamma||\alpha|+1}\langle \gamma, \beta, \alpha\rangle.\]
        \item (The three-fold Jocabian identity)
        \[(-1)^{|\alpha||\gamma|}\langle \alpha, \beta, \gamma\rangle+(-1)^{|\beta||\alpha|}\langle \beta, \gamma, \alpha\rangle+(-1)^{|\gamma||\beta|}\langle \gamma, \alpha, \beta\rangle \ni 0.\]
        \item (The five-fold Jocabian identity)
        \[\langle \langle\alpha, \beta, \gamma\rangle, \delta, \epsilon\rangle+(-1)^{|\alpha|}\langle \alpha, \langle \beta, \gamma, \delta\rangle, \epsilon \rangle+(-1)^{|\alpha|+\beta|}\langle \alpha,\beta, \langle \gamma, \delta, \epsilon\rangle \rangle \ni 0.\]
        \item \[\alpha\langle  \beta, \gamma, \delta, \epsilon\rangle=\langle\alpha, \beta, \gamma, \delta\rangle\epsilon.\]
        \item \[\begin{split}
            \alpha\langle  \beta, \gamma, \delta, \epsilon\rangle&\subseteq -\langle\langle\alpha, \beta, \gamma\rangle, \delta, \epsilon\rangle;\\
            \langle\alpha, \beta, \gamma, \delta\rangle\epsilon &\subseteq -\langle\alpha, \beta, \langle\gamma, \delta, \epsilon\rangle\rangle.
        \end{split}\]
    \end{enumerate}
\end{theorem}
\begin{proof}
    See \cite[3.5--3.9]{toda1962composition}, \cite[Theorem 2.3.4]{kochman1990stablehomotopy} and \cite[Theorem 2.35]{isaksenwangxu2023stable}.
\end{proof}

To start with, let $\alpha_1$ and $\alpha_2$ be the homotopy elements detected by $h_0$ and $-a_1h_0$ respectively in the Adams spectral sequence.

\begin{prop}\label{alpha1toda}
    $\langle \alpha_1, \alpha_1, 3 \rangle=\langle3, \alpha_1, \alpha_1\rangle=-\langle\alpha_1, 3,\alpha_1\rangle =\alpha_2$.
\end{prop}
\begin{proof}
    In the cobar complex, $h_{0}$ is represented by $[\xi_1]$, $a_0$ is represented by $[\tau_0]$. Since $d_1(\xi_1^2)=\xi_1|\xi_1$, $d_1(\tau_1)=-\xi_1|\tau_0$, $d_1(\tau_0\xi_1)=\tau_0|\xi_1-\xi_1|\tau_0$, we have
    \begin{align*}
            \langle h_{0}, h_{0}, a_0 \rangle&=[-\xi_1|\tau_1+\xi_1^2|\tau_0]\qquad (*)\\   
            \langle h_{0}, a_0, h_{0}\rangle&=[\xi_1|\tau_0\xi_1-\xi_1|\tau_1-\tau_1|\xi_1]\qquad (**)\\
            \langle a_0, h_{0}, h_{0} \rangle&=[-\tau_0|\xi_1^2-\tau_1|\xi_1+\tau_0\xi_1|\xi_1]\qquad (***).
    \end{align*}
    As $(*)+(**)=-d(\xi_1\tau_1)$, 
    $(*)-(***)=d(\xi_1\tau_1+\xi_1^2\tau_0)$, we have \[\langle h_{0}, h_{0}, a_0 \rangle=\langle a_0, h_{0}, h_{0}\rangle=-\langle h_{0}, a_0, h_{0}\rangle\]
    in $\Ext_\calA$. Since $[-\xi_1|\tau_1+\xi_1^2|\tau_0]$ corresponds to $-a_1h_{10}$ in the May spectral sequence, which detects $-a_1h_0$ in $\Ext_\calA$, the result follows from the Moss convergence theorem \cite{moss1970secondary}. 
\end{proof}

\begin{rmk}\label{signissue}
There are multiple ways to choose sign conventions when defining (matric) Massey products. The sign we choose here follows from that of \cite{massey1958some}\cite{moss1970secondary}. In particular, the Jacobian identity in \cref{shuffingformula} also holds for the Massey products in \cref{alpha1toda}. Note that this sign convention differs from that of \cite{may1969matric}\cite{ravenel1986complexcobordism}. 
\end{rmk}

\begin{prop}\label{alpha2toda}
    The Toda bracket $\langle \alpha_2, \alpha_1, 3 \rangle$ contains three nonzero elements. We choose an element in it and denote it by $\bar{\alpha}_3$. The Toda bracket $\langle \alpha_2, 3, \alpha_1 \rangle$ contains a single nonzero element which we denote by $\alpha_3$.
\end{prop}
\begin{proof}
For the first bracket, the Adams differential $d_2(h_{1})=a_0b_{10}\doteq  a_1h_{0} \cdot h_0$ implies there is a Massey product
\[\langle a_1h_{0}, h_{0}, a_0\rangle \doteq  a_0h_{1}\neq 0\]
in the Adams $E_3$-page. By the Moss convergence theorem, $\langle \alpha_2, \alpha_1, 3 \rangle$ detects a non-trivial element in $\pi_{11}^{cl}$, which we call it $\bar{\alpha}_3$. Since $a_0^2h_1$ survives to the $E_\infty$-page, $3\cdot \bar{\alpha}_3\neq 0$. Moreover, $\pi^{cl}_{11} \cong \bbZ/9$, so the indeterminacy of the bracket consists of elements of order 3 in $\pi^{cl}_{11}$.

For the second bracket, using the Lambda algebra, the corresponding Massey product can be represented by
        \begin{align*}
            \langle h_{0},\  &a_0,\ a_1h_{0}\rangle=[\mu_1\mu_1\lambda_1].\\
            \lambda_1,\ &\mu_0,\ \mu_1\lambda_1\\
            \mu_1 &\qquad 0
        \end{align*}
From the curtis table \cite{tangora1985computing}, we know that $[\mu^{2}_{1}\lambda_1]$ corresponds to $\pm a_0^2h_{1}$. By the Moss convergence theorem, $\langle \alpha_2, 3, \alpha_1 \rangle$ contains a nonzero element which is of order $3$. In particular, $3\cdot \bar{\alpha}_3\doteq \alpha_3$.
\end{proof}

\begin{prop}\label{alpha3toda}
$\langle \alpha_3, 3, \alpha_1 \rangle= -\langle \bar{\alpha}_3, \alpha_1, 3  \rangle \neq 0 \in \pi^{cl}_{15}$. Denote this unique element by $\alpha_4$.
\end{prop}
\begin{proof}
The May differential $d_5(a_1^3)\doteq a_0^3h_{1}$ in \cref{maydiff} plus the May convergence theorem gives the Massey product $\langle a_0^2h_{1}, a_0, h_{0}\rangle \doteq a_1^3h_{0}$ in $\Ext_\calA$. The Moss convergence theorem then implies that $\langle \alpha_3, 3, \alpha_1 \rangle$ contains a nonzero element in $\pi_{15}^{cl}$.

For the second bracket, the 5-fold Jacobi identity \cite{toda1962composition} renders
        \[0\in \langle \langle \alpha_2, \alpha_1, 3\rangle, \alpha_1, 3\rangle+ (-1)^7\langle \alpha_2, \langle \alpha_1, 3, \alpha_1\rangle, 3\rangle+ (-1)^{7+3}\langle \alpha_2,  \alpha_1,\langle 3, \alpha_1, 3\rangle\rangle,\]
        For degree reasons, there are no indeterminacies for any of the brackets above. Since $\langle 3, \alpha_1, 3\rangle =0$ for degree reasons, we have
        \[\langle \bar{\alpha}_3, \alpha_1 ,3\rangle=\langle \langle \alpha_2, \alpha_1, 3\rangle, \alpha_1, 3\rangle= -\langle \alpha_2, \alpha_2, 3\rangle=-\langle \langle \alpha_1, \alpha_1, 3\rangle, \alpha_2, 3\rangle.\]
        Applying the 5-fold Jacobi once more, 
        \[0\in \langle \langle \alpha_1, \alpha_1, 3\rangle, \alpha_2, 3\rangle+(-1)^3\langle  \alpha_1, \langle\alpha_1, 3, \alpha_2\rangle, 3\rangle+(-1)^{3+3}\langle  \alpha_1, \alpha_1, \langle3, \alpha_2, 3\rangle\rangle.\]
        For degree reasons, there are no indeterminacies as well. Since $\langle 3, \alpha_2, 3\rangle =0$ for degree reasons and $\langle\alpha_1, 3, \alpha_2\rangle=(-1)^{7\cdot 3+7\cdot 0+3\cdot 0+1}\langle\alpha_2, 3, \alpha_1\rangle=\langle\alpha_2, 3, \alpha_1\rangle=\alpha_3$,
        \[\langle \langle \alpha_1, \alpha_1, 3\rangle, \alpha_2, 3\rangle=\langle \alpha_1, \alpha_3, 3\rangle.\]
        By the 3-fold Jacobi identity,
        \[0\in(-1)^{11\cdot 3}\langle \alpha_3, 3, \alpha_1 \rangle+(-1)^{3\cdot 0}\langle \alpha_1, \alpha_3, 3\rangle+(-1)^{0\cdot 11}\langle 3, \alpha_1, \alpha_3\rangle.\]
        All bracket above has no indeterminacies, so $\langle 3, \alpha_1, \alpha_3\rangle\doteq\langle 3, \alpha_1, \bar{\alpha}_3\rangle\cdot 3$, which must be 0 since $\pi^{cl}_{15}\cong \bbZ/3$. Therefore,\[\langle \alpha_3, 3, \alpha_1 \rangle=\langle \alpha_1,\alpha_3, 3 \rangle=\langle \alpha_2, \alpha_2, 3\rangle=-\langle \bar{\alpha}_3, \alpha_1, 3\rangle \neq 0 \in \pi^{cl}_{15}.\]  
        For degree reasons, the brackets have no indeterminacy, so they contain a unique element in $\pi_{15}^{cl}$.
\end{proof}

In particular, since $\alpha_1, \alpha_2, \bar{\alpha}_3$ and $\alpha_4$ are in the Hurewicz image of $j$, the relations of Toda brackets in \cref{alpha1toda}-\cref{alpha3toda} all hold in $\pi_*^{cl}(j)$. 

We will next prove some general relations of Toda bracket concerning $\alpha$-family elements. In particular, we will state the results in $\pi_*(j)$, and due the split surjection in \cref{homotopyofj}, the results also holds true in $\pi_*^{cl}$ except that the Toda brackets have potentially larger sets of indeterminacy.

\begin{prop}\label{alpha3k+1toda}
    In $\pi_*(j)$, assume $\alpha_{3k+1} \in \pi_{12k+3}(j)$ is defined, then
    \[\langle \alpha_{3k+1}, \alpha_1, 3 \rangle=-\langle\alpha_{3k+1}, 3,\alpha_1\rangle\]
    and they contain a single nonzero element which we will denote by $\alpha_{3k+2}$.
\end{prop}
\begin{proof}
    Since $h_0$ and $a_1h_0$ maps to $a_1$ and $a_2$ by $e$,
    we have
    \[a_{2} \doteq e_{\ast}\langle h_0,h_0,a_0\rangle = \langle a_{1}, a_{1}, v_0 \rangle.\]
    by \cref{alpha1toda}. Shuffling in $\text{Ext}_{\calA}^{\ast,\ast} (C,\bbF_{3})$ in \cref{imageofj}, we have 
    \[w_{1}^{k}a_{2} \doteq  w_{1}^{k}\langle a_{1}, a_{1},v_0\rangle = \langle w_{1}^{k} a_{1}, a_{1},v_0\rangle.\]
Since all of $w_1^ka_1$ ($k \in \bbN$) and $v_0$ sit in ker($\delta_X$), the first bracket pulls back to $E_2(j)$. As $w_1^ka_i (i=1,2)$ are surviving cycles in $E_2(j)$, by the Moss convergence theorem, $\langle \alpha_{3k+1}, \alpha_1, 3 \rangle$ must contain a nonzero element detected by $\pm w_1^ka_2$. The other bracket and the relation follow from \cref{alpha1toda} and similar arguments. Since $\pi_{12k+3}(j)=\bbZ/3$, there is no indeterminacy.
\end{proof}

\begin{prop}\label{alpha3k+2toda}
In $\pi_{\ast}(j)$, assume $\alpha_{3k+2} \in \pi_{12k+7}(j)$ is defined, then the Toda bracket $\langle \alpha_{3k+2}, \alpha_1, 3 \rangle$ does not contain zero. We choose one element and denote it by $\bar{\alpha}_{3k+3}$. The Toda bracket $\langle \alpha_{3k+2}, 3, \alpha_1 \rangle$ contains a single nonzero element, which we denote by $\alpha_{3k+3}$.
\end{prop}
\begin{proof}
For the first bracket, the $d_{2}$-Adams differential 
\[d_{2}(w_{1}^{n}v_{0}\bar{\bar{b}}) \doteq w_{1}^{n}a_{1}a_{2}\]
implies the Massey product $\langle w_1^na_2, a_1, v_0\rangle \doteq w_1^{n}a_3$ in $E_3(j)$. Applying the Moss convergence theorem and \cref{alpha3k+1toda}, $\langle \alpha_{3k+2}, \alpha_1, 3 \rangle$ contains an element detected by $\pm w_1^na_3$. In particular, this element is a generator of $\pi_{12k+11}^{cl}(j)$, and the indeterminacy are the 3-multiples in $\pi_{12k+11}^{cl}(j)$.

For the second bracket, due to filtration reasons, $\langle \alpha_{3k+2}, 3, \alpha_{1}\rangle$ contains either 0 or an element of the highest Adams filtration in $\pi^{cl}_{12k+11}$. Hence all the brackets in the following 5-fold Jacobian identity
\[\langle \langle \alpha_{3k+2}, 3, \alpha_{1}\rangle ,3,\alpha_1\rangle - \langle  \alpha_{3k+2},\langle 3 ,\alpha_1,3 \rangle ,\alpha_1\rangle - \langle  \alpha_{3k+2},3,\langle \alpha_1,3,\alpha_1\rangle \rangle \ni 0\]
is well defined in $\pi_\ast(j)$. For degree reasons, all the brackets above have no indeterminacy. The second bracket is zero since $\langle 3 ,\alpha_1,3 \rangle\in \pi_4(j) = 0$. The third bracket is detected by 
\[\langle w_{1}^{n}a_{2}, v_0, -a_{2}\rangle \supseteq -w_{1}^{n}\langle a_{2}, v_0, a_{2}\rangle \doteq w_{1}^{n+1}a_{1} \in \text{Ext}_{\calA}^{\ast,\ast}(C,\bbF_{3}).\]
where the last equality follows from
\[w_1a_1 = e_\ast(a^3_1h_0) \doteq e_\ast\langle a_1h_0,a_0,a_1 h_0\rangle = \langle a_2, v_
0, a_2\rangle.\]
As a result, $\langle \alpha_{3k+2}, 3, \alpha_{1}\rangle$ must be nonzero with no indeterminacy for degree reasons. We denote this element by $\alpha_{3k+3}$. In particular, $3\cdot \alpha_{3k+3}=0$ and 
\[3^{\text{ord}_3(3k+3)}\cdot \bar{\alpha}_{3k+3}\doteq \alpha_{3k+3}.\]
\end{proof}

\begin{prop}\label{alpha3k+3toda}
    In $\pi_{\ast}(j)$, assume $\bar{\alpha}_{3k+3}$ and $\alpha_{3k+3}$ are defined, then we have \[\langle \bar{\alpha}_{3k+3}, \alpha_1, 3 \rangle = -\langle \alpha_{3k+3}, 3, \alpha_1 \rangle\] 
    and they contain a single nonzero element, which we denote by $\alpha_{3k+4}$.
\end{prop}
\begin{proof}
    For the left-hand-side bracket, applying the 5-fold Jacobian identity renders
\[\langle \langle \alpha_{3k+2}, \alpha_1, 3\rangle ,\alpha_1,3\rangle - \langle  \alpha_{3k+2},\langle \alpha_1 ,3,\alpha_1 \rangle ,3 \rangle - \langle  \alpha_{3k+2},\alpha_1,\langle 3,\alpha_1,3\rangle \rangle \ni 0\]
in $\pi_\ast(j)$. For degree reasons, there are no indeterminacy. The last bracket is 0 since $\langle 3,\alpha_1,3\rangle=0$. Thus, for similar reasons as in \cref{alpha3k+2toda}, the second bracket is detected by 
\[\langle w_{1}^{n}a_{2}, -a_{2}, v_0\rangle \supseteq -w_{1}^{n}\langle a_{2}, a_{2}, v_
0\rangle \doteq -w_{1}^{n+1}a_{1}.\] 
By the Moss convergence theorem, it detects a nonzero element, and so does the first bracket. For degree reasons, there is no indeterminacy.

To prove the equality, we apply the 5-fold Jacobian identity to the right-hand-side bracket
\[\langle \langle \alpha_{3k+2}, 3, \alpha_1\rangle, 3 ,\alpha_1\rangle - \langle  \alpha_{3k+2},\langle 3, \alpha_1 ,3\rangle ,\alpha_1 \rangle - \langle  \alpha_{3k+2},3, \langle \alpha_1, 3,\alpha_1\rangle \rangle \ni 0.\]
Now, $\langle 3, \alpha_1, 3\rangle=0$, the middle term is 0.
Since $\langle \alpha_2, \alpha_2, 3\rangle=(-1)^{0+0+7\cdot 7+1}\langle 3, \alpha_2, \alpha_2\rangle=\langle 3, \alpha_2, \alpha_2\rangle$ and the 3-fold Jacobian identity
\[0\in \langle 3, \alpha_2, \alpha_2\rangle -\langle  \alpha_2, 3,\alpha_2\rangle+\langle \alpha_2, \alpha_2, 3\rangle\]
together imply 
\[\langle \alpha_2, \alpha_2,3\rangle =-\langle  \alpha_2, 3,\alpha_2\rangle.\]
In particular, we also have the corresponding Massey products
\[\langle a_2, a_2,v_0\rangle =-\langle  a_2, v_0,a_2\rangle.\]
in $E_2(j)$ and $\text{Ext}_{\calA(1)}$. Multiplying by $w_1^k$ and pulling back to $E_2(j)$, by the Moss convergence theorem we acquire
\[\langle \alpha_{3k+2}, \alpha_2, 3\rangle =-\langle  \alpha_{3k+2}, 3,\alpha_2\rangle.\]
As a result,
\[\begin{split}
    \langle \alpha_{3k+3}, 3, \alpha_1 \rangle&=\langle \langle \alpha_{3k+2}, 3, \alpha_1 \rangle, 3, \alpha_1 \rangle\\
    &=\langle  \alpha_{3k+2},3, \langle \alpha_1, 3,\alpha_1\rangle \rangle\\
    &=\langle  \alpha_{3k+2},3, -\alpha_2 \rangle\\
    &=-\langle  \alpha_{3k+2}, -\alpha_2 , 3\rangle\\
    &=-\langle  \alpha_{3k+2},\langle \alpha_1 ,3,\alpha_1 \rangle ,3 \rangle\\
    &=-\langle \bar{\alpha}_{3k+3}, \alpha_1, 3 \rangle.
\end{split}\]
\end{proof}

In summary, the $\alpha_{n}$ always denotes an element of order 3 in $\pi_{4n-1}(j)$. When $n\equiv 1\ \text{or}\ 2 \pmod3$, it is also the generator of $\pi_{4n-1}(j)$. When $n\equiv 0\pmod 3$, $\bar{\alpha}_n$ is a generator of $\pi_{4n-1}(j)$. In particular, $\langle \_, \alpha_1, 3\rangle$ is an operator that takes a generator of $\pi_*(j)$ to the generator of $\pi_{*+4}(j)$. This set of notation agrees with that of \cite[Theorem 4.4.20]{ravenel1986complexcobordism} up to a sign.

Next, we compute some Toda brackets outside the $\alpha$-family.

\begin{prop}\label{todabetafami}
    \hspace{5pt}
    \begin{enumerate}
        \item $\langle \alpha_1, \alpha_1, \alpha_1 \rangle$ contains a single nonzero element, which we denote by $\beta_1$;
        \item $\langle \beta_1\alpha_1, \alpha_1, \alpha_1 \rangle=\beta_1^2$;
        \item $\langle \beta_2\alpha_1, \alpha_1, \alpha_1 \rangle=\beta_2\beta_1$, where $\beta_2$ is defined in \cref{beta2todabracket}.
        \item $\langle \beta_1^3, \alpha_1, \alpha_1 \rangle$ contains a single nonzero element, which we denote by $x_{37}$.
    \end{enumerate}
\end{prop}
\begin{proof} 
\hspace{5pt}
\begin{enumerate}
    \item Using the cobar complex, the corresponding Massey product can be computed as
\begin{align*}
            \langle h_{0}, & h_{0}, h_{0}\rangle =[\xi_1|\xi_1^2+\xi_1^2|\xi_1]=b_{10}\neq 0.\\
            \xi_1,\ &\xi_1,\ \xi_1\\
            \xi_1^2 &\ \ \ \xi_1^2
        \end{align*}
    Thus, part (1) follows from the Moss convergence theorem.
    \item Shuffling gives 
    \[\begin{split}
        \langle \beta_1\alpha_1, \alpha_1, \alpha_1 \rangle &=(-1)^{10}\beta_1 \langle \alpha_1, \alpha_1, \alpha_1 \rangle\\
        &=\beta_1^2.
    \end{split}\]
    \item Shuffling gives 
    \[\begin{split}
        \langle \beta_2\alpha_1, \alpha_1, \alpha_1 \rangle &=(-1)^{26}\beta_2 \langle \alpha_1, \alpha_1, \alpha_1 \rangle\\
        &=\beta_2\beta_1.
    \end{split}\]
    \item Since there is an Adams Novikov differential $d(\beta_{3/3})=\beta_1^3\alpha_1$ \cite{Zahler1972ANSS}, the result follows from the corresponding Massey product and the Moss convergence theorem.
\end{enumerate}
\end{proof}

\begin{cor}\label{x37hidden}
    There is a hidden $\alpha_1$-extension from $x_{37}$ to $-\beta_1^4$.
\end{cor}
\begin{proof}
    The last bracket in  \cref{todabetafami} plus shuffling gives
    \[x_{37}\cdot \alpha_1=\langle \beta_1^3, \alpha_1, \alpha_1 \rangle\cdot \alpha_1=(-1)^{30+1}\beta_1^3\langle \alpha_1, \alpha_1, \alpha_1\rangle=-\beta_1^4\]
\end{proof}

\begin{prop}\label{beta2todabracket}
    The following relations of Toda brackets hold in $\pi_*^i$.
    \begin{enumerate}
        \item $\langle \alpha_2, \alpha_1, \alpha_2, \beta_1\rangle = \langle \beta_1, 3, \alpha_2, \beta_1\rangle=\beta_2\alpha_1$;
        \item $\langle \alpha_2, \alpha_1, \alpha_2, \alpha_1, \alpha_1\rangle \doteq \beta_2$;
        \item $\langle \alpha_2, \alpha_1, \alpha_2, \alpha_1, \alpha_2\rangle \doteq \beta_1^3$.
    \end{enumerate}
\end{prop}
\begin{proof}\hspace{5pt}
    \begin{enumerate}
        \item $\langle \alpha_2, \alpha_1, \alpha_2\rangle=0=\langle  \alpha_1, \alpha_2, \beta_1\rangle$ for degree reasons, so the four-fold bracket is strictly defined.
        As in the proof of \cref{todabetafami}, 
        \[a_0b_{10} = a_0\langle h_0, h_0, h_0\rangle=(-1)^{0+1}\langle a_0, h_0, h_0\rangle h_0=-a_1h_0\cdot h_0.\]
        The Adams differentials 
        \[d_2(h_1)=a_0b_{10}=-a_1h_0\cdot h_0, \quad d_2(h_0h_{20})=b_{10}\cdot a_1h_{0}\]
        imply the following Massey products
        \[\langle -a_1h_0, h_0, a_1h_0, b_{10}\rangle=\langle b_{10}, a_0, a_1h_0, b_{10}\rangle = h_0h_1h_{20}.\]
        Note that $h_1h_{20}$ detects a non-trivial element in $\pi_{26}^{cl}$, which we denote by $\beta_2$.
        By the Moss convergence theorem, we have
        \[\langle \alpha_2, \alpha_1, \alpha_2, \beta_1\rangle = \langle \beta_1, 3, \alpha_2, \beta_1\rangle=\beta_2\alpha_1.\]
        
        \item For degree reasons, the five-fold bracket in part (2) is strictly defined and has no indeterminacy. Since $\beta_1=\langle \alpha_1, \alpha_1, \alpha_1\rangle$ from  \cref{todabetafami}, shuffle the bracket in part (1) we get
        \[\beta_2\alpha_1=\langle \alpha_2, \alpha_1, \alpha_2, \langle\alpha_1, \alpha_1, \alpha_1\rangle\rangle\doteq \langle \alpha_2, \alpha_1, \alpha_2, \alpha_1, \alpha_1\rangle \cdot \alpha_1\]
        so that $\beta_2\doteq\langle \alpha_2, \alpha_1, \alpha_2, \alpha_1, \alpha_1\rangle$.
        \item For degree reasons, the five-fold bracket in part (3) is strictly defined and has no indeterminacy. From \cite[A3.4]{ravenel1986complexcobordism}, there is a Toda bracket $\beta_1^3\doteq \langle \beta_2, 3, \alpha_1\rangle$. Using part (2), we have
        \[\begin{split}
            \beta_1^3&\doteq \langle \langle \alpha_2, \alpha_1, \alpha_2, \alpha_1, \alpha_1\rangle, 3, \alpha_1\rangle\\
            &\doteq \langle \alpha_2, \alpha_1, \alpha_2, \alpha_1, \langle\alpha_1, 3, \alpha_1\rangle\rangle\\
            &\doteq \langle \alpha_2, \alpha_1, \alpha_2, \alpha_1,\alpha_2\rangle.
        \end{split}\]
    \end{enumerate}
\end{proof}

\subsection{The 3-primary stunted lens spaces}
Denote $BC_3$ as the classifying space of $C_3$. Following \cite{brunermaymccluresteinberger1986hinfinity}, Let \[(BC_3)^\infty_{2j}:= \mathrm{Thom}(BC_3, j\lambda), \quad j\in \bbZ\] denote the Thom spectrum of the (virtual) bundle $j\lambda$ over $BC_3$. It is a CW-spectrum and has one stable cell in each dimension greater than or equal to $2j$. In particular, its bottom cell is in dimension $2j$ and let $(BC_3)^\infty_{2j+1}$ denote the cofiber of the inclusion map
\[S^{2j}\hookrightarrow (BC_3)^\infty_{2j}.\]
For example, $\Sigma^\infty BC_3=(BC_3)_1^\infty$. Let $(BC_3)^m_n$ denote the $m$-skeleton of $(BC_3)_n^\infty$.

\begin{prop}\label{cohofBC3} The cohomology of the classifying space $BC_3$ is
    \[H^*(BC_3; \bbF_3)=\bbF_3[y]\otimes E(x), |x|=1, |y|=2.\]
Some $\calA$-module structures are
\[\beta(xy^k)=y^{k+1}\]
\[P^1(y^k)=\begin{cases}y^{k+2}\quad & k=3i+1\\ -y^{k+2}\quad & k=3i+2\\ 0\quad & k=3i
\end{cases}\]
\[P^1(xy^k)=\begin{cases}xy^{k+2}\quad & k=3i+1\\ -xy^{k+2}\quad & k=3i+2\\ 0\quad & k=3i
\end{cases}\]
\[P^3(y^k)=\begin{cases}y^{k+6}\quad & k\equiv 3, 4, 5 \pmod 9\\ -y^{k+6}\quad & k\equiv 6,7,8 \pmod 9\\ 0\quad & k\equiv 0,1,2 \pmod 9
\end{cases}\]
\[P^3(xy^k)=\begin{cases}xy^{k+6}\quad & k\equiv 3, 4, 5 \pmod 9\\ -xy^{k+6}\quad & k\equiv 6,7,8 \pmod 9\\ 0\quad & k\equiv 0,1,2 \pmod 9
\end{cases}\]
\end{prop}
\begin{proof}
    The cohomology ring structure may be found in \cite{hatcher2002algebraictopology}. For degree reason, $\beta(x)=y$, $P^1(y)=y^3$. The rest $\calA$-action follows from the Cartan formula.
\end{proof}

\begin{figure}
    \centering
    \begin{tikzpicture}[scale=0.4]

\draw[thick](24,-1) circle (0.5);
\node at (24, -1) {20};
\draw[thick](22, -1) circle (0.5);
\node at (22, -1) {19};
\draw[thick](16, -1) circle (0.5);
\node at (16, -1) {16};
\draw[thick](14, -1) circle (0.5);
\node at (14, -1) {15};
\draw[thick](8, -1) circle (0.5);
\node at (8, -1) {12};
\draw[thick](6, -1) circle (0.5);
\node at (6, -1) {11};
\draw[thick](0, -1) circle (0.5);
\node at (0, -1) {8};
\draw[thick](-2, -1) circle (0.5);
\node at (-2, -1) {7};
\draw[thick](-8, -1) circle (0.5);
\node at (-8, -1) {4};
\draw[thick](-10, -1) circle (0.5);
\node at (-10, -1) {3};

\draw[thick] (22.5, -1)--(23.5, -1);
\node at (23, -0.8){$\beta$};
\draw[thick] (14.5, -1)--(15.5, -1);
\node at (15, -0.8){$\beta$};
\draw[thick] (6.5, -1)--(7.5, -1);
\node at (7, -0.8){$\beta$};
\draw[thick] (-0.5, -1)--(-1.5, -1);
\node at (-1, -0.8){$\beta$};
\draw[thick] (-8.5, -1)--(-9.5, -1);
\node at (-9, -0.8){$\beta$};

\draw[thick,dashed] (24, -1.5) .. controls (22,-3) and (18, -3) .. (16, -1.5);
\node at (20, -3){$-P^{1}$};

\draw[thick] (22, -0.5) .. controls (20, 1) and (16,1) .. (14, -0.5);
\node at (18, 1){$P^{1}$};

\draw[thick,dashed] (14, -0.5) .. controls (12,1) and (8, 1) .. (6, -0.5);
\node at (10, 1){$-P^{1}$};

\draw[thick] (8, -1.5) .. controls (6, -3) and (2, -3) .. (0, -1.5);
\node at (4, -3){$P^{1}$};

\draw[thick,dashed] (0, -1.5) .. controls (-2, -3) and (-6, -3) .. (-8, -1.5);
\node at (-4, -3){$-P^{1}$};

\draw[thick] (-2, -0.5) .. controls (-4, 1) and (-8, 1) .. (-10, -0.5);
\node at (-6, 1){$P^{1}$};


\draw[thick](20, -7) circle (0.5);
\node at (20, -7) {18};
\draw[thick](18, -7) circle (0.5);
\node at (18, -7) {17};
\draw[thick](12, -7) circle (0.5);
\node at (12, -7) {14};
\draw[thick](10, -7) circle (0.5);
\node at (10, -7) {13};
\draw[thick](4, -7) circle (0.5);
\node at (4, -7) {10};
\draw[thick](2, -7) circle (0.5);
\node at (2, -7) {9};
\draw[thick](-4, -7) circle (0.5);
\node at (-4, -7) {6};
\draw[thick](-6, -7) circle (0.5);
\node at (-6 , -7) {5};
\draw[thick](-12, -7) circle (0.5);
\node at (-12, -7) {2};
\draw[thick](-14, -7) circle (0.5);
\node at (-14, -7) {1};

\draw[thick] (18.5, -7)--(19.5, -7);
\node at (19, -6.8){$\beta$};
\draw[thick] (10.5, -7)--(11.5, -7);
\node at (11, -6.8){$\beta$};
\draw[thick] (2.5, -7)--(3.5, -7);
\node at (3, -6.8){$\beta$};
\draw[thick] (-5.5, -7)--(-4.5, -7);
\node at (-5, -6.8){$\beta$};
\draw[thick] (-13.5, -7)--(-12.5, -7);
\node at (-13, -6.8){$\beta$};

\draw[thick] (20, -7.5) .. controls (18, -9) and (14, -9) .. (12, -7.5);
\node at (16, -9){$P^{1}$};

\draw[thick,dashed] (12, -7.5) .. controls (10, -9) and (6, -9) .. (4, -7.5);
\node at (8, -9){$-P^{1}$};

\draw[thick] (10, -6.5) .. controls (8, -5) and (4, -5) .. (2, -6.5);
\node at (6, -5){$P^{1}$};

\draw[thick,dashed] (2, -6.5) .. controls (0, -5) and (-4, -5) .. (-6, -6.5);
\node at (-2, -5){$-P^{1}$};

\draw[thick] (-4, -7.5) .. controls (-6, -9) and (-10, -9) .. (-12, -7.5);
\node at (-8, -9){$P^{1}$};

\end{tikzpicture}
    \caption{The $\calA(1)$-action on $H^*(BC_{3})^{20}_{1}$}
    \label{H(BC3)^{20}}
\end{figure}
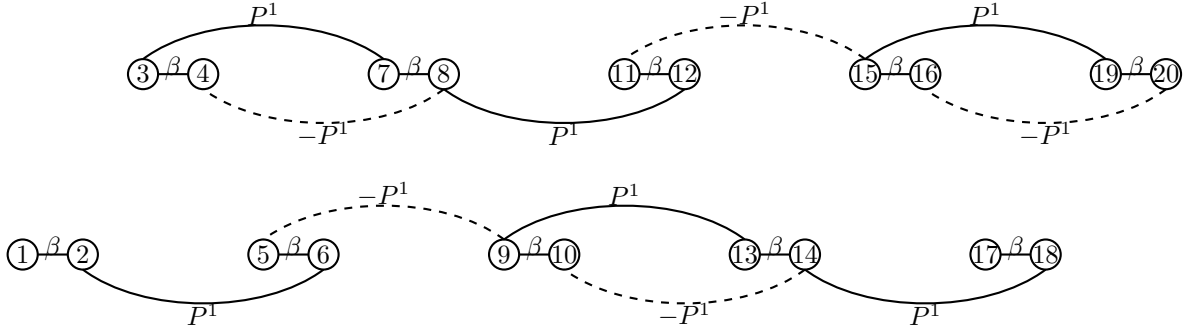

The cohomology
The $\calA(1)$-module structure of $BC_3$ can be pictured as in \cref{H(BC3)^{20}}, where $-P^1$ denotes that the generator is sent to the negative of the generator at the corresponding degrees. In particular, \cref{H(BC3)^{20}} hints that there is no $\calA(1)$-action between cells of dimension congruent to $0, 3 \pmod 4$ and cells of dimension congruent to $1, 2 \pmod 4$. This $\calA(1)$-module splitting can actually be upgraded to a splitting in homotopy:
\begin{theorem}[\cite{brunermaymccluresteinberger1986hinfinity}]\label{BC3splitting}
    For $0\leq k\leq \infty$, there is a natural map
    \[(BC_3)^{n+k}_n\to (B\Sigma_3)_n^{n+k}\]
    induced by the inclusion $C_3 \hookrightarrow \Sigma_3$ that is projection onto the wedge summand.
\end{theorem}

The $(B\Sigma_3)^{n+k}_n$, as its notation suggests, is some Thom spectrum over the classifying space $B\Sigma_3$. It has a stable cell in each dimension congruent to $0$ or $-1 \pmod 4$ between $n$ and $n+k$. We will denote the other wedge summand by $X$ (which is also some Thom complex over $B\Sigma_3$), and for convenience, we will denote the splitting as
\[(BC_3)^{n+k}_n \simeq (B\Sigma_3)_n^{n+k} \vee X_n^{n+k}.\]
The computation of $\pi_{i-j}(BC_3)^{\infty}_{-j}$ thus splits into computations of these two parts as well.

The following theorem in \cite[V.2.6, V.2.9]{brunermaymccluresteinberger1986hinfinity} will be helpful in analyzing the cell structure of the lens spaces.
\begin{theorem}\label{coreducible} Let $\psi(k):=\floor{\frac{k}{4}}$.
\begin{enumerate}
    \item $D(BC_3)_n^m \simeq \Sigma(BC_3)_{-m-1}^{-n-1}$ for $-\infty \leq n\leq m \leq \infty$.
    \item  The bottom cell of $(BC_3)_{2j}^{2j+k}$ splits off if and only if $j \equiv 0 \pmod {3^{\psi(k)}}$
    \item For $\epsilon=0 \ \mathrm{or}\ 1$, the top cell of $(BC_3)_{2j+\epsilon}^{2j+k}$ splits off if and only if $k=\epsilon$ or $k$ is odd and $2j+k+1\equiv 0 \pmod {3^{\psi(k)}}$.
    \item  The bottom cell of $(B\Sigma_3)_{4j}^{4j+k}$ splits off if and only if $j \equiv 0 \pmod {3^{\psi(k)}}$.
    \item For $\epsilon=0 \ \mathrm{or}\ 1$, the top cell of $(B\Sigma_3)_{4j-\epsilon}^{4j+k}$ splits off if and only if $k=\epsilon=0$ or $k =4i-1$ and $i+j\equiv 0 \pmod {3^{i+\epsilon-1}}$.
\end{enumerate}
\end{theorem}

We also record the James periodicity \cite[V.2.6, V.2.9]{brunermaymccluresteinberger1986hinfinity} here.
\begin{theorem}\label{jamesperiodicity} For $\epsilon=0$ or $1$, 
\begin{enumerate}
    \item If $i\equiv j\pmod{3^{\psi(k)}}$,
    \[(BC_3)_{2i+\epsilon}^{2i+k} \simeq \Sigma^{2(i-j)} (BC_3)_{2j+\epsilon}^{2j+k}\]
    \item If $i\equiv j \pmod{3^{\psi(k+2\epsilon)}}$,
    \[(B\Sigma_3)_{4i-\epsilon}^{4i+k} \simeq \Sigma^{4(i-j)} (B\Sigma_3)_{4j-\epsilon}^{4j+k}\]
\end{enumerate}
\end{theorem}

\subsubsection{The cell structures of the stunted lens spaces}
In order for the computation of differentials in the Atiyah--Hirzebruch spectral sequence, we need more detailed information about the cell structures of the stunted lens spaces.

We first recall some terminology from \cite{wangxu2017triviality}. 

\begin{defn}
Let $A, B, C$ and $D$ be CW spectra and let $i, q$ be maps
    \[A\overset{i}{\hookrightarrow}B,\quad B\overset{q}{\twoheadrightarrow}C.\]  
$(A,i)$ (or just $A$ if $i$ is clear in the context) is an $H\bbF_3$-subcomplex of $B$, if the map $i$ induces an injection on mod 3 homology.
We denote an $H\bbF_3$-subcomplex by an hooked arrow as above.

We say that $(C,q)$ (or just $C$ if $q$ is clear in the context) is an $H\bbF_3$-quotient complex of $B$, if the map $q$ induces a surjection on mod 3 homology. We denote an $H\bbF_3$-quotient complex by a double-headed arrow above.

Furthermore, we say $D$ is an $H\bbF_3$-subquotient of $B$, if $D$ is an $H\bbF_3$-subcomplex of an $H\bbF_3$-
quotient complex of $B$, or an $H\bbF_3$-quotient complex of an $H\bbF_3$-subcomplex of $B$.
\end{defn}

For any element $\alpha\in \pi_*^{cl}$, let $C\alpha$ denote the cofiber of $\alpha: S^{|\alpha|}\to S^0$. According to the solution of the odd primary Hopf invariant 1 problem, $P^1$ detects $\alpha_1$ and $\beta$ detects $3$; equivalently, $P^1$ acts non-trivially in $H^*(C\alpha_1;\bbF_3)$ and $\beta$ acts non-trivially in $H^*(C3;\bbF_3)$. 

\begin{lemma}\label{3andalpha1subquotient}
\hspace{5pt}
\begin{enumerate}
    \item For any $k$ odd, $(BC_3)^{k+1}_{k}\simeq  \Sigma^{k}C3 $. For any $k$ even, $(BC_3)^{k+1}_{k} \simeq S^k\vee S^{k+1}$.
    \item If $P^1$ acts non-trivially in $H^{k}((BC_3)^{k+4}_{k};\bbF_3)$, then $\Sigma^{k}C\alpha_1$ is a subquotient of $(BC_3)^{k+4}_{k}$.
\end{enumerate}
\end{lemma}
\begin{proof}
    \hspace{5pt}
    \begin{enumerate}
        \item By \cref{cohofBC3}, $\beta$ acts non-trivially on $H^{k}((BC_3)^{k+1}_{k};\bbF_3)$ if and only if $k$ is odd. Therefore, if $k$ is odd, the naturality of cohomology operation forces the two cell complex $(BC_3)^{k+1}_{k}$ to be homotopy equivalent to $\Sigma^{k}C3$. If $k$ is even, consider the cofiber sequence
        \[S^{k}\overset{\partial}{\to} \Sigma^{k-1}C3\simeq (BC_3)^{k}_{k-1} \hookrightarrow (BC_3)^{k+1}_{k-1} \twoheadrightarrow S^{k+1}. \]
        Since $\pi_1^{cl}=0$, $\partial$ does not factor through the bottom cell inclusion $S^{k-1}\hookrightarrow \Sigma^{k-1}C3$, so we have a commutative diagram
        \[\begin{tikzcd}
                        & S^{k} \arrow[d, "\partial"] \arrow[rd, "\tilde\partial"] &                    &     \\
S^{k-1} \arrow[r, hook] & \Sigma^{k-1}C3 \arrow[r, two heads]                      & S^k \arrow[r, "3"] & S^k
\end{tikzcd}\]
        Since for any nonzero $\tilde\partial\in \pi_0$, $3\cdot \tilde\partial\neq 0$. The cofiber sequence forces $\partial=0$, and we have $(BC_3)^{k+1}_{k} \simeq S^k\vee S^{k+1}$.
        \item If $k\equiv 0\pmod 4$, then by \cref{BC3splitting}, \[(BC_3)^{k+4}_{k} \simeq (B\Sigma_3)^{k+4}_{k}\vee \Sigma^{k+1}C3.\]
        The $(k+3)$-cell of $(B\Sigma_3)^{k+4}_{k}$ is a bottom cell for degree reasons, so we may consider the cofiber
        \[S^{k+3}\hookrightarrow (B\Sigma_3)^{k+4}_{k} \to T,\]
        which only consists of cells in dimension $k$ and $k+4$.
        If $P^1$ acts non-trivially in $H^{k}((BC_3)^{k+4}_{k};\bbF_3)$, naturality implies $P^1$ also acts non-trivially in $H^{k}(T;\bbF_3)$. Since $P^1$ detects $\alpha_1$, the naturality of cohomology operation forces $T$ to be homotopy equivalent to $\Sigma^{k}C\alpha_1$. 

        If $k\equiv -1\pmod 4$, then by \cref{BC3splitting}, \[(BC_3)^{k+4}_{k} \simeq (B\Sigma_3)^{k+4}_{k}\vee \Sigma^{k+2}C3.\]
        The $(k+1$)-cell is a top cell for degree reason,  so we may consider the fiber
        \[T'\hookrightarrow (B\Sigma_3)^{k+4}_{k} \to S^{k+1},\]
        which only consists of cells in dimension $k$ and $k+4$.
        If $P^1$ acts non-trivially in $H^{k}((BC_3)^{k+4}_{k};\bbF_3)$, naturality implies $P^1$ also acts non-trivially in $H^{k}(T';\bbF_3)$. Since $P^1$ detects $\alpha_1$, the naturality of cohomology operation forces $T'$ to be homotopy equivalent to $\Sigma^{k}C\alpha_1$. 

        If $k\equiv 1\ \text{or}\ 2 \pmod 4$, then \[(BC_3)^{k+4}_{k} \simeq X^{k+4}_{k}\vee \Sigma^{k+1+\epsilon}C3\]
        for $\epsilon =$ 0 or 1. The rest arguments follows similarly.
    \end{enumerate}
\end{proof}

Let us now consider a slightly more complicated situation, where the signs in $\calA$-module structure matters.

\begin{lemma}
    Let $Q_0$ denote the cofiber of the map $S^3 \vee S^0\overset{(\alpha_1, 3)}{\to} S^0$. Equivalently, $Q_0\simeq C3\vee_{S^0} C\alpha_1$. Then, $C3\wedge C\alpha_1$ fits into the following cofiber sequence
    \[S^{4} \overset{f_0}{\to} Q_0 \hookrightarrow C3\wedge C\alpha_1 \twoheadrightarrow S^5\]
    such that the composite $S^{4} \overset{f_0}{\to} Q_0 \twoheadrightarrow S^4\vee S^1$ is $(3, -\alpha_1)$.

    As a result, $(B\Sigma_3)_3^8$ fits into the cofiber sequence
    \[S^{7} \overset{g_0}{\to} \Sigma^3Q_0 \hookrightarrow (B\Sigma_3)_3^8 \twoheadrightarrow S^8\]
    where the composite $S^{7} \overset{g_0}{\to} \Sigma^3 Q_0 \twoheadrightarrow S^7\vee S^4$ is $(3, \alpha_1)$.
\end{lemma}
\begin{proof}
    By definition and \cite{may2001additivityoftraces}, $C3\wedge C\alpha_1$ fits into the following $4\times 4$-cofiber sequences
    \[\begin{tikzcd}
S^3 \arrow[r, "\alpha_1"] \arrow[d, "3"]        & S^0 \arrow[r, "i"] \arrow[d, "3"] & C\alpha_1 \arrow[r, "p"] \arrow[d, "3"]          & S^4 \arrow[d, "3"]         \\
S^3 \arrow[r, "\alpha_1"] \arrow[d, "j"]        & S^0 \arrow[r, "i"] \arrow[d, "j"] & C\alpha_1 \arrow[r, "p"] \arrow[d, "j"]          & S^4 \arrow[d, "j"]         \\
\Sigma^3C3 \arrow[r, "\alpha_1"] \arrow[d, "q"] & C3 \arrow[r, "i"] \arrow[d, "q"]  & C3\wedge C\alpha_1 \arrow[r, "p"] \arrow[d, "q"] & \Sigma^4C3 \arrow[d, "-q"] \\
S^4 \arrow[r, "\alpha_1"]                       & S^1 \arrow[r, "i"]                & \Sigma C\alpha_1 \arrow[r, "-p"]                 & S^5                       
\end{tikzcd}\]
    where all the squares commute except the bottom right square commutes up to a sign $-1$. We modify the signs by multiplying a $-1$ to the $\alpha_1$ and $-p$ on the bottom row.
    \[\begin{tikzcd}
S^3 \arrow[r, "\alpha_1"] \arrow[d, "3"]        & S^0 \arrow[r, "i"] \arrow[d, "3"] & C\alpha_1 \arrow[r, "p"] \arrow[d, "3"]          & S^4 \arrow[d, "3"]         \\
S^3 \arrow[r, "\alpha_1"] \arrow[d, "j"]        & S^0 \arrow[r, "i"] \arrow[d, "j"] & C\alpha_1 \arrow[r, "p"] \arrow[d, "j"]          & S^4 \arrow[d, "j"]         \\
\Sigma^3C3 \arrow[r, "\alpha_1"] \arrow[d, "q"] & C3 \arrow[r, "i"] \arrow[d, "q"]  & C3\wedge C\alpha_1 \arrow[r, "p"] \arrow[d, "q"] & \Sigma^4C3 \arrow[d, "-q"] \\
S^4 \arrow[r, "-\alpha_1"]                      & S^1 \arrow[r, "i"]                & \Sigma C\alpha_1 \arrow[r, "p"]                  & S^5                       
\end{tikzcd}\]
    so that each row and columns are still cofiber sequences and all the squares commute except the left lower one commutes up to a sign $-1$. 
    
    Now in the second diagram, the right lower $3\times 3$ diagram commutes and are cofiber sequences. The pushout of $\begin{tikzcd}
S^0 \arrow[r, "i"] \arrow[d, "j"] & C\alpha_1 \\
C3                                &          
\end{tikzcd}$ is the $Q_0$ we defined. Therefore, the axioms of triangulated categories imply there is a cofiber sequence 
\[S^4 \overset{f_0}{\to} Q_0 \to C_3\wedge C\alpha_1\overset{p\circ q}{\to} S^5 .\]

By the Octahedron axiom, $S^4 \overset{f_0}{\to} Q_0 \twoheadrightarrow S^1\vee S^4 \twoheadrightarrow S^1$ must be the boundary homomorphism of $\Sigma C\alpha_1\overset{p}{\to}S^5$, which is $-\alpha_1$. Similarly, $S^4 \overset{f_0}{\to} Q_0 \twoheadrightarrow S^1\vee S^4 \twoheadrightarrow S^4$ must be the boundary homomorphism of $\Sigma^4 C3\overset{-q}{\to}S^5$, which is $3$. The first claim then follows.

As a result, the cell diagram of $C3\wedge C\alpha_1$ is as drawn at the right part of \cref{C3Calpha1}. Note that by the Cartan formula, the $\calA$-module structure of $C3\wedge C\alpha_1$ is as drawn at the left part of \cref{C3Calpha1}. 

\begin{figure}
    \centering
    \begin{tikzpicture}[scale=0.35]
\draw[thick](-1,0) circle (0.5);
\node at (-1,0) {5};
\draw[thick](-1,-2) circle (0.5);
\node at (-1,-2) {4};
\draw[thick](-1,-8) circle (0.5);
\node at (-1,-8) {1};
\draw[thick](-1,-10) circle (0.5);
\node at (-1,-10) {0};

\draw[thick] (-1,-0.5)--(-1,-1.5);
\node at (-0.7,-1){$\beta$};
\draw[thick] (-1,-8.5)--(-1,-9.5);
\node at (-0.7,-9){$\beta$};

\draw[thick] (-1.5,0) .. controls (-3,-2) and (-3,-6) .. (-1.5,-8);
\node at (-3.5,-4){$P^{1}$};

\draw[thick] (-0.5,-2) .. controls (1,-4) and (1,-8) .. (-0.5,-10);
\node at (1.5,-6){$P^{1}$};


\draw[thick](13,0) circle (0.5);
\node at (13,0) {5};
\draw[thick](13,-2) circle (0.5);
\node at (13,-2) {4};
\draw[thick](13,-8) circle (0.5);
\node at (13,-8) {1};
\draw[thick](13,-10) circle (0.5);
\node at (13,-10) {0};

\draw[thick] (13,-0.5)--(13,-1.5);
\node at (13.3,-1){$3$};
\draw[thick] (13,-8.5)--(13,-9.5);
\node at (13.3,-9){$3$};

\draw[thick] (12.5,0) .. controls (11,-2) and (11,-6) .. (12.5,-8);
\node at (10.5,-4){$-\alpha_1$};

\draw[thick] (13.5,-2) .. controls (15,-4) and (15,-8) .. (13.5,-10);
\node at (15.5,-6){$\alpha_1$};


\end{tikzpicture}
    \caption{The $\calA$-module structure and the cell diagram of  $C3\wedge C\alpha_1$}
    \label{C3Calpha1}
\end{figure}
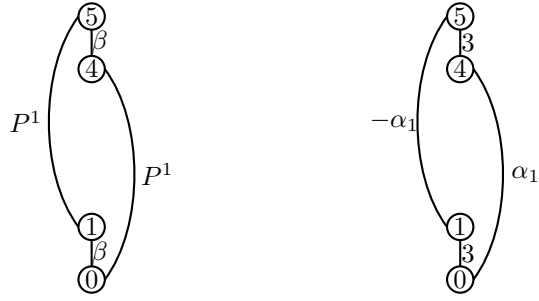

On the other hand, consider $(B\Sigma_3)_3^8$, which has cells in dimensions $3, 4, 7, 8$. The $\calA$-module structure on $(B\Sigma_3)_3^8$ can be read off from \cref{cohofBC3} and is drawn at the left of \cref{BSigma3:3-8}. In particular, $(B\Sigma_3)_3^8$ and $\Sigma^3C3\wedge C\alpha_1$ are not homotopy equivalent, since their $\calA$-module structure differs by a sign. 

If we restrict our attention to $(B\Sigma_3)_3^7$, the non-trivial $P^1$ and $\beta$-action on $H^3((B\Sigma_3)_3^7; \bbF_3)$ implies that $(B\Sigma_3)_3^7$ is homotopy equivalent to $\Sigma^3 Q_0$ by \cref{3andalpha1subquotient}. Therefore, we must also have the cofiber sequence
\[S^{7} \overset{g_0}{\to} \Sigma^3 Q_0\simeq (B\Sigma_3)_3^7  \hookrightarrow (B\Sigma_3)_3^8 \twoheadrightarrow S^8.\]
Since $(B\Sigma_3)_3^8 \not\simeq \Sigma^3C3\wedge C\alpha_1$, the only possibility is that the signs differs by $-1$. In other words, the composite $S^{7} \overset{g_0}{\to} \Sigma^3 Q_0 \twoheadrightarrow S^7\vee S^4$ is either $(3, \alpha_1)$ or $(-3, -\alpha_1)$. Without loss of generality, we may choose it to be $(3, \alpha_1)$.

The cell diagram of $(B\Sigma_3)_3^8$ is drawn at the right part of \cref{BSigma3:3-8}.

\begin{figure}
    \centering
    \begin{tikzpicture}[scale=0.35]
\draw[thick](-1,0) circle (0.5);
\node at (-1,0) {8};
\draw[thick](-1,-2) circle (0.5);
\node at (-1,-2) {7};
\draw[thick](-1,-8) circle (0.5);
\node at (-1,-8) {4};
\draw[thick](-1,-10) circle (0.5);
\node at (-1,-10) {3};

\draw[thick] (-1,-0.5)--(-1,-1.5);
\node at (-0.7,-1){$\beta$};
\draw[thick] (-1,-8.5)--(-1,-9.5);
\node at (-0.7,-9){$\beta$};

\draw[thick] (-1.5,0) .. controls (-3,-2) and (-3,-6) .. (-1.5,-8);
\node at (-3.5,-4){$-P^1$};

\draw[thick] (-0.5,-2) .. controls (1,-4) and (1,-8) .. (-0.5,-10);
\node at (1.5,-6){$P^1$};

\draw[thick](13,0) circle (0.5);
\node at (13,0) {8};
\draw[thick](13,-2) circle (0.5);
\node at (13,-2) {7};
\draw[thick](13,-8) circle (0.5);
\node at (13,-8) {4};
\draw[thick](13,-10) circle (0.5);
\node at (13,-10) {3};

\draw[thick] (13,-0.5)--(13,-1.5);
\node at (13.3,-1){$3$};
\draw[thick] (13,-8.5)--(13,-9.5);
\node at (13.3,-9){$3$};

\draw[thick] (12.5,0) .. controls (11,-2) and (11,-6) .. (12.5,-8);
\node at (10.5,-4){$\alpha_1$};

\draw[thick] (13.5,-2) .. controls (15,-4) and (15,-8) .. (13.5,-10);
\node at (15.5,-6){$\alpha_1$};

\end{tikzpicture}
    \caption{$\calA$-module structure and the cell diagram of  $(B\Sigma_3)_3^8$}
    \label{BSigma3:3-8}
\end{figure}
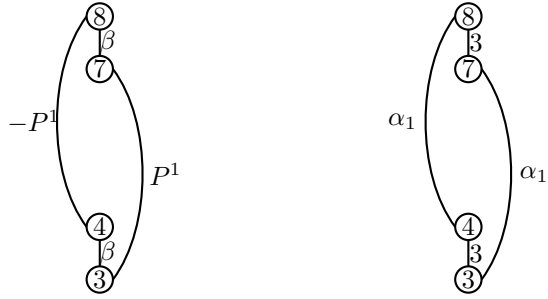
\end{proof}

By considering the Spanier--Whitehead dual and the James periodicity, we may obtain all the information about $C3$ or $C\alpha_1$ as subquotients of $(BC_3)_n^m$ for all $m\geq n$. 

\begin{lemma}\label{subquotientofBSigma3kk-5}
    Define $T_1$ to be the 3-cell complex fitting into the following cofiber sequences
     \[\Sigma^{-5}C3 \hookrightarrow T_1\twoheadrightarrow S^{0}\overset{a_1}{\to} \Sigma^{-4} C3, \]
    \[S^{-5}\hookrightarrow T_1\twoheadrightarrow \Sigma^{-4} C\alpha_1\overset{b_1}{\to} S^{-4}.\]
    Define $T_2$ to be the 3-cell complex fitting into the following cofiber sequences
     \[\Sigma^{-5}C\alpha_1 \hookrightarrow T_2\twoheadrightarrow S^{0}\overset{a_2}{\to} \Sigma^{-4}C\alpha_1, \]
    \[S^{-5}\hookrightarrow T_1\twoheadrightarrow \Sigma^{-1} C3\overset{b_2}{\to} S^{-4}.\]
    Define $T_3$ to be the 4-cell complex $\Sigma^{-8}(B\Sigma_3)_3^8$. The cell diagram of $T_1,T_2, T_3$ are pictured as follows.
    
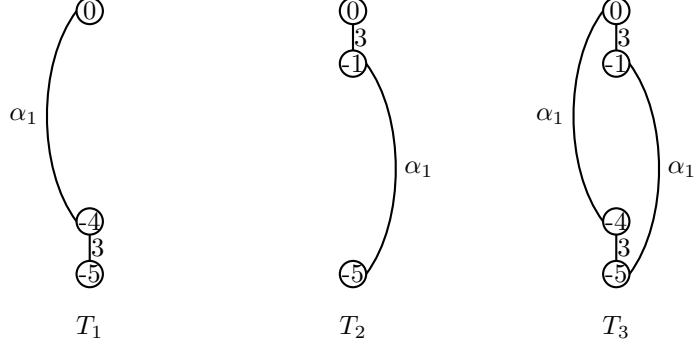
\[\begin{tikzpicture}[scale=0.35]
\draw[thick](-1,0) circle (0.5);
\node at (-1,0) {0};
\draw[thick](-1,-8) circle (0.5);
\node at (-1,-8) {-4};
\draw[thick](-1,-10) circle (0.5);
\node at (-1,-10) {-5};

\draw[thick] (-1,-8.5)--(-1,-9.5);
\node at (-0.7,-9){$3$};

\draw[thick] (-1.5,0) .. controls (-3,-2) and (-3,-6) .. (-1.5,-8);
\node at (-3.5,-4){$\alpha_1$};

\node at (-1,-12){$T_1$};

\draw[thick](9,0) circle (0.5);
\node at (9,0) {0};
\draw[thick](9,-2) circle (0.5);
\node at (9,-2) {-1};
\draw[thick](9,-10) circle (0.5);
\node at (9,-10) {-5};

\draw[thick] (9,-0.5)--(9,-1.5);
\node at (9.3,-1){$3$};

\draw[thick] (9.5,-2) .. controls (11,-4) and (11,-8) .. (9.5,-10);
\node at (11.5,-6){$\alpha_1$};

\node at (9,-12){$T_2$};

\draw[thick](19,0) circle (0.5);
\node at (19,0) {0};
\draw[thick](19,-2) circle (0.5);
\node at (19,-2) {-1};
\draw[thick](19,-8) circle (0.5);
\node at (19,-8) {-4};
\draw[thick](19,-10) circle (0.5);
\node at (19,-10) {-5};

\draw[thick] (19,-0.5)--(19,-1.5);
\node at (19.3,-1){$3$};
\draw[thick] (19,-8.5)--(19,-9.5);
\node at (19.3,-9){$3$};

\draw[thick] (18.5,0) .. controls (17,-2) and (17,-6) .. (18.5,-8);
\node at (16.5,-4){$\alpha_1$};

\draw[thick] (19.5,-2) .. controls (21,-4) and (21,-8) .. (19.5,-10);
\node at (21.5,-6){$\alpha_1$};

\node at (19,-12){$T_3$};

\end{tikzpicture}\]

Then \begin{enumerate}
    \item $\Sigma^kT_1$ is a quotient complex of $(BC_3)_{k-5}^k$ if and only if $k \equiv 0\ \mathrm{or}\ 6 \pmod {12}$;
    \item $\Sigma^kT_2$ is a subcomplex of $(BC_3)_{k-5}^k$ if and only if $k \equiv 4\ \mathrm{or}\ 10 \pmod {12}$;
    \item $\Sigma^kT_3$ is a wedge summand of $(BC_3)_{k-5}^k$ if and only if $k \equiv 8\ \mathrm{or}\ 2 \pmod {12}$.
\end{enumerate}
\end{lemma}
\begin{proof} 
We first remark that in the first two cases, the signs do not matter because the induced map on cohomology may be modified by multiplying $-1$. 

If $k \equiv 0, 4\ \mathrm{or}\ 8 \pmod {12}$, $(BC_3)^k_{k-5}\simeq (B\Sigma_3)^k_{k-5}\vee \Sigma^{k-3}C3$ where $\Sigma^{k-3}C3 \simeq X^{k}_{k-5}$. On the other hand, If $k \equiv 6, 10\ \mathrm{or}\ 2 \pmod {12}$, $(BC_3)^k_{k-5}\simeq X^k_{k-5}\vee \Sigma^{k-3}C3$ where $\Sigma^{k-3}C3 \simeq (B\Sigma_3)^{k}_{k-5}$. The rest proof only deal with the former case since the arguments for the latter case are basically the same.
\begin{enumerate}
    \item When $k \equiv 0\pmod {12}$, $(BC_3)^k_{k-5}\simeq (B\Sigma_3)^k_{k-5}\vee \Sigma^{k-3}C3$. Restricting to $(B\Sigma_3)^{k-1}_{k-5}$, by \cref{cohofBC3}, there is a non-trivial $\beta$-action on $H^{k-5}((B\Sigma_3)^{k-1}_{k-5}; \bbF_3)$, while $P^1$-acts trivially on $H^{k-5}((B\Sigma_3)^{k-1}_{k-5}; \bbF_3)$. This implies that \[(B\Sigma_3)^{k-1}_{k-5} \simeq S^{k-1}\vee \Sigma^{k-5}C3.\] We may consider the inclusion $S^{k-1}\hookrightarrow (B\Sigma_3)^k_{k-5}$, whose cofiber is $\Sigma^{k}T_1$ because there is a non-trivial $P^1$-action on $H^{k-4}((B\Sigma_3)^{k}_{k-5}; \bbF_3)$. 
    \item When $k \equiv 4\pmod {12}$, $(BC_3)^k_{k-5}\simeq (B\Sigma_3)^k_{k-5}\vee \Sigma^{k-3}C3$. Restricting to $(B\Sigma_3)^{k}_{k-4}$, by \cref{cohofBC3}, there is a non-trivial $\beta$-action on $H^{k-1}((B\Sigma_3)^{k}_{k-4}; \bbF_3)$, while $P^1$-acts trivially on $H^{k-4}((B\Sigma_3)^{k}_{k-4}; \bbF_3)$. This implies that \[(B\Sigma_3)^{k}_{k-4} \simeq S^{k-4}\vee \Sigma^{k-1}C3.\] We may consider the projection $(B\Sigma_3)^k_{k-5} \twoheadrightarrow S^{k-4}$, whose fiber is $\Sigma^{k}T_2$ because there is a non-trivial $P^1$-action on $H^{k-5}((B\Sigma_3)^{k}_{k-5}; \bbF_3)$. 
    \item When $k\equiv 8\pmod{12}$, $(BC_3)^k_{k-5}\simeq (B\Sigma_3)^k_{k-5}\vee \Sigma^{k-3}C3$. By the James periodicity in \cref{jamesperiodicity},
    \[(B\Sigma_3)^k_{k-5} \simeq \Sigma^kT_3.\]
\end{enumerate}
\end{proof}

We can also obtain the information about $C\alpha_2$ as subquotients of the stunted lens spaces.

\begin{lemma}\label{alpha2attachingmap}
\hspace{5pt}
\begin{enumerate} 
    \item $\Sigma^{k} C\alpha_2$ is a quotient complex of $(B\Sigma_3)_k^{k+8}$ if $k\equiv 12\ \text{or}\ 24 \pmod{36}$. Similarly, $\Sigma^{k} C\alpha_2$ is a quotient complex of $X_k^{k+8}$ if $k\equiv 6\ \text{or}\ 30 \pmod{36}$. 
    \item $\Sigma^{k-9} C\alpha_2$ is a subcomplex of $(B\Sigma_3)_{k-9}^{k-1}$ if $k\equiv 12\ \text{or}\ 24 \pmod{36}$. Similarly, $\Sigma^{k-9} C\alpha_2$ is a subcomplex of $X_{k-9}^{k-1}$ if $k\equiv 6\ \text{or}\ 30 \pmod{36}$. 
\end{enumerate}
\end{lemma}
\begin{proof} \hspace{5pt}
\begin{enumerate}
    \item Consider the cofiber sequence
    \[\Sigma^{-1} (B\Sigma_3)_{k+1}^{k+8}=\Sigma^{-1} (B\Sigma_3)_{k+3}^{k+8}\overset{\delta}{\to}  S^k\hookrightarrow(B\Sigma_3)_k^{k+8}.\]
    According to \cref{coreducible}, the bottom cell of $(B\Sigma_3)_{k}^{k+8}$ does not split off if $k\equiv 12\ \mathrm{or}\ 24 \pmod {36}$, so that $\delta$ is not nullhomotopic. Since $\pi_2=0$, the restriction $\delta|_{S^{k+2}}$ is nullhomotopic and thus factors through $\Sigma^{-1} (B\Sigma_3)_{k+4}^{k+8}$. Since \cref{cohofBC3} shows that $P^1$ acts trivially in $H^k((B\Sigma_3)_k^{k+8}; \bbF_3)$, the restriction $\delta|_{S^{k+3}} \not\simeq \alpha_1$, so it must be nullhomotopic as well. Note that $\pi_{*\leq7}^{cl}$ is nonzero only if $n=0, 3, 7$. Therefore, $\delta$ must factor through the composition
    \[\Sigma^{-1} (B\Sigma_3)_{k+1}^{k+8} \twoheadrightarrow S^{k+7}\overset{\alpha_2}{\to} S^k.\]
    As a result, we have cofiber sequences forming a pullback in the middle
    \[\begin{tikzcd}
S^k \arrow[d, Rightarrow, no head] \arrow[r, hook] & (B\Sigma_3)_{k}^{k+8} \arrow[r, two heads] \arrow[d, two heads] & (B\Sigma_3)_{k+1}^{k+8} \arrow[d, two heads] \arrow[r, "\Sigma\delta"] & S^{k+1} \arrow[d, Rightarrow, no head] \\
S^k \arrow[r, hook]                                & \Sigma^k C\alpha_2 \arrow[r, two heads]                         & S^{k+8} \arrow[r, "\alpha_2"]                                          & S^{k+1}                      
\end{tikzcd}\]
    and by \cite[Lemma 4.4]{wangxu2017triviality}, $(B\Sigma_3)_{k}^{k+8} \to \Sigma^k C\alpha_2$ is a quotient.
    Similarly, \cref{coreducible} tells that the bottom cell of $X_{k}^{k+8}$ does not split off if $k\equiv 6\ \text{or}\ 30 \pmod{36}$. The rest argument follows similarly.

    \item The ideas are essentially dual to part (1). Consider the cofiber sequence
    \[(B\Sigma_3)_{k-9}^{k-1}\twoheadrightarrow  S^{k-1} \overset{\epsilon}{\to}\Sigma(B\Sigma_3)_{k-9}^{k-2}\simeq \Sigma(B\Sigma_3)_{k-9}^{k-4}.\]
    According to \cref{coreducible}, the top cell of $(B\Sigma_3)_{k-9}^{k-1}$ does not split off if $k\equiv 12\ \mathrm{or}\ 24 \pmod {36}$, so that $\epsilon$ is not nullhomotopic. Since $\pi_2=0$, the composition $S^{k-1}\overset{\epsilon}{\to}\Sigma(B\Sigma_3)_{k-9}^{k-4} \twoheadrightarrow S^{k-3}$ is nullhomotopic and thus $\epsilon$ factors through $\Sigma (B\Sigma_3)_{k-9}^{k-5}$. Since \cref{cohofBC3} shows that $P^1$ acts trivially in $H^{k-5}((B\Sigma_3)_{k-9}^{k-1}; \bbF_3)$, the composition $S^{k-1}\overset{\epsilon}{\to}\Sigma(B\Sigma_3)_{k-9}^{k-5} \twoheadrightarrow S^{k-4}$ is not $\alpha_1$, so it must be nullhomotopic as well. Note that $\pi_{*\leq7}^{cl}$ is nonzero only if $n=0, 3, 7$. Therefore, $\epsilon$ must factor through the composition
    \[S^{k-1}\overset{\alpha_2}{\to} S^{k-8}\hookrightarrow \Sigma(B\Sigma_3)_{k-9}^{k-2}.\]
    As a result, we have cofiber sequences and a pullback in the middle
    \[\begin{tikzcd}
S^{k-2} \arrow[r, "\alpha_2"] \arrow[d, Rightarrow, no head] & S^{k-9} \arrow[d, hook] \arrow[r, hook] & \Sigma^{k-9}C\alpha_2 \arrow[d, hook] \arrow[r, two heads] & S^{k-1} \arrow[d, Rightarrow, no head] \\
S^{k-2} \arrow[r, "\Sigma^{-1}\epsilon"]            & (B\Sigma_3)_{k-9}^{k-2} \arrow[r, hook] & (B\Sigma_3)_{k-9}^{k-1} \arrow[r, two heads]               & {S^{k-1},}                   
\end{tikzcd}\]
    and by \cite[Lemma 4.4]{wangxu2017triviality}, $ \Sigma^k C\alpha_2\to (B\Sigma_3)_{k-9}^{k-1}$ is a subcomplex.
    Similarly, \cref{coreducible} tells that the top cell of $X_{k-9}^{k-1}$ does not split off if $k\equiv 6\ \text{or}\ 30 \pmod{36}$. The rest argument follows similarly.
\end{enumerate}
\end{proof}

We remark that if $k\equiv 0\pmod{36}$, then \cref{coreducible} does imply that 
\[(B\Sigma_3)_k^{k+8}\simeq S^k\vee (B\Sigma_3)_{k+3}^{k+8}, \quad (B\Sigma_3)_{k-9}^{k-1}\simeq  S^{k-1}\vee (B\Sigma_3)_{k-9}^{k-4}.\]
Similarly, if $k\equiv 18 \pmod{36}$, then 
\[X_k^{k+8}\simeq S^k\vee X_{k+3}^{k+8}, \quad X_{k-9}^{k-1}\simeq  S^{k-1}\vee X_{k-9}^{k-4}.\]

For illustrative purposes, a cell diagram of $(BC_3)_{-9}^9$ with $3, \alpha_1, \alpha_2$-attaching maps drawn is shown in \cref{BC3^9_-9}. In particular, the splitting $(BC_3)_{-9}^9\simeq (B\Sigma_3)^9_{-9}\vee X^9_{-9}$ is demonstrated where the upper part is $(B\Sigma_3)^9_{-9}$ and the lower part is $X^9_{-9}$.

\begin{figure}
    \centering
    \begin{tikzpicture}[scale=0.4]

\draw[thick](0, 5) circle (0.5);
\node at (0, 5) {8};
\draw[thick](-2, 5) circle (0.5);
\node at (-2, 5) {7};
\draw[thick](-8, 5) circle (0.5);
\node at (-8, 5) {4};
\draw[thick](-10, 5) circle (0.5);
\node at (-10, 5) {3};
\draw[thick](-16, 5) circle (0.5);
\node at (-16, 5) {0};
\draw[thick](-18, 5) circle (0.5);
\node at (-18, 5) {-1};
\draw[thick](-24, 5) circle (0.5);
\node at (-24, 5) {-4};
\draw[thick](-26, 5) circle (0.5);
\node at (-26, 5) {-5};
\draw[thick](-32, 5) circle (0.5);
\node at (-32, 5) {-8};
\draw[thick](-34, 5) circle (0.5);
\node at (-34, 5) {-9};

\draw[thick] (-32.5, 5)--(-33.5, 5);
\node at (-33, 5.2){$3$};
\draw[thick] (-24.5, 5)--(-25.5, 5);
\node at (-25, 5.2){$3$};
\draw[thick] (-16.5, 5)--(-17.5, 5);
\node at (-17, 5.2){3};
\draw[thick] (-0.5, 5)--(-1.5, 5);
\node at (-1, 5.2){$3$};
\draw[thick] (-8.5, 5)--(-9.5, 5);
\node at (-9, 5.2){3};

\draw[thick] (0, 4.5) .. controls (-2, 3) and (-6, 3) .. (-8, 4.5);
\node at (-4, 3){$\alpha_1$};

\draw[thick] (-2, 5.5) .. controls (-4, 7) and (-8, 7) .. (-10, 5.5);
\node at (-6, 7){$\alpha_1$};

\draw[thick] (-18, 5.5) .. controls (-16, 7) and (-12, 7) .. (-10, 5.5);
\node at (-14, 7){$\alpha_1$};

\draw[thick] (-24, 4.5) .. controls (-22, 3) and (-18, 3) .. (-16, 4.5);
\node at (-20, 3){$\alpha_1$};
\draw[thick] (-32, 4.5) .. controls (-30, 3) and (-26, 3) .. (-24, 4.5);
\node at (-28, 3){$\alpha_1$};
\draw[thick] (-34, 5.5) .. controls (-32, 7) and (-28, 7) .. (-26, 5.5);
\node at (-30, 7){$\alpha_1$};

\draw[thick](2, -1) circle (0.5);
\node at (2, -1) {9};
\draw[thick](-4, -1) circle (0.5);
\node at (-4, -1) {6};
\draw[thick](-6, -1) circle (0.5);
\node at (-6 , -1) {5};
\draw[thick](-12, -1) circle (0.5);
\node at (-12, -1) {2};
\draw[thick](-14, -1) circle (0.5);
\node at (-14, -1) {1};
\draw[thick](-20, -1) circle (0.5);
\node at (-20, -1) {-2};
\draw[thick](-22, -1) circle (0.5);
\node at (-22, -1) {-3};
\draw[thick](-28, -1) circle (0.5);
\node at (-28, -1) {-6};
\draw[thick](-30, -1) circle (0.5);
\node at (-30, -1) {-7};

\draw[thick] (-20.5, -1)--(-21.5, -1);
\node at (-21, -0.8){$3$};
\draw[thick] (-28.5, -1)--(-29.5, -1);
\node at (-29, -0.8){$3$};
\draw[thick] (-5.5, -1)--(-4.5, -1);
\node at (-5, -0.8){$3$};
\draw[thick] (-13.5, -1)--(-12.5, -1);
\node at (-13, -0.8){$3$};

\draw[thick] (2, -0.5) .. controls (0, 1) and (-4, 1) .. (-6, -0.5);
\node at (-2, 1){$\alpha_1$};

\draw[thick] (-4, -1.5) .. controls (-6, -3) and (-10, -3) .. (-12, -1.5);
\node at (-8, -3){$\alpha_1$};

\draw[thick] (-20, -1.5) .. controls (-18, -3) and (-14, -3) .. (-12, -1.5);
\node at (-16, -3){$\alpha_1$};

\draw[thick] (-14, -0.5) .. controls (-16, 1) and (-20, 1) .. (-22, -0.5);
\node at (-18, 1){$\alpha_1$};

\draw[thick] (-30, -0.5) .. controls (-28, 1) and (-24, 1) .. (-22, -0.5);
\node at (-26, 1){$\alpha_1$};


\draw[thick, dashed] (-12, -1.5) .. controls (-16, -4.5) and (-24, -4.5) .. (-28, -1.5);
\node at (-20, -4.3){$\alpha_2$};

\draw[thick, dashed] (-6, -0.5) .. controls (-10, 2.5) and (-18, 2.5) .. (-22, -0.5);
\node at (-14, 2.3){$\alpha_2$};

\end{tikzpicture}
    \caption{A cell-diagram of $(BC_3)_{-9}^9$ and the splitting}
    \label{BC3^9_-9}
\end{figure}
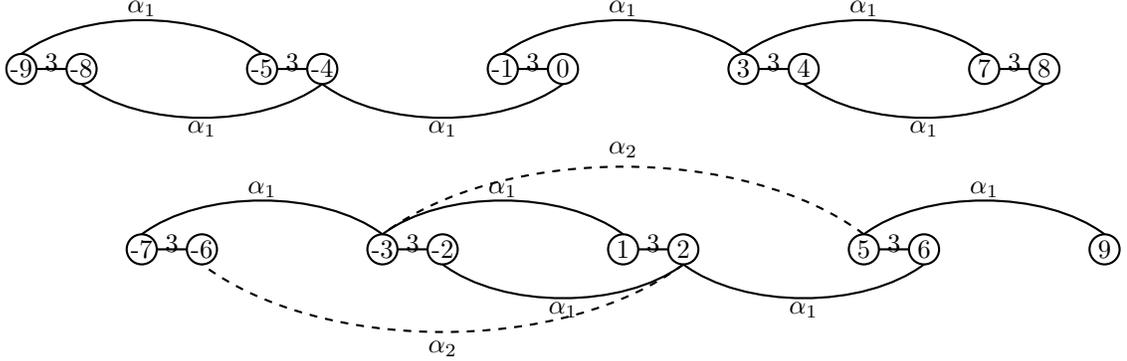

We note the following general facts.

\begin{lemma}\label{etaeta^2}
    \hspace{5pt}
    \begin{enumerate}
        \item For $\alpha, \beta\in \pi_*^{cl}$, consider the cofiber sequence
        \[S^{|\alpha|}\vee S^{|\beta|}\overset{(\alpha, \beta)}{\to} S^0 \to F_1 \twoheadrightarrow S^{|\alpha|+1}\vee S^{|\beta|+1}.\]
        
        If $\alpha$ is a multiple of $\beta$, then $F_1\simeq C\beta \vee S^{|\alpha|+1}$.
        \item More generally, for $\alpha, \beta, \gamma\in \pi_*^{cl}$ with $\beta \gamma=0$, consider the cofiber sequence
        \[S^{|\alpha|}\vee S^{|\beta|+|\gamma|+1} \overset{f}{\to} C\gamma \hookrightarrow F_2 \twoheadrightarrow S^{|\alpha|+1}\vee S^{|\beta|+|\gamma|+2}\]
        where $f|_{S^{|\alpha|}}$ factors as $S^{|\alpha|} \overset{\alpha}{\to}S^0 \hookrightarrow C\gamma$, and $f|_{S^{|\beta|+|\gamma|+1}}: S^{|\beta|+|\gamma|+1} \to  C\gamma$ post-composed with the projection $C\gamma \twoheadrightarrow S^{|\gamma|+1}$ equals $\beta$ (such a map exists since $\beta\gamma=0$).

        If $\alpha \in \langle \_, \beta, \gamma\rangle$, then $F_2\simeq T\vee S^{|\alpha|+1}$, where $T$ is the three cell complex that fits into the cofiber sequences
        \[C\gamma \hookrightarrow T \twoheadrightarrow S^{|\beta|+|\gamma|+2}\overset{\tilde\beta}{\to} \Sigma C\gamma, \]
        \[S^{0}\hookrightarrow T\twoheadrightarrow \Sigma^{|\gamma|+1}C\beta\overset{\tilde\gamma}{\to} S^{1}.\]
    \end{enumerate}
\end{lemma}
\begin{proof}
    \hspace{5pt}
    \begin{enumerate}
        \item By the octahedron axiom, $F_1$ is the cofiber of the composite \[S^{|\alpha|}\overset{\alpha}{\to}S^0\hookrightarrow C\beta.\]
        Since $\alpha$ is a multiple of $\beta$, the composite is nullhomotopic, so that $F_1\simeq C\beta \vee S^{|\alpha|+1}$.
        \item Note that $T$ is the cofiber of the map $f|_{S^{|\beta|+|\gamma|+1}}$. Assume $\alpha =\langle \epsilon, \beta, \gamma\rangle.$ Then by definition, we can find lifts $\tilde \gamma$ and $\tilde \epsilon$
        \[\begin{tikzcd}
S^{|\gamma|} \arrow[d, "i"', hook] \arrow[rd, "\gamma"] &     & S^{|\gamma|+|\beta|+|\epsilon|+1} \arrow[rd, "\epsilon"'] \arrow[r, "\tilde\epsilon"] & \Sigma^{|\gamma|}C\beta \arrow[d, "p", two heads] \\
\Sigma^{|\gamma|}C\beta \arrow[r, "\tilde\gamma"']      & S^0 &                                                                                       & S^{|\gamma|+|\beta|+1}                           
\end{tikzcd}\]
    such that $\alpha=\tilde \gamma \circ \tilde \epsilon$. Now the $F_2$ defined in the proposition is the cofiber of the composite
    \[S^{|\alpha|}\overset{\alpha}{\to} S^0 \hookrightarrow C\gamma \hookrightarrow T.\]
    Notice $\tilde\gamma$ post-composed with $S^0\hookrightarrow T$ is the consecutive two terms in a cofiber sequence, so in particular, the above composite is nullhomotopic. As a result, $F_2\simeq T\vee S^{|\alpha|+1}$.
    \end{enumerate}
\end{proof}

We can now prove the following lemma about longer attaching maps.

\begin{lemma}\label{alpha3attachingmap}
    $\Sigma^{k-13}C(\bar{\alpha}_3)$ is a subcomplex of $(B\Sigma_3)_{k-13}^{k-1}$ if $k\equiv 36 \ \text{or}\ 72 \pmod{108}$. Similarly, $\Sigma^{k-13}C(\bar{\alpha}_3)$ is a subcomplex of $X_{k-13}^{k-1}$ if $k\equiv 18 \ \text{or}\ 90 \pmod{108}$. 
    
    Dually, $\Sigma^{k}C(\bar{\alpha}_3)$ is a quotient complex of $(B\Sigma_3)_{k}^{k+12}$ if $k\equiv 36 \ \text{or}\ 72 \pmod{108}$, and $\Sigma^{k}C(\bar{\alpha}_3)$ is a quotient complex of $X_{k}^{k+12}$ if $k\equiv 18 \ \text{or}\ 90 \pmod{108}$.
\end{lemma}
\begin{proof}
    According to \cref{coreducible}, if $k\equiv 36 \ \text{or}\ 72 \pmod{108}$, the top cell of $(B\Sigma_3)_{k-12}^{k-1}$ splits off, while it does not split off in $(B\Sigma_3)_{k-13}^{k-1}$. Therefore, in the cofiber sequence
    \[\Sigma^{-1}(B\Sigma_3)_{k-12}^{k-1}\simeq S^{k-2}\vee \Sigma^{-1}(B\Sigma_3)_{k-12}^{k-2}\overset{\delta}{\to} S^{k-13}\hookrightarrow (B\Sigma_3)_{k-13}^{k-1}\twoheadrightarrow (B\Sigma_3)_{k-12}^{k-1},\]
    the composition $f: S^{k-2}\hookrightarrow \Sigma^{-1}(B\Sigma_3)_{k-12}^{k-1} \overset{\delta}{\to} S^{k-13}$ must be non-trivial. In particular, we have cofiber sequences and the pullback in the middle
    \[\begin{tikzcd}
S^{k-13} \arrow[r, hook] \arrow[d, Rightarrow, no head] & \Sigma^{k-13}Cf \arrow[r, two heads] \arrow[d, hook] & S^{k-1} \arrow[r, "f"] \arrow[d, hook]     & S^{k-12} \arrow[d, Rightarrow, no head] \\
S^{k-13} \arrow[r, hook]                       & (B\Sigma_3)_{k-13}^{k-1} \arrow[r, two heads]         & (B\Sigma_3)_{k-12}^{k-1} \arrow[r, "\Sigma\delta"] & S^{k-12}                      
\end{tikzcd}\]
    so that $\Sigma^{k-13}Cf$ is a subcomplex of $(B\Sigma_3)_{k-13}^{k-1}$. For degree reason, the only possibility is either $f=\pm\alpha_3\ \text{or} \ \pm\bar{\alpha}_3$. 

    Let us now restrict attention to $g: S^{k-2}\vee S^{k-13} \hookrightarrow S^{k-2}\vee \Sigma^{-1}(B\Sigma_3)_{k-12}^{k-2}\overset{\delta}{\to} S^{k-13}$. Since there is a non-trivial $\beta$-action on $H^{k-13}((B\Sigma_3)_{k-13}^{k-1}; \bbF_3)$, $g=(f, 3)$. By \cref{etaeta^2}, if $f=\pm \alpha_3$ which is a multiple of $3$, then $S^{k-1}$ splits off in $Cg$. In particular, this implies $S^{k-1}$ will also split off in $(B\Sigma_3)_{k-13}^{k-1}$, which is a contradiction. 
    
    Therefore, $f$ must be $\pm \bar{\alpha}_3$. The result then follows as $C\bar{\alpha}_3\simeq C(-\bar{\alpha}_3)$. Using similar arguments, we may obtain the claim for $X_{k-13}^{k-1}$. The claim for quotient complexes can be obtained using the Spanier--Whitehead duality.
\end{proof}

\begin{rmk}
    We may also use secondary cohomology operations \cite{adams1960nonexistence} to prove \cref{alpha3attachingmap}. For example, let us focus on $\Sigma^{-36}C(\bar{\alpha}_3)$ as a quotient complex of $(B\Sigma_3)_{-36}^{-24}$.

    There is a non-trivial $P^9$-action on $H^{-36}((B\Sigma_3)_{-36}^{0}; \bbF_3)$. According to \cite{liulevicius1962factor}, $P^9$ has a decomposition in terms of secondary cohomology operations:
    \[ P^9=a\mathfrak{R}_8+b_{11}\Gamma_{11}+c_{12}\Psi_{12}+b_{19}\Gamma_{19}+b_{27}\Gamma_{27}+b_{35}\Gamma_{35}\]
    for some $a, b_i, c_j\in\calA$ and $\mathfrak{R}_8, \Psi_{12}$ and $\Gamma_i$'s are secondary operations whose degrees equal their subscripts. In particular, $\mathfrak{R}_8$ detects $\alpha_2$, $\Gamma_{11}$ detects $\beta_1$, $\Psi_{12}$ detects $\bar{\alpha}_3$, and $\Gamma_{27}$ detects $\beta_2$. Since all primary cohomology operations on $H^{-36}((B\Sigma_3)_{-36}^{0}; \bbF_3)$ of degree less than $P^9$ are zero by \cref{cohofBC3} and the James periodicity, the above secondary operations are all well defined. After quotient off the bottom complex $(B\Sigma_3)_{-35}^{-25}$ as $S^{-36}$ splits off, the images of the relevant primary cohomology operations are also trivial for degree reasons. Therefore, the right and side of the equation must act non-trivially on the generator of $H^{-36}((B\Sigma_3)_{-36}^{0}; \bbF_3)$.

    The $\mathfrak{R}_8$ and $\Gamma_{11}$ act trivially since $S^{-36}$ splits off. By \cref{coreducible}, the $-1$-cell splits off from $(B\Sigma_3)_{-36}^{-1}$, so that $\Gamma_{35}$ also acts trivially. The Adams differential $d_2(h_{0}h_{20})\doteq b_{10}\cdot a_1h_0$ implies that $\Gamma_{19}$ is decomposable in the secondary Steenrod algebra in terms of $\mathfrak{R}_8$ and $\Gamma_{11}$ \cite{baues2004computatione3termadamsspectral}. Therefore, it also acts trivially on the generator of $H^{-36}((B\Sigma_3)_{-36}^{0}; \bbF_3)$. Finally, it can be shown that $\Sigma^{-36}\beta_2$ is not a subcomplex (by considering the non-trivial $P^1$-action on $H^{-13}(B\Sigma_3)_{-36}^{-9}; \bbF_3)$, which implies $\Gamma_{27}$ also acts trivially. Therefore, $\Psi_{12}$ must act non-trivially on $H^{-36}((B\Sigma)_{-36}^0, \bbF_3)$, so $\Sigma^{-36}C\bar{\alpha}_3$ must be a subquotient.
\end{rmk}

\begin{rmk}
    By \cref{coreducible}, the top cell of $(BC_3)_{k-13}^{k-1}$ does split off if $k\equiv 0 \ \text{or}\ 54 \pmod {108}$.
\end{rmk}

\section{The spoke-grading in \texorpdfstring{$C_3$-equivariant}{C3 equivariant} homotopy}\label{2}
We will recollect some facts in the $C_3$-equivariant homotopy categories. Throughout this section, everything will be 3-completed.

As suggested in the introduction, our computation of the $C_3$-equivariant stable homotopy groups are graded by the additional spoke-sphere $S^\Yright$. We here briefly recall its basic properties.

By definition, we have the following cofiber sequence
    \[{C_3}_+ \rightarrow S^0 \rightarrow S^\Yright\]
and we denote the last map by $a_\Yright$. A key property of the spoke sphere is that \[S^{\Yright} \wedge S^{\Yright} \simeq S^{\lambda} \vee \Sigma^2 {C_3}_+.\] In particular, the composition
\[S^{0} \xrightarrow{a_{\Yright}} S^{\Yright} \xrightarrow{a_{\Yright}\wedge id} S^{\Yright}\wedge S^{\Yright}\]
factors through the Euler class $a_\lambda$ associated to $\lambda$. We will abuse the notation by denoting the composite 
\[S^{\Yright}\to S^{\Yright}\wedge S^{\Yright} \twoheadrightarrow S^\lambda\]
by $a_\Yright$, following the same manner as in \cite{balderrama2025cpnequivariantmahowaldinvariants} (where they denote $a_\Yright$ by $a^{1/2}$). 
Thus we have the decomposition
\[a_\lambda:S^0 \xrightarrow{a_\Yright} S^ \Yright \xrightarrow{a_\Yright} S^\lambda.\]
Furthermore, applying the geometric fixed point functor $\Phi^{C_3}$
\[\Phi^{C_3}({C_3}_+)\simeq * \rightarrow S^0 \rightarrow \Phi^{C_3}(S^\Yright)=S^0,\]
since $\Phi^{C_3}(a_\lambda)=\id$, we see that $\Phi^{C_3}(a_\Yright)=\id$ always holds true.

Recall from \cref{defofspokegradedhomotopygroup}, for any $X \in \Sp^{C_{3}}$,
\[\pi_{i,j}^{C_3}(X):= \begin{cases}[S^{i-j+k\lambda} \wedge S^\Yright,X]^{C_{3}}&\text{for odd} \hspace{3pt} j=2k+1\\ [S^{i-j+k\lambda},X]^{C_{3}}&\text{for even} \hspace{3pt} j=2k.
\end{cases}
\]

When $j$ is even, the cofiber sequence ${C_3}_+ \rightarrow S^0 \overset{a_\Yright}{\to} S^\Yright$ induces the following long exact sequence of spoke-graded homotopy groups: 
\[\cdots \xrightarrow{tr} \pi^{C_3}_{i+1,j+1}(X) \xrightarrow{\cdot a_{\Yright}} \pi^{C_3}_{i,j}(X) \xrightarrow{Res} \pi^{cl}_{i}(X) \xrightarrow{tr} \pi^{C_3}_{i,j+1}(X) \xrightarrow{a_\Yright} \cdots,\]
where $\pi^{cl}_{\ast}(X)$ is the classical homotopy groups of the underlying spectrum of $X$ and the maps $Res$ and $tr$ are the usual restriction and transfer maps. When $j$ is odd, we still have the long exact sequences as above, which together with the maps $Res$ and $tr$ are discussed in detail in \cite{balderrama2025cpnequivariantmahowaldinvariants}.



An isomorphism of $RO(C_3)$-graded homotopy groups does not necessarily imply a $C_3$-weak equivalence; see \cite[Example 3.1]{guillouisaksen2024c2}. In contrast, a crucial feature of the spoke-graded homotopy groups is that they detect equivalences in $\Sp^{C_3}$.
\begin{prop}\label{spokedetectequivalence}
A map $f: X \rightarrow Y \in \Sp^{C_3}$ is an equivalence if and only if it induces isomorphisms between the spoke-graded homotopy groups of $X$ and $Y$.   
\end{prop}
\begin{proof}
The forward direction is trivial. For the backward direction, it is enough to prove that $f$ also induces an equivalence between the underlying spectra of $X$ and $Y$. Consider the following commutative diagram.
\begin{center}
\begin{tikzcd}
{\cdots \pi^{C_3}_{i+1,1}(X)} \arrow[r,"\cdot a_{\Yright}"] \arrow[d] & {\pi^{C_3}_{i,0}(X)} \arrow[r,"Res"] \arrow[d] & {\pi^{cl}_{i}(X)} \arrow[r,"tr"] \arrow[d] & {\pi^{C_3}_{i,1}(X)} \arrow[r,"\cdot a_{\Yright}"] \arrow[d] & {\pi^{C_3}_{i-1,0}(X) \cdots} \arrow[d] \\
{\cdots \pi^{C_3}_{i+1,1}(Y)} \arrow[r,"\cdot a_{\Yright}"']           & {\pi^{C_3}_{i,0}(Y)} \arrow[r,"Res"']           & {\pi^{cl}_{i}(Y)} \arrow[r,"tr"']           & {\pi^{C_3}_{i,1}(Y)} \arrow[r,"\cdot a_{\Yright}"']           & {\pi^{C_3}_{i-1,0}(Y) \cdots}          
\end{tikzcd}  
\end{center}
By assumption, the first two and the last two vertical maps are equivalences. The five lemma then yields the claimed result.
\end{proof}

Next we will set eyes on the the $C_3$-equivariant stable stems. 
Similar to the ideas of \cite{behrensshah2020c2}, we have the following identification.
\begin{prop}\label{S^0borel}
    There are isomorphisms \[\pi_{i, j}^{C_3}\cong \pi_{i-j}^{cl}F((BC_3)_{j}^\infty,S^0)\]
    where $F(\_,\_)$ denotes mapping spectrum in $\Sp$.
\end{prop}
\begin{proof}
The Segal conjecture \cite{gunawardena1980segal} implies that we have a $C_3$-equivariant $3$-adic equivalence
\[S^0 \overset{\sim}{\to} F({EC_3}_+, S^0).\]
For even $j = 2k$, as $(S^{k\lambda})_{hC_3} \simeq (BC_3)_{j}^\infty$ by definition in \cite{brunermaymccluresteinberger1986hinfinity}, 
\[\begin{split}
        \pi_{i, j}^{C_3}&=[S^{i-j}\wedge S^{k\lambda}, S^0]^{C_3}\\
                  &\cong[S^{i-j}\wedge S^{k\lambda}, F({EC_3}_+, S^0)]^{C_3}\\
                  &\cong[S^{i-j}, F({EC_3}_+\wedge S^{k\lambda}, S^0)]^{C_3}\\
                  &\cong[S^{i-j}, F({EC_3}_+\wedge_{C_3} S^{k\lambda}, S^0)]\\
                  &\cong[S^{i-j}, F((BC_3)_{j}^\infty,S^0)].
    \end{split}\]
For odd $j=2k+1$ the proof is identical to the even case, only that we need to identify $(S^{k\lambda} \wedge S^\Yright)_{hC_3}$ with $(BC_3)_{j}^\infty$:\\
Consider the cofiber sequence
\[S^{k\lambda} \wedge {C_3}_+ \rightarrow S^{k\lambda} \rightarrow S^{k\lambda} \wedge S^\Yright. \]
Taking homotopy orbit yields
\[(S^{k\lambda} \wedge {C_3}_+)_{hC_3}\rightarrow (BC_3)_{2k}^\infty \rightarrow (S^{k\lambda} \wedge S^\Yright)_{hC_3}.\]
Since ${C_3}_+ \wedge S^{k\lambda}$ is free, $(S^{k\lambda} \wedge {C_3}_+)_{hC_3}\simeq S^{2k}$ and it maps to $(BC_3)_{2k}^\infty$ by the inclusion of the bottom cell. Therefore, the cofiber is $(BC_3)_{2k+1}^\infty$.
\end{proof}

The next Proposition reveals the equivariant nature of the cofiber sequence mentioned in \cref{Explanation of methodology}.

\begin{prop}\label{isotropyLES}
    Denote $\pi^{C_3}_{i,j}({EC_3}_+)$ by $L_{i, j}$. There is an isomorphism of long exact sequences
    \[\begin{tikzcd}
\cdots \pi_{i-j+1}^{cl} \arrow[r] \arrow[d, "\cong"]        & {L_{i, j}} \arrow[r] \arrow[d, "\cong"]   & {\pi_{i,j}^{C_3}} \arrow[r, "\Phi^{C_3}"] \arrow[d, "\cong"] & \pi_{i-j}^{cl}\cdots \arrow[d, "\cong"]       \\
\cdots \pi_{i-j}^{cl}(BC_3)_{-\infty}^{\infty} \arrow[r, "M"] & \pi_{i-j}^{cl}(BC_3)_{-j}^{\infty} \arrow[r] & \pi_{i-j-1}^{cl}(BC_3)_{-\infty}^{-j-1} \arrow[r]               & \pi_{i-j-1}^{cl}(BC_3)_{-\infty}^{\infty}\cdots
\end{tikzcd} \]

\end{prop}
\begin{proof}
    Applying $[S^{{i-j}+k\lambda}, \_]^{C_3}$ (resp. $[S^{{i-j}+k\lambda+\Yright}, \_]^{C_3}$) to the isotropy separation sequence
\[{EC_3}_+\to S^0\to \widetilde{EC_3}\]
gives the top long exact sequence
\[\cdots \to \pi_{i-j+1}^{cl} \to L_{i, j}\to \pi_{i,j}^{C_3}\overset{\Phi^{C_3}}{\to}\pi_{i-j}^{cl}\to \cdots .\]
Moreover, for even $j = 2k$ 
\[\begin{split}
    L_{i, j}&=[S^{{i-j}+k\lambda},{EC_3}_+]^{C_3}\\
            &=[S^{i-j},{EC_3}_+\wedge S^{-k\lambda}]^{C_3}\\
            &=[S^{i-j}, (S^{-k\lambda})_{hC_3}]\\
            &=\pi_{i-j}(BC_3)_{-j}^{\infty}
\end{split}\]
The argument for odd weights is similar to that of \cref{S^0borel}. Again we only have to identify the following cofiber sequences
\begin{center}
\begin{tikzcd}
{(S^{-k\lambda} \wedge S^{-\Yright})_{hC_3}} \arrow[r] \arrow[d, "\simeq"] & {(S^{-k\lambda})_{hC_3}} \arrow[d, "\simeq"] \arrow[r] & {(S^{-k\lambda} \wedge {C_3}_+)_{hC_3}} \arrow[d, "\simeq"] \\
{(BC_3)_{-2k-1}^{\infty}} \arrow[r]                     & {(BC_3)_{-2k}^{\infty}} \arrow[r]                     & {S^{-2k}.}                    
\end{tikzcd}
\end{center}

The equivalences $\pi^{C_3}_{i, j}\cong \pi_{i-j-1}(BC_3)_{-\infty}^{-j-1}$ hold since $F((BC_3)_k^\infty,S^0)\simeq \Sigma(BC_3)_{-\infty}^{-k-1}$ by \cref{coreducible}. The last vertical equivalence follows from the Segal conjecture.
\end{proof}

As hinted in the introduction, our computation of the $C_3$-stable stems shall follow these steps below:
\begin{enumerate}
    \item Compute $\pi_*(BC_3)_{-j}^{\infty}$ for various $j$ using the 3-local $\pi_*^{cl}$ as an input for the Atiyah--Hirzebruch spectral sequence.
    \item Compute the Mahowald invariant map $M$ in  \cref{isotropyLES} and solve the extension problems, concluding the 3-primary components of the $C_3$-equivariant stable stems.
    \item Compute the other primary information.
\end{enumerate}

In fact, Step (2) is trivial in a large range of $(i, j)$, as we will prove now:
\begin{prop}\label{splitofSES}
    If $i<j-1$ or $2i>3j $, the long exact sequence in  \cref{isotropyLES} splits. 
\end{prop}
\begin{proof}
When $i<j-1$, $\pi_{i-j+1}^{cl} = 0 = \pi_{i-j}^{cl}$ for degree reasons.

When $2i>3j $, we will show that the composite
\[\widetilde{N}: \pi_{i-j}^{cl} \xrightarrow{N^{C_3}_{e}} \pi_{3(i-j),2(i-j)}^{C_3} \xrightarrow{\cdot (a_{\Yright})^{2i-3j}} \pi^{C_3}_{i,j}\]
gives the desired splitting to $\Phi^{C_3}$.

As $S^{0} \in \Sp^{C_3}$ has trivial $C_3$-action, the additive formula of norms \cite{schwede2019lectures} reduces to
\[N^{C_3}_{e}(x+y) = N^{C_3}_{e}(x)+N^{C_3}_{e}(y)+tr(x^2y+y^2x).\]
Since $a_{\Yright} \circ tr =0$, further composing with $a_{\Yright}$ yields
\[\widetilde{N}(x+y) = \widetilde{N}(x)+\widetilde{N}(y)\]
so that $\widetilde{N}$ is a group homomorphism. Since $\Phi^{C_3}(a_{\Yright})=\id$ and $\Phi^{C_3} \circ N^{C_3}_e \simeq \id$, it follows that $\Phi^{C_3} \circ \widetilde{N} \simeq \id$ which proves the claim.
\end{proof}

\begin{rmk}
    When $j \leq 0$ the splitting also follows from the tom-Dieck splitting \cite{tomdieck1975orbittypen}. Using norm maps, the range now enlarges to $j < \frac{2}{3}i$.
\end{rmk}

\begin{cor}\label{negative stem}
When $i<0$,  \[\pi^{C_3}_{i,j} \cong \pi^{cl}_{i-j}.\]
\end{cor}
\begin{proof}
    When $i<0$, $\pi_{i-j}(BC_3)^\infty_{-j}=0$ by the cellular approximation. 
\end{proof}

\begin{rmk}
    The isomorphisms in \cref{negative stem} hold integrally since the isotropy separation is integral. We also recover this identification in \cref{otherpcomplete}. From now on, we only have to focus on the positive stem part of the $C_3$-equivariant stable stems. 
\end{rmk}

\section{The Atiayh-Hirzebruch spectral sequence}\label{3}
\subsection{The $\pi_*^{cl}$-Atiyah--Hirzebruch spectral sequence}\label{AHSSsetup}

We provide a filtered-spectrum setup of the $\pi_*^{cl}$-Atiyah--Hirzebruch spectral sequence based on \cite{gheorgheisaksenkrausericka2022c, burklundhahnsenger2022galois}. The category of filtered spectra Fil($\Sp$) = Fun($\bbZ^{op}, \Sp$) consists of
\[Y=\{\cdots \to Y(2)\to Y(1)\to Y(0)\to Y(-1)\to \cdots\}, \quad Y(i)\in \Sp.\]
Fil($\Sp$) admits a symmetric monoidal structure with the symmetric unit:
\[S^{0,0}:=\{\cdots \to 0\to 0\to S^0\overset{id}{\to} S^0 \to \cdots\}\]
where the first $S^0$ is at the $0^{th}$ spot. More generally, 
\[S^{0,-n}:=\{\cdots \to 0\to 0\to S^0\overset{id}{\to} S^0 \to \cdots\}\]
where the first $S^0$ is at the $n^{th}$ spot. There is also a canonical map $\lambda:S^{0,-1}\to S^{0, 0}$
\[\begin{tikzcd}
{S^{0, -1} =\{\cdots} \arrow[r] & 0 \arrow[r] \arrow[d] & 0 \arrow[r] \arrow[d] & S^0 \arrow[r] \arrow[d, "id"] & S^0 \arrow[d, "id"] \arrow[r] & \cdots\} \\
{S^{0, 0} =\{\cdots} \arrow[r]  & 0 \arrow[r]           & S^0 \arrow[r]         & S^0 \arrow[r]                 & S^0 \arrow[r]                 & \cdots\}
\end{tikzcd}\]

Given a bounded below CW spectrum $X$, consider the filtered spectrum
\[ X^\star:= \cdots \to 0 \to X^{k}\to  X^{k+1}\to  X^{k+2}\to\cdots\]
where the $n^{th}$ spot is $X^{-n}$, the $(-n)$-skeleton of $X$ and the adjoining maps are inclusion of skeletons. Then the $\pi_*^{cl}$-Atiyah--Hirzebruch spectral sequence of $X$ is the associated spectral sequence of the filtered spectrum $X^\star$, whose signature is
\[E_1^{s, t}=\pi_{t,s}(X^\star/\lambda)\cong\pi_t^{cl}(X^{-s}/X^{-s+1}) \Rightarrow \pi_{t}^{cl}(X),\]
\[d_r:E_r^{s, t}\to E_{r}^{s+r, t-1}.\]
Moreover, there is also the $\lambda$-Bockstein spectral sequence
\[E_1^{s, t, w}= E_1^{s,t} \otimes \bbZ[\lambda] \Rightarrow \pi_{t, w}(X^\star)\]
where elements in $E_1^{s,t}$ are considered to be in tridegree $(s,t,s)$ and $\lambda$ has tridegree $(0,0,-1)$. The differentials 
\[d_r^{\text{Bockstein}}: E_r^{s, t, w}\to E_{r}^{s+r, t-1, w}\]
are rigid in the sense that there is a differential $d_r(x)=y$ in the associated spectral sequence if and only if there is a Bockstein differential $d_r^{\text{Bockstein}}(x)=\lambda^r y$. In particular, the bigraded homotopy group of $X^\star$ can be identified with 
\[\pi_{t, s}X^\star\cong \pi_t^{cl}(X^{-s}).\]

In this paper, $X$ will always be some spectrum with at most one cell in each dimension. Denote elements in the $E_1$-page of the associated spectral sequence by $\alpha[n]$, where $\alpha\in \pi^{cl}_*$, and ``$[n]$'' suggests that it comes from $\pi_*^{cl}(X^n/X^{n-1})$.

There are cofiber sequences
\[\Sigma^{0, -r} X^\star \overset{\lambda^r}{\to} X^\star \overset{\rho_r}{\to} X^\star/\lambda^r \overset{\delta_r}{\to} \Sigma^{1,-r}X^\star,\]
\[\Sigma^{0, -s} X^\star/\lambda^{r-s} \overset{\lambda^s}{\to} X^\star/\lambda^{r} \overset{\rho_{s, r}}{\to} X^\star/\lambda^s \overset{\delta_{s,r}}{\to} \Sigma^{1,-s}X^\star/\lambda^{r-s},\]
where the maps $\delta_r$ and $\delta_{r,s}$ encode the data of the total differentials \cite{linwangxu2025lastkervaire}. Explicitly, we recall the following general patterns of the Atiyah--Hirzebruch differentials from \cite{wangxu2017triviality}.
\begin{theorem}\label{AHdiff}\hspace{5pt}
\begin{enumerate}
    \item \label{AHdiff1}Consider the two-cell complex $\Sigma^{t_2}C\beta$, the cofiber of $\beta\in\pi_{t_1-t_2-1}^{cl}$, whose cells are in dimensions $t_1, t_2$ with $t_2<t_1$. 
    \[\begin{tikzpicture}[scale=0.7]
\draw[thick](-1,0) circle (0.3);
\node at (-1,0) {$t_1$};
\draw[thick](-1,-2) circle (0.3);
\node at (-1,-2) {$t_2$};

\draw[thick] (-1,-0.3)--(-1,-1.7);
\node at (-0.8,-1){$\beta$};
\end{tikzpicture}\]
    Then the only nonzero Atiyah--Hirzebruch differentials are of the form
    \[d_{t_1-t_2}(\alpha[t_1])=\alpha \cdot \beta[t_2],\]
    where $\alpha\in \pi_*^{cl}$ satisfies $\alpha\cdot \beta\neq 0$.

    \item \label{AHdiff2} Let $T$ be a three-cell complex with cells in dimensions $t_1, t_2, t_3$, where $t_3<t_2< t_1$. Suppose we have cofiber sequences
    \[\Sigma^{t_3}C\gamma \overset{i_1}{\hookrightarrow} T\overset{q_1}{\twoheadrightarrow} S^{t_1}\overset{a_1}{\to} \Sigma^{t_3+1}C\gamma \]
    \[S^{t_3}\overset{i_2}{\hookrightarrow} T\overset{q_2}{\twoheadrightarrow} \Sigma^{t_2}C\beta\overset{a_2}{\to} S^{t_3+1}\]
    where $C\beta$ is the cofiber of $\beta\in \pi_{t_1-t_2-1}^{cl}$, $C\gamma$ is the cofiber of $\gamma\in \pi_{t_2-t_3-1}^{cl}$, and $\beta, \gamma$ are non-trivial classes such that $\beta\cdot \gamma=0$. In other words, the cell diagram of $T$ is
    
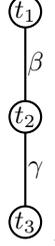
\[\begin{tikzpicture}[scale=0.7]
\draw[thick](-1,0) circle (0.3);
\node at (-1,0) {$t_1$};
\draw[thick](-1,-2) circle (0.3);
\node at (-1,-2) {$t_2$};
\draw[thick](-1,-4) circle (0.3);
\node at (-1,-4) {$t_3$};

\draw[thick] (-1,-0.3)--(-1,-1.7);
\node at (-0.8,-1){$\beta$};
\draw[thick] (-1,-2.3)--(-1,-3.7);
\node at (-0.8,-3){$\gamma$};
\end{tikzpicture}.\]

    Suppose the class $\alpha\in\pi_{t_0}^{cl}$ satisfies the condition: $\alpha\cdot \beta$=0. Then we have an Atiyah--Hirzebruch differential:
    \[d_{t_1-t_3}(\alpha[t_1])\subseteq\langle \alpha, \beta, \gamma\rangle[t_3]\]
    If moreover $\alpha\cdot \pi_{t_1-t_3-1}^{cl}\subseteq \gamma\cdot \pi_{t_0+t_1-t_2}^{cl}$ in $\pi_{t_0+t_1-t_3-1}^{cl}$, then the above inclusion is an equality as $\langle \alpha, \beta, \gamma\rangle[t_3]$ has zero indeterminacy in $E_{t_1-t_3}$.

    \item \label{AHdiff3} Similarly, let $T'$ be a 4-cell complex with cells in dimensions $t_1, t_2, t_3, t_4$, where $t_4<t_3\leq t_2< t_1$. Suppose there are cofiber sequences
    \[\Sigma^{t_4}C\gamma \overset{i_1}{\hookrightarrow} T'\overset{q_1}{\twoheadrightarrow} \Sigma^{t_2} C\delta \overset{a_1}{\to} \Sigma^{t_4+1}C\gamma \]
    \[\Sigma^{t_4}C\epsilon \overset{i_2}{\hookrightarrow} T'\overset{q_2}{\twoheadrightarrow} \Sigma^{t_3}C\beta\overset{a_2}{\to} \Sigma^{t_4+1}C\epsilon \]
    \[\Sigma^{t_4}C\gamma \vee_{S^0} \Sigma^{t_4}C\epsilon \overset{i_3}{\hookrightarrow} T'\overset{q_3}{\twoheadrightarrow} S^{t_1} \overset{a_3}{\to} \Sigma^{t_4+1}C\gamma \vee_{S^0} \Sigma^{t_4+1}C\epsilon \]
    where $\beta, \gamma, \delta, \epsilon$ are non-trivial classes such that $\beta \cdot \gamma+\delta\cdot \epsilon=0$, and the composite $S^{t_1} \overset{a_3}{\to} \Sigma^{t_4+1}C\gamma \vee_{S^0} \Sigma^{t_4+1}C\epsilon \twoheadrightarrow S^{t_2+1}\vee S^{t_3+1}$ is $(\delta, \beta)$. In other words, the cell diagram of $T'$ is
    
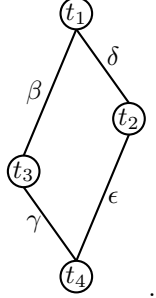
\[\begin{tikzpicture}[scale=0.7]
\draw[thick](-1,0) circle (0.3);
\node at (-1,0) {$t_1$};
\draw[thick](0,-2) circle (0.3);
\node at (0,-2) {$t_2$};
\draw[thick](-2,-3) circle (0.3);
\node at (-2,-3) {$t_3$};
\draw[thick](-1,-5) circle (0.3);
\node at (-1,-5) {$t_4$};

\draw[thick] (-1,-0.3)--(0,-1.7);
\node at (-0.3,-0.8){$\delta$};
\draw[thick] (-1,-0.3)--(-2,-2.7);
\node at (-1.8,-1.5){$\beta$};
\draw[thick] (-1,-4.7)--(0,-2.3);
\node at (-0.3,-3.5){$\epsilon$};
\draw[thick] (-1,-4.7)--(-2,-3.3);
\node at (-1.8,-4.0){$\gamma$};
\end{tikzpicture}.\]
    Suppose the class $\alpha\in\pi_{t_0}^{cl}$ satisfies the condition: $\alpha\cdot \beta=\alpha\cdot \delta=0$. Then we have an Atiyah--Hirzebruch differential:
    \[d_{t_1-t_4}(\alpha[t_1])\subseteq\langle \alpha, \begin{pmatrix}
        \beta, \delta
    \end{pmatrix}, \begin{pmatrix}
        \gamma \\
        \epsilon
    \end{pmatrix}\rangle[t_3]\]
    If moreover $\alpha\cdot \pi_{t_1-t_4-1}^{cl}\subseteq \gamma\cdot \pi_{t_0+t_1-t_3}^{cl}+\epsilon\cdot \pi_{t_0+t_1-t_2}^{cl}$ in $\pi_{t_0+t_1-t_4-1}^{cl}$, then the above inclusion is an equality as $\langle \alpha, \begin{pmatrix}
        \beta, \delta
    \end{pmatrix}, \begin{pmatrix}
        \gamma \\
        \epsilon
    \end{pmatrix}\rangle[t_3]$ has zero indeterminacy in $E_{t_1-t_4}$.
\end{enumerate}
\end{theorem}
\begin{proof}
    See \cite[Lemma 6.1]{wangxu2017triviality} and slightly generalize their proofs for part (c).
\end{proof}

\begin{rmk}
    For more complicated cell complexes, there are generalizations of \cref{AHdiff} with subtle conditions on the indeterminacies. See \cite{wangxu2017triviality} for a detailed discussion.
\end{rmk}

In general, suppose $f : X \to X'$ is a map between $H\bbF_3$-subquotients of $BC_3$, which is a composite of inclusion and quotient maps. Suppose further that there exists an element $\alpha[n]$ which is a generator of both $E_1^{s,t}(X)$ and $E_1^{s, t}(X')$. From the naturality of the Atiyah--Hirzebruch spectral sequence, we must have that, with the right choices, $\alpha[n]$ in $E_1(X)$ maps to $\alpha[n]$ in $E_1(X)'$. The naturality also ensures that one can pull back or push forward differentials in appropriate situations.

We also describe a general way to compute hidden extensions in the $E_\infty$-page of Atiyah--Hirzebruch spectral sequences, essentially the same ideas as those of \cite[Prop. 3.1.6]{isaksen2019stable}.
\begin{prop}\label{generalhiddenextension}
\hspace{5pt}
\begin{enumerate}
    \item Let $\alpha, \beta, \gamma \in \pi_*^{cl}$ with $\alpha\beta=0, \beta\gamma=0$. Then in the Atiayh--Hirzebruch spectral sequence of $C\alpha$, $\beta[|\alpha|+1]$ is a permanent cycle and there is a hidden extension in the $E_\infty$-page
    \[\gamma \cdot \beta[|\alpha|+1] \in \langle \alpha, \beta, \gamma\rangle[0].\]
    \item Generalizing part (1), let $\alpha_1,\alpha_2, \beta, \gamma \in \pi_*^{cl}$ with $\alpha_1\beta=0, \alpha_2\beta=0, \beta\gamma=0$. Consider the cofiber sequence
    \[S^0\overset{\alpha_1\vee\alpha_2}{\to} S^{-|\alpha_1|}\vee S^{-|\alpha_2|}\hookrightarrow Q \twoheadrightarrow S^1.\]
    Then in the Atiayh--Hirzebruch spectral sequence of $Q$, $\beta[1]$ is a permanent cycle and there is a hidden extension in the $E_\infty$-page
    \[\gamma \cdot \beta[1] \in \langle \alpha_1, \beta, \gamma\rangle[-|\alpha_1|]+\langle \alpha_2, \beta, \gamma\rangle[-|\alpha_2|].\]
\end{enumerate}
\end{prop}
\begin{proof}
\hspace{5pt}
\begin{enumerate}
    \item Since $\beta\cdot\alpha=0$, by \cref{AHdiff}, $\beta[|\alpha|+1]$ must be a permanent cycle. It is not hit by any differential for degree reasons. As an element in $\pi_*(C\alpha)$, $\beta[|\alpha|+1]$ can be represented by $\tilde\beta$ as a lift of $\beta$
    \[\begin{tikzcd}
S^{|\beta|+|\alpha|+1} \arrow[d, "\tilde\beta"'] \arrow[rd, "\beta"] &                \\
C\alpha \arrow[r, two heads]                                         & S^{|\alpha|+1}.
\end{tikzcd}\]
Consider the following commutative diagram, 
\[\begin{tikzcd}
                                                  & S^{|\gamma|+|\beta|+|\alpha|+1} \arrow[rd, "\gamma"] \arrow[d, "\tilde\gamma", dashed] &                                                                     &                                    \\
S^{|\alpha|} \arrow[r, hook] \arrow[d, "\simeq"'] & \Sigma^{|\alpha|}C\beta \arrow[r, "p"', two heads] \arrow[d, "\tilde\alpha"]           & S^{|\beta|+|\alpha|+1} \arrow[d, "\tilde\beta"'] \arrow[r, "\beta"] & S^{|\alpha|+1} \arrow[d, "\simeq"] \\
S^{|\alpha|} \arrow[r, "\alpha"]                  & S^0 \arrow[r, "i", hook]                                                               & C\alpha \arrow[r, two heads]                                        & S^{|\alpha|+1}                    
\end{tikzcd}\]
The two rows are cofiber sequences, so based on the property of triangulated categories, the middle square commutes. 

Now, $\gamma \cdot \beta[|\alpha|+1]$ can be represented by the composite $\tilde\beta\circ\gamma$. Since $\gamma\cdot \beta=0$, we have a lift $\tilde\gamma:S^{|\gamma|+|\beta|+|\alpha|+1} \to \Sigma^{|\alpha|}C\beta$. Therefore, $\tilde\beta\circ\gamma = i \circ \tilde\alpha \circ\tilde\gamma$. By definition, $\tilde\alpha \circ\tilde\gamma \in \langle \alpha, \beta, \gamma\rangle$, so claim follows.
    \item Replace the commutative diagram by
    \[\begin{tikzcd}
                                           & S^{|\gamma|+|\beta|+1} \arrow[rd, "\gamma"] \arrow[d, "\tilde\gamma", dashed] &                                                            &                           \\
S^{0} \arrow[r, hook] \arrow[d, "\simeq"'] & C\beta \arrow[r, "p"', two heads] \arrow[d, "\tilde\alpha"]                   & S^{|\beta|+1} \arrow[d, "\tilde\beta"'] \arrow[r, "\beta"] & S^{1} \arrow[d, "\simeq"] \\
S^{0} \arrow[r, "\alpha_1\vee\alpha_2"]    & S^{-|\alpha_1|}\vee S^{-|\alpha_2|} \arrow[r, "i", hook]                      & Q \arrow[r, two heads]                                     & S^{1}                    
\end{tikzcd}\]
    where $\tilde\alpha$ and $\tilde\beta$ exists since $\beta\cdot\alpha_1=\beta\cdot\alpha_2=0$. The rest
    proofs are the same as part (a).
\end{enumerate}
\end{proof}

\begin{rmk}
    Similar formula as in \cref{generalhiddenextension} also applies in the synthetic $\lambda$-Bockstein spectral sequence. One only needs to modify by some powers of $\lambda$ to equate the bi-degree.
\end{rmk}

\subsection{Atiyah--Hirzebruch differentials based on the $\calA(1)$-module structure} The full charts containing the data of the Atiyah--Hirzebruch spectral sequence of $(BC_3)_{-j}^\infty$ for various $j$ can be found at \cite{houzhang2025AHSS}.

In this subsection, we prove the Atiyah--Hirzebruch differentials that only use information of the $\calA(1)$-module structure of $(BC_3)_{-j}^\infty$.
\begin{prop}\label{d1-diff}
    For $(BC_3)_{j}^\infty$, if $k=2n>j$, there are differentials
    \begin{align*}
             d_1(1[k])&=3[k-1],\\
             d_1 (\bar{\alpha}_3[k])&=\alpha_3[k-1],\\
             d_1 (\bar{\alpha}_6[k])&=\alpha_6[k-1],\\
            d_1 (\bar{\alpha}_9[k])&=3\bar{\alpha}_9[k-1],\\
             d_1 (3\bar{\alpha}_9[k])&=\alpha_9[k-1].
    \end{align*}
\end{prop}
\begin{proof}
    By \cref{3andalpha1subquotient}, $\Sigma^{k-1}C3$ is a subquotient if and only if $k$ is even. The claimed differentials then follows from \cref{AHdiff} and naturality.
\end{proof}

\begin{prop}\label{d4-diff1}
    For $(BC_3)_{j}^\infty$, if $k \equiv 1, 3, 7 \ \mathrm{or}\ 9 \pmod {12}$ and $k-4\geq j$, then there are differentials
    \[
             d_4(1[k])=\alpha_1[k-4]
    \]
\end{prop}
\begin{proof}
    By \cref{d1-diff}, $1[k]$'s survive to the $E_4$-page only if $k$ is odd. In this case, by \cref{cohofBC3}, there are non-trivial $P^1$-action on $H^{k-4}((BC_3)_{j}^\infty; \bbF_3)$ if and only if $ k \equiv 1, 3, 7$ or $9 \pmod {12}$. Then by \cref{3andalpha1subquotient}, $\Sigma^{k-4}C\alpha_1$ is a subquotient of $(BC_3)_{j}^\infty$ if and only if $ k \equiv 1, 3, 7$ or $9 \pmod {12}$. The claim then follows from naturality and \cref{AHdiff}. 
\end{proof}

\begin{prop}\label{d4-diff2}
    For $(BC_3)_{j}^\infty$, if $k \equiv 0, 1, 2, 3, 6, 7, 8 \ \mathrm{or}\ 9 \pmod {12}$ and $k-4\geq j$, then there are differentials
    \begin{align*}
        d_4(\beta_1[k])&=\beta_1\alpha_1[k-4],\\
        d_4(\beta_1^2[k])&=\beta_1^2\alpha_1[k-4],\\
        d_4(\beta_2[k])&=\beta_2\alpha_1[k-4],\\
        d_4(x_{37}[k])&=-\beta_1^4[k-4].
    \end{align*}
\end{prop}
\begin{proof}
    All the classes listed above survive to the $E_4$-page. Similar to \cref{d4-diff1}, since $\Sigma^{k-4}C\alpha_1$ is a subquotient of $(BC_3)_{j}^\infty$ if and only if $k \equiv 0,1, 2, 3, 6, 7, 8$ or $9 \pmod {12}$, the claim follows from naturality and \cref{AHdiff}.
\end{proof}

\begin{prop}\label{d5-diff1}
    For $(BC_3)_{j}^\infty$, 
    \begin{enumerate}
        \item there are differentials
    \[
        d_5(\alpha_{3l+1}[k])=\pm\alpha_{3l+2}[k-5]
    \]
    if and only if $k \equiv 0, 4, 6\ \mathrm{or}\ 10 \pmod {12}$ and $k-5\geq j$;
    \item there are differentials
    \[
        d_5(\alpha_{3l+2}[k])=\pm\bar{\alpha}_{3l+3}[k-5]
    \]
    if and only if $k \equiv 0, 2, 6\ \mathrm{or}\ 8 \pmod {12}$ and $k-5\geq j$;
    \item there are differentials
    \[
        d_5(\alpha_{3l}[k])=\pm\alpha_{3l+1}[k-5]
    \]
    if and only if $k \equiv  2, 4, 8\ \mathrm{or}\ 10 \pmod {12}$ and $k-5\geq j$;
    \end{enumerate}
\end{prop}
\begin{proof} For degree reasons, the classes $\alpha_{n}[k]$ for $k$ odd survives to the $E_5$-page and does not support $d_5$ differentials. It is enough to consider the complex $(BC_3)^k_{k-5}$ for various $k$'s. 
\begin{enumerate}
    \item By \cref{subquotientofBSigma3kk-5}, $\Sigma^kT_1$ is a quotient complex of $(BC_3)^k_{k-5}$ for $k \equiv 0\ \mathrm{or}\ 6 \pmod {12}$. By \cref{AHdiff} and \cref{alpha3k+1toda}, there are differentials
    \[d_5(\alpha_{3l+1}[k])=\langle \alpha_{3l+1}, \alpha_1, 3\rangle[k-5]=\alpha_{3l+2}[k-5]\]
    in $\Sigma^kT_1$. Because the relevant classes all survive to the $E_5$-page, naturality pulls back the differential to $(BC_3)^k_{k-5}$.

    Similarly, by \cref{subquotientofBSigma3kk-5}, $\Sigma^kT_2$ is a subcomplex of $(BC_3)^k_{k-5}$ for $k \equiv 4 \ \mathrm{or}\ 10 \pmod {12}$. \cref{AHdiff} and \cref{alpha3k+1toda} then implies the differential
    \[d_5(\alpha_{3l+1}[k])=\langle \alpha_{3l+1},  3, \alpha_1\rangle[k-5]=-\alpha_{3l+2}[k-5]\]
    in $\Sigma^kT_2$. Because the relevant classes all survive to the $E_5$-page, naturality pushed forward the differential to $(BC_3)^k_{k-5}$.

    For $k\equiv2 \ \mathrm{or}\  8 \pmod {12}$, $\Sigma^k T_3$ is a summand of $(BC_3)^k_{k-5}$. According to \cref{AHdiff},
    \[d_5(\alpha_{3l+1}[k])=\langle \alpha_{3l+1},  (3, \alpha_1), \begin{pmatrix}
        \alpha_1\\
        3
    \end{pmatrix}\rangle[k-5]=0.\]
    Part (1) thus follows.
    \item We first notice that $\alpha_{3l+3}[k-5]$ does not survive to $E_5$ page, so there are no indeterminacies in the Toda brackets below.
    Since $\Sigma^kT_1$ is a quotient complex of $(BC_3)^k_{k-5}$ for $k \equiv 0\ \mathrm{or}\ 6 \pmod {12}$, by \cref{AHdiff},  \cref{alpha3k+2toda} and naturality,
    \[d_5(\alpha_{3l+2}[k])\subseteq\langle \alpha_{3l+2}, \alpha_1, 3\rangle[k-5]=\bar{\alpha}_{3l+3}[k-5].\]
    
    Similarly, since $\Sigma^k T_3$ is a summand of $(BC_3)^k_{k-5}$ for $k\equiv2 \ \mathrm{or}\  8 \pmod {12}$,
    \[d_5(\alpha_{3l+2]}[k])=\langle \alpha_{3l+2},  (3, \alpha_1), \begin{pmatrix}
        \alpha_1\\
        3
    \end{pmatrix}\rangle[k-5]=\bar{\alpha}_{3l+3}[k-5].\]
    
    On the other hand, since $\Sigma^kT_2$ is a subcomplex of $(BC_3)^k_{k-5}$ for $k \equiv 4 \ \mathrm{or}\ 10 \pmod {12}$, we have in $\Sigma^kT_2$,
    \[d_5(\alpha_{3l+2}[k])=\langle \alpha_{3l+2},  3, \alpha_1\rangle[k-5]=\alpha_{3l+3}[k-5]=0.\]
    Pushing forward to $(BC_3)^k_{k-5}$, we have $\alpha_{3l+2}[k]$ cannot support $d_{\leq 5}$ differentials as well. This concludes part (2).
    \item Since $\Sigma^kT_2$ is a subcomplex of $(BC_3)^k_{k-5}$ for $k \equiv 4 \ \mathrm{or}\ 10 \pmod {12}$, by \cref{AHdiff},  \cref{alpha3k+3toda} and naturality, 
    \[d_5(\alpha_{3l}[k])=\langle \alpha_{3l},  3, \alpha_1\rangle[k-5]=-\alpha_{3l+1}[k-5].\]
    
    For $k\equiv2 \ \mathrm{or}\  8 \pmod {12}$, $\Sigma^k T_3$ is a summand of $(BC_3)^k_{k-5}$, according to \cref{AHdiff} and \cref{alpha3k+3toda},
    \[d_5(\alpha_{3l}[k])=\langle \alpha_{3l},  (3, \alpha_1), \begin{pmatrix}
        \alpha_1\\
        3
    \end{pmatrix}\rangle[k-5]=\alpha_{3l+1}[k-5].\]
    
    For $k \equiv 0\ \mathrm{or}\ 6 \pmod {12}$, $\Sigma^kT_1$ is a quotient complex of $(BC_3)^k_{k-5}$. In $\Sigma^kT_1$,
    \[d_5(\alpha_{3l}[k])=\langle \alpha_{3l}, \alpha_1, 3\rangle[k-5]=0\neq \alpha_{3l+1}[k-5].\]
    Pulling back to $(BC_3)^k_{k-5}$, we have that $\alpha_{3l+1}[k-5]$ cannot be hit by $d_{\leq 5}$-differentials. Thus, for degree reasons,
    $d_5(\alpha_{3l}[k])=0$ in $(BC_3)^k_{k-5}$, which concludes part (3) 
\end{enumerate}
\end{proof}

\begin{prop}\label{d8-diff1}
\hspace{5pt}

\begin{enumerate}
    \item For $(BC_3)_{j}^\infty$, if $k \equiv 1\ \mathrm{or}\ 7 \pmod {12}$ and $k-8\geq j$, then there are differentials
    \[
             d_8(\alpha_1[k])=\beta_1[k-8].
    \]
    \item For $(BC_3)_{j}^\infty$, if $k \equiv 0,1, 6 \ \mathrm{or}\  7 \pmod {12}$ and $k-8\geq j$, then there are differentials
    \begin{align*}
         d_8(\beta_1\alpha_1[k])&=\beta_1^2[k-8],\\
        d_8(\beta_1^2\alpha_1[k])&=\beta_1^3[k-8],\\
         d_8(\beta_1^3[k])&=x_{37}[k-8],\\
    \end{align*}
\end{enumerate}
\end{prop}
\begin{proof}
    By \cref{d4-diff1}, $\alpha_1[k]$'s for $k$ odd survive to the $E_8$-page only in the degrees listed in (1). On the other hand, by \cref{d4-diff2}, $\beta_1\alpha_1[k]$'s, $\beta_1^2\alpha_1[k]$'s and $\beta_1^3[k]$'s survive to the $E_8$-page in the degrees listed in (2).
    
    Within $(B\Sigma_3)^{k}_{k-8}$ for $k \equiv 7 \pmod {12}$, the $(k-3)$-cell and the $(k-7)$-cell are top cells for degree reasons, so we may consider the fiber
    \[F'\hookrightarrow (B\Sigma_3)^{k}_{k-8} \twoheadrightarrow S^{k-3}\vee S^{k-7}.\]
    $F'$ consists of cells in dimension $k, k-4$ and $k-8$, and there are non-trivial $P^1$-actions on $H^{k-8}(F';\bbF_3)$ and $H^{k-4}(F';\bbF_3)$. Therefore, by \cref{AHdiff} and \cref{todabetafami}, there are differentials in $F'$
    \begin{align*}
    d_8(\alpha_1[k])&=\langle \alpha_1, \alpha_1, \alpha_1\rangle [k-8]=\beta_1[k-8],\\
        d_8(\beta_1\alpha_1[k])&=\langle \beta_1\alpha_1, \alpha_1, \alpha_1\rangle [k-8]=\beta_1^2[k-8],\\
        d_8(\beta_1^2\alpha_1[k])&=\langle \beta_1^2\alpha_1, \alpha_1, \alpha_1\rangle [k-8]=\beta_1^3[k-8],\\
        d_8(\beta_1^3[k])&=\langle \beta_1^3, \alpha_1, \alpha_1\rangle [k-8]=x_{37 }[k-8].
    \end{align*}
    By naturality, the corresponding differentials happen in $(B\Sigma_3)^{k}_{k-8}$ for $k \equiv 7 \pmod {12}$.
    
    Simiarly for $k \equiv 0 \pmod {12}$, the $(k-1)$-cell and the $(k-5)$-cell are bottom cells of $(B\Sigma_3)^{k}_{k-8}$ for degree reasons, so we may consider the cofiber
    \[S^{k-1}\vee S^{k-5}\hookrightarrow (B\Sigma_3)^{k}_{k-8} \twoheadrightarrow C'.\]
    $C'$ also consists of cells in dimension $k, k-4$ and $k-8$, and there are non-trivial $P^1$-actions on $H^{k-8}(C';\bbF_3)$ and $H^{k-4}(C';\bbF_3)$. Therefore, by \cref{AHdiff} and \cref{todabetafami}, there are differentials in $C'$  
    \begin{align*}
        d_8(\beta_1\alpha_1[k])&=\langle \beta_1\alpha_1, \alpha_1, \alpha_1\rangle [k-8]=\beta_1^2[k-8],\\
        d_8(\beta_1^2\alpha_1[k])&=\langle \beta_1^2\alpha_1, \alpha_1, \alpha_1\rangle [k-8]=\beta_1^3[k-8],\\
        d_8(\beta_1^3[k])&=\langle \beta_1^3, \alpha_1, \alpha_1\rangle [k-8]=x_{37 }[k-8].
    \end{align*}
    By naturality, the corresponding differentials happen in $(B\Sigma_3)^{k}_{k-8}$ for $k \equiv 0 \pmod {12}$.  
    
    The same cell structures of $X^k_{k-8}$ for $k \equiv 1\ \text{or}\ 6  \pmod {12}$ respectively provide the same patterns of differentials.
\end{proof}

\subsection{Differentials based on longer attaching maps}
\begin{prop}\label{alpha2diff}
    For $(BC_3)_j^\infty$, there are $d_8$-differentials 
    \[d_8(1[k])=\alpha_2[k-8]\] if and only if 
    $k\equiv 5, 11, 23\ \mathrm{or}\ 29 \pmod {36}$ and $k-8\geq j$.
\end{prop}
\begin{proof}
    By \cref{d1-diff} and \cref{d4-diff1}, $1[k]$'s survive to the $E_8$-page if and only if $k\equiv -1 \pmod{6}$. By \cref{alpha2attachingmap}, there are $\Sigma^{k-8}C\alpha_2$ is a subcomplex of $(BC_3)^k_{k-8}$ if and only if $k\equiv 5, 11, 23\ \mathrm{or}\ 29 \pmod {36}$. The claimed differentials then follows from \cref{AHdiff} and naturality. 
\end{proof}

\begin{prop}\label{alpha3diff}
    For $(BC_3)_{j}^{\infty}$, there are $d_{12}$-differentials 
    \[d_{12}(1[k])=\bar{\alpha}_3[k-12]\] if and only if 
    $k= 17, 35, 71\ \text{or}\ 89 \pmod{108}$ and $k-12\geq j$.
\end{prop}
\begin{proof}
    By \cref{d1-diff}, \cref{d4-diff1} and \cref{alpha2diff}, $1[k]$'s in the Atiyah--Hirzebruch spectral sequence of $(BC_3)_{j}^{\infty}$ survive to the $E_{12}$-page if and only if $k\equiv -1 \pmod {18}$. By \cref{alpha3attachingmap}, there are $\Sigma^{k-12}C\bar{\alpha}_3$ is a subcomplex if and only if $k\equiv 17, 35, 71 \ \mathrm{or}\ 89 \pmod{108}$. The claimed differentials then follows from \cref{AHdiff} and naturality. 
\end{proof}

\begin{prop}\label{leibnizrule}
\hspace{5pt}

\begin{enumerate}
    \item   For $(BC_3)_j^\infty$, for $k\equiv 6, 12, 24\ \mathrm{or}\ 30 \pmod {36}$, there are $d_9$-differentials
    \[d_9(\alpha_n[k-4n])=\begin{cases}
        \alpha_{n+2}[k-4n-9] & \mathrm{if}\ n\equiv 0, 2 \pmod 3\\
        \bar{\alpha}_{n+2}[k-4n-9] & \mathrm{if}\ n\equiv 1 \pmod 3
    \end{cases}\]
    as long as $k-4n-9>j$.
    \item     For $(BC_3)_{-16}^{60}$, for $k=18$ or 36, there are $d_{13}$-differentials
    \[d_{13}(\alpha_n[k-4n])=\begin{cases}
        \alpha_{n+3}[k-4n-13] & \mathrm{if}\ n\equiv 1, 2 \pmod 3\\
        \bar{\alpha}_{n+3}[k-4n-13] & \mathrm{if}\ n\equiv 0 \pmod 3
    \end{cases}\]
    as long as $k-4n-13>-16$.
\end{enumerate}
  
\end{prop}
\begin{proof} \hspace{5pt}

\begin{enumerate}
    \item We will consider the the filtered spectrum $((BC_3)^\infty_j)^\star$ and the $\lambda$-Bockstein spectral sequences defined at the beginning this section. Recall that the canonical maps $\delta_r: ((BC_3)^\infty_j)^\star/\lambda^r \to \Sigma^{1, -r}((BC_3)^\infty_j)^\star$ encodes the data of the total differentials. 

   For $k\equiv 0 \pmod 6$, $\Sigma^{k-1}C3$ and $\Sigma^{k-4}C\alpha_1$ are subquotients of $(BC_3)_{k-8}^k$. Thus, we have
    \[\delta_1(1[k])=3\cdot 1[k-1]+\lambda^3\alpha_1[k-4] \quad \pmod {\lambda^7}. \]
    By exactness, this implies in $\pi_{*, *}((BC_3)^\infty_j)^\star/\lambda^6$, 
    \[\lambda\cdot (3\cdot 1[k-1]+\lambda^3 \alpha_1[k-4])=0.\]
    Moreover, $\alpha_2[k-9], \bar{\alpha}_3[k-13] \in \pi_{*, *}((BC_3)^\infty_j)^\star/\lambda$ lift to $\pi_{*, *}((BC_3)^\infty_j)^\star/\lambda^6$ as they survive to the $E_6$-page. Since there is $\Sigma^{k-13}C\alpha_1$ as a subquotient and \cref{alpha2toda} gives $\langle 3, \alpha_2, \alpha_1\rangle=\bar{\alpha}_3$ modulo indeterminacy, by \cref{generalhiddenextension} there is a hidden extension
    \[3\cdot \alpha_2[k-9]= \lambda^4\bar{\alpha}_3[k-13] \]
    in $\pi_{*, *}((BC_3)^\infty_j)^\star/\lambda^6$. There is a potential crossing of the above extension $3\cdot \lambda^3\beta_1[k-12]=\lambda^4\bar{\alpha}_3[k-13]$; however, this cannot happen since $\langle 3, \beta_1, 3\rangle =0$ and does not contain $\bar{\alpha}_3$. Thus, any such lifts in $\pi_{*, *}((BC_3)^\infty_j)^\star/\lambda^6$ satisfy the above relation.
    
    Now consider $\delta_6: ((BC_3)^\infty_j)^\star/\lambda^6 \to \Sigma^{1, -6}((BC_3)^\infty_j)^\star$. \cref{alpha2diff} implies that for $k\equiv 6, 12, 24\ \mathrm{or}\ 30 \pmod {36}$,
    \[\delta_6(1[k-1])=\lambda^3\alpha_2[k-9]+ \text{lower cellular filtration terms}.
    \]
    Inspection of $E_6$-page gives that lower cellular filtration terms are at least $\lambda^7$-multiples. 
    Since $\delta_6$ is $\bbZ[\lambda]$-linear, multiplying by $3\lambda$, we have
    \[\begin{split}
        \delta_6(-\lambda^4\alpha_1[k-4])&=3\lambda \cdot (\lambda^3\alpha_2[k-9]+ \text{lower cellular filtration terms})\\
        &=\lambda^8\bar{\alpha}_3[k-13] + 3\lambda\cdot \text{lower cellular filtration terms},
    \end{split}\]
    where inspection of $E_6$-page gives that $3\lambda\cdot \text{(lower cellular filtration terms)}$ are at least $\lambda^{12}$-multiples. Since $\bar{\alpha}_3[k-13]$ survives to $E_9$-page, this corresponds to the Atiyah--Hirzebruch differential \[d_9(\alpha_1[k-4])=\bar{\alpha}_3[k-13].\]

    Similarly, all the other claimed $d_9$-differentials will follow from the $\bbZ[\lambda]$-linearity of $\delta_6$ once we prove the following hidden extensions in $\pi_{*, *}((BC_3)^\infty_j)^\star/\lambda^6$ for $k\equiv 6, 12, 24\ \mathrm{or}\ 30 \pmod {36}$.
    
    \begin{align*}
    \begin{cases}
        3\cdot \alpha_n[k-4n] &=\lambda^4 \alpha_{n+1}[k-4n-4]\\
        3\cdot \alpha_n[k-4n-1]&=\lambda^4\alpha_{n+1}[k-4n-5] 
    \end{cases} \quad \text{if}\ \ n\equiv 1 \pmod 3, \\
    \begin{cases}
        3\cdot \alpha_n[k-4n] &=-\lambda^4 \alpha_{n+1}[k-4n-4]\\
        3\cdot \alpha_n[k-4n-1]&=\lambda^4\bar{\alpha}_{n+1}[k-4n-5] 
    \end{cases} \quad \text{if}\ \ n\equiv 2 \pmod 3, \\
    \begin{cases}
        3\cdot \alpha_n[k-4n] &=\lambda^4 \alpha_{n+1}[k-4n-4],\\
        3\lambda\cdot \bar{\alpha}_n[k-4n-1]&=-\lambda^5\alpha_{n+1}[k-4n-5]
    \end{cases}  \quad \text{if}\ \ n\equiv 0 \pmod 3.\\
    \end{align*}

    \begin{itemize}
        \item When $n\equiv 1 \pmod 3$. Since there are subquotients $\Sigma^{k-4n-4}C\alpha_1$ and $\Sigma^{k-4n-5}C\alpha_1$, and \cref{alpha3k+1toda} gives $\langle 3, \alpha_n, \alpha_1\rangle =\alpha_{n+1}$, by \cref{generalhiddenextension} we have hidden extensions
    \begin{align*}
        3\cdot \alpha_n[k-4n]&=\lambda^4 \alpha_{n+1}[k-4n-4],\\
        3\cdot \alpha_n[k-4n-1]&=\lambda^4\alpha_{n+1}[k-4n-5]. \\
    \end{align*}
    \item When $n\equiv 2 \pmod 3$. Since $(B\Sigma_3)^{k-4n}_{k-4n-4} \simeq S^{k-4n-4}\vee\Sigma^{k-4n-1}C3$, we may choose a lift $\alpha_n[k-4n] \in \pi_{*. *}((BC_3)^\infty_j)^\star/\lambda^5$ such that $3\cdot \alpha_n[k-4n] =0$ in $\pi_{*. *}((BC_3)^\infty_j)^\star/\lambda^5$. Since $\Sigma^{k-4n-1}C3$ and $\Sigma^{k-4n-5}C\alpha_1$ are subquotients, and \cref{alpha3k+2toda} gives $\langle \alpha_n, 3, \alpha_1\rangle=\alpha_{n+1}$, we have that
    \[\delta_5(\alpha_n[k-4n])=\langle \alpha_n, 3, \alpha_1\rangle[k-4n-5]=\alpha_{n+1}[k-4n-5] \pmod{\lambda^3}.\]
    On the other hand, notice that $\lambda^4\bar{\alpha}_{n+1}[k-4n-4]$ survives in $\pi_{*. *}((BC_3)^\infty_j)^\star/\lambda^5$. Since there is a subquotient $\Sigma^{k-4n-5}C3$, we have
    \[\delta_5(\lambda^4\bar{\alpha}_{n+1}[k-4n-4])=\alpha_{n+1}[k-4n-5]\pmod{\lambda^4}.\]
    Consider the cofiber sequence
    \[\Sigma^{0, -5} X^\star/\lambda \overset{\lambda^5}{\to} X^\star/\lambda^{6} \overset{\rho_{5, 6}}{\to} X^\star/\lambda^5 \overset{\delta_{5,6}}{\to} \Sigma^{1,-s}X^\star/\lambda.\]
    Project to $((BC_3)^\infty_j)^\star/\lambda$, we have
    \[\begin{split}
        \delta_{5, 6}(\alpha_n[k-4n]-\lambda^4\bar{\alpha}_{n+1}[k-4n-4])&=\alpha_{n+1}[k-4n-5]-\alpha_{n+1}[k-4n-5]\\
        &=0.
    \end{split}\]
    By exactness, there exists an element $x\in \pi_{*. *}((BC_3)^\infty_j)^\star/\lambda^6$ such that \[\rho_{5,6}(x)=\alpha_n[k-4n]-\lambda^4\bar{\alpha}_{n+1}[k-4n-4].\]
    For degree reason, the only possible choice is $x=\alpha_n[k-4n] \in \pi_{*. *}((BC_3)^\infty_j)^\star/\lambda^6$. This implies
    \[\begin{split}
        \rho_{5, 6}(3\cdot x)&=3\cdot (\alpha_n[k-4n]-\lambda^4\bar{\alpha}_{n+1}[k-4n-4])\\
        &=-3\cdot \lambda^4\bar{\alpha}_{n+1}[k-4n-4]\\
        &=-\lambda^4\alpha_{n+1}[k-4n-4]\\
        &=\rho_{5, 6}(-\lambda^4\alpha_{n+1}[k-4n-4]).
    \end{split}\]
    Therefore, by exactness, $3\cdot x+\lambda^4\alpha_{n+1}[k-4n-4]$ is an $\lambda^5$ multiple. Inspection of the $E_1$-page implies there is no such $\lambda^5$-multiples other than 0 at those degrees, so we have
    \[3\cdot \alpha_n[k-4n]=-\lambda^4\alpha_{n+1}[k-4n-4]\] in
    $\pi_{*. *}((BC_3)^\infty_j)^\star/\lambda^6$.

    Moreover, since there is a subquotient $\Sigma^{k-4n-5}C\alpha_1$ and \cref{alpha3k+2toda} gives $\langle 3, \alpha_n, \alpha_1\rangle =\bar{\alpha}_{n+1}$, by \cref{generalhiddenextension} we have hidden extensions
    \[3\cdot \alpha_n[k-4n-1]=\lambda^4\bar{\alpha}_{n+1}[k-4n-5]. \]
    \item When $n\equiv 0 \pmod 3$. Since there is a subquotient $\Sigma^{k-4n-4}C\alpha_1$ and \cref{alpha3k+3toda} gives $\langle 3, \alpha_n, \alpha_1\rangle =\alpha_{n+1}$, \cref{generalhiddenextension} we have hidden extensions
    \[3\cdot \alpha_n[k-4n]=\lambda^4 \alpha_{n+1}[k-4n-4].\]

    Moreover, there are subquotients $\Sigma^{k-4n-1}C3$ and $\Sigma^{k}T_1$. Thus, we have the total differential
    \[\begin{split}
        \delta_1(\bar{\alpha}_n[k-4n])&=3\cdot \bar{\alpha}_n[k-4n-1]+\lambda^4 \langle \bar{\alpha}_n, \alpha_1, 3\rangle [k-4n-5] \pmod{\lambda^8}\\
        &= 3\cdot \bar{\alpha}_n[k-4n-1]+ \lambda^4\alpha_{n+1}[k-4n-5] \pmod{\lambda^8}.
    \end{split}\]
        By exactness, we have the relation
        \[3\lambda\cdot \bar{\alpha}_n[k-4n-1]=-\lambda^5\alpha_{n+1}[k-4n-5]\]
        in $\pi_{*, *}((BC_3)^\infty_j)^\star/\lambda^6$.
    \end{itemize}

    We remark that the extensions we have proved do not have any crossings for degree reasons.
    \item The $d_{13}$-differentials are deduced from \cref{alpha3diff} plus the same extensions in part (1).
\end{enumerate}
\end{proof}

\begin{rmk}
     Our argument can be regarded as a generalized Leibniz rule using the synthetic homotopy framework. Recent work of Lin-Wang-Xu \cite{linwangxu2025lastkervaire} has proved a generalized Leibniz rule in the context of the Adams spectral sequences. In particular, there are similar stretch-of-differential phenomena in both cases, and we expect these ideas are applicable in a larger variety of spectral sequences.
\end{rmk}

\begin{rmk}
    The differentials in \cref{alpha2diff} and \cref{leibnizrule} would also appear in the Atiyah--Hirzebruch spectral sequence of $j_*(BC_3)$. On the other hand, \cite[1.5.20]{ravenel1986complexcobordism} provided an alternative method for computing $j_*(BC_3)$, which was used in \cite{behrens2006root} to deduce the differentials in the $BP_*$-based algebraic Atiyah--Hizebruch spectral sequence. Our proof relies purely on the analysis of cell structures and techniques in filtered spectra.
\end{rmk}

\begin{prop}\label{alpha2tobeta1^3}
    Within $(BC_3)_{-16}^{60}$, there are differentials
    \begin{align*}
        d_{24}(\alpha_2[27])&=\beta_1^3[3],\\
    d_{24}(\alpha_2[9])&=\beta_1^3[-15],\\
    d_{24}(\alpha_2[45])&=\beta_1^3[21].
    \end{align*}
\end{prop}
\begin{proof}
    By \cref{cohofBC3} and \cref{3andalpha1subquotient}, $\Sigma^{23}C\alpha_1$ and $\Sigma^{11}C\alpha_1$ are subquotients of $(B\Sigma_3)_{3}^{27}$. By \cref{alpha2attachingmap}, $\Sigma^{15}C\alpha_2$ and $\Sigma^{3}C\alpha_2$ are also subquotients of $(B\Sigma_3)_{3}^{27}$. 
    
    We first consider $(B\Sigma_3)^{27}_{19}$. For degree reasons, the $23$-cell is a bottom cell, so we consider the cofiber 
    \[S^{23}\hookrightarrow(B\Sigma_3)^{27}_{19} \twoheadrightarrow C_1.\]
    This $C_1$ fits into a cofiber sequence
    \[S^{26}\overset{\delta}{\to} \Sigma^{24}T_1 \hookrightarrow C_1 \twoheadrightarrow S^{27},\]
    where $T_1$ is defined in \cref{subquotientofBSigma3kk-5}. For degree reason, $\delta$ factors as $S^{26}\overset{h_1}{\to} S^{19} \hookrightarrow \Sigma^{24}T_1$, where this $h_1$ is either $\alpha_2$ or 0. However, since $\alpha_2 \in \langle \alpha_1, \alpha_1, 3\rangle$, by \cref{etaeta^2} we can make a choice of $\delta$ such that $h_1=0$. In particular, $C_1$ is homotopy equivalent to $S^{27}\vee \Sigma^{24}T_1$. Thus, there is a sequence of projections
    \[(B\Sigma_3)^{27}_{3}\twoheadrightarrow (B\Sigma_3)^{27}_{19}\twoheadrightarrow C_1\twoheadrightarrow \Sigma^{24}T_1.\]
    and we can consider the fiber 
    \[F_1'\hookrightarrow (B\Sigma_3)^{27}_{3}\twoheadrightarrow\Sigma^{24}T_1.\]
    Similarly, for degree reasons, the $16$-cell is a top cell of $F_1'$, so we may consider the fiber again
    \[F_2'\hookrightarrow F_1' \twoheadrightarrow S^{16}.\]
    A cell diagram of $F_2'$ is shown in \cref{F2'}.
    \begin{figure}
        \centering
        \begin{tikzpicture}[scale=0.4]
        \draw[thick](30, 5) circle (0.5);
\node at (30, 5) {27};
\draw[thick](25, 5) circle (0.5);
\node at (25, 5) {23};
\draw[thick](15, 5) circle (0.5);
\node at (15, 5) {15};
\draw[thick](12, 5) circle (0.5);
\node at (12, 5) {12};
\draw[thick](10, 5) circle (0.5);
\node at (10, 5) {11};
\draw[thick](7, 5) circle (0.5);
\node at (7, 5) {8};
\draw[thick](5, 5) circle (0.5);
\node at (5, 5) {7};
\draw[thick](2, 5) circle (0.5);
\node at (2, 5) {4};
\draw[thick](0, 5) circle (0.5);
\node at (0, 5) {3};

\draw[thick] (5.5, 5)--(6.5, 5);
\node at (6, 5.3){3};
\draw[thick] (0.5, 5)--(1.5, 5);
\node at (1, 5.3){3};
\draw[thick] (10.5, 5)--(11.5, 5);
\node at (11, 5.3){3};

\draw[thick] (15, 5.5) .. controls (18, 6.5) and (22, 6.5) .. (25, 5.5);
\node at (20, 7){$\alpha_2$};
\draw[thick] (0, 5.5) .. controls (3, 6.5) and (7, 6.5) .. (10, 5.5);
\node at (5, 7){$\alpha_2$};

\draw[thick] (0, 5.5) .. controls (1, 6) and (4, 6) .. (5, 5.5);
\node at (2.5, 6.2){$\alpha_1$};
\draw[thick] (10, 5.5) .. controls (11, 6) and (14, 6) .. (15, 5.5);
\node at (12.5, 6.2){$\alpha_1$};
\draw[thick] (25, 5.5) .. controls (26, 6) and (29, 6) .. (30, 5.5);
\node at (27.5, 6.2){$\alpha_1$};

\draw[thick] (2, 4.5) .. controls (3, 4) and (6, 4) .. (7, 4.5);
\node at (4.5, 3.8){$\alpha_1$};
\draw[thick] (12, 4.5) .. controls (11, 4) and (8, 4) .. (7, 4.5);
\node at (9.5, 3.8){$\alpha_1$};

    \end{tikzpicture}
        \caption{The cell structure of $F_2'$.}
        \label{F2'}
    \end{figure}
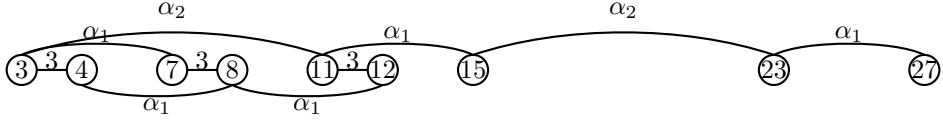

    Now consider $(F_2')_7^{27}$. It is clear that the 11-cell is a bottom cell, so consider its cofiber
    \[S^{11}\hookrightarrow (F_2')_7^{27} \twoheadrightarrow C_2.\]
    This $C_2$ fits into the cofiber sequence
    \[\Sigma^{-1}(F_2')^{27}_{15}\overset{\epsilon}{\to}\Sigma^{12}T_1\hookrightarrow C_2\twoheadrightarrow (F_2')^{27}_{15}.\]
    For degree reason, the composite $S^{14}\hookrightarrow \Sigma^{-1}(F_2')^{27}_{15}\overset{\epsilon}{\to}\Sigma^{12}T_1$ must factors as $S^{14}\overset{h_2}{\to} S^7 \hookrightarrow \Sigma^{12}T_1$, where $h_2$ is either $\alpha_2$ or 0. However, since $\alpha_2 \in \langle \alpha_1, \alpha_1, 3\rangle$, by \cref{etaeta^2} we can make a choice of $\epsilon$ such that $h_2=0$. In particular, $\epsilon$ factors through $\Sigma^{-1}(F_2')^{27}_{23}\overset{\epsilon}{\to}\Sigma^{12}T_1$. 

    Now, consider $(C_2)_{8}^{23}$. We have a cofiber sequence
    \[S^{22} \overset{\gamma}{\to}(C_2)_{8}^{15}\hookrightarrow (C_2)_{8}^{23} \twoheadrightarrow S^{23}.\]
    For degree reasons, $(C_2)_{8}^{15}\simeq \Sigma^{8} C\alpha_1\vee S^{15}$. Then, the composite $S^{22} \overset{\gamma}{\to} (C_2)_{8}^{15} \twoheadrightarrow \Sigma^{8} C\alpha_1 \twoheadrightarrow S^{12}$ has to be nullhomotopic; otherwise, it is $\beta_1$, but in the following commutative diagram 
    \[\begin{tikzcd}
S^{22} \arrow[r, "\gamma"] \arrow[rrd, "\beta_1"] & (C_2)^{15}_8 \arrow[r, two heads] & \Sigma^8 C\alpha_1 \arrow[d, two heads] \\
                                                  &                                   & S^{12} \arrow[d, "\alpha_1"]            \\
                                                  &                                   & S^9                                    
\end{tikzcd}\] 
    $\beta_1\alpha_1\neq 0$, which is a contradiction. Thus, for degree reasons, $S^{22}\hookrightarrow\Sigma^{-1}(F_2')^{27}_{23}\overset{\epsilon}{\to}\Sigma^{12}T_1$ has to factors as $S^{22}\overset{h_3}{\to}S^{7}\hookrightarrow \Sigma^{12}T_1$, where $h_3$ is either $\alpha_4$ or 0. Again by \cref{etaeta^2}, since $\alpha_4 \in \langle \bar{\alpha}_3, \alpha_1, 3\rangle$, we can make a choice of $\epsilon$ such that $h_3=0$. As a result, $\epsilon$ factors through $S^{26}\overset{\epsilon}{\to}\Sigma^{12}T_1$. 

    Finally, for degree reasons, $S^{26}\overset{\epsilon}{\to}\Sigma^{12}T_1$ must factors as $S^{26}\overset{h_4}{\to}S^7\hookrightarrow \Sigma^{12}T_1$, where $h_4$ is either $\alpha_5$ or 0. By \cref{etaeta^2}, since $\alpha_5 \in \langle \alpha_4, \alpha_1, 3\rangle$, we can make a choice of $\epsilon$ such that $h_4=0$. As a result, we can choose $\epsilon \simeq 0$, and in particular,
    \[C_2\simeq \Sigma^{12}T_1 \vee (F_2')^{27}_{15}.\]

    Now if we restrict our attention to $(F_2')_4^{27}$, it is clear that the 7-cell and the 11-cell are bottom cells, so we consider its cofiber
    \[S^{11}\vee S^7\hookrightarrow (F_2')_4^{27} \twoheadrightarrow C_3.\]
    In particular, $(C_3)_8^{27}\simeq (C_2)_8^{27}$. Now, in the cofiber sequence
    \[\Sigma^{-1}(F_2')_{15}^{27}\simeq\Sigma^{-1}(C_3)_{15}^{27}\overset{\gamma}{\to} (C_3)_4^{12} \hookrightarrow C_3 \twoheadrightarrow(F_2')_{15}^{27},\]
    for degree reasons the composite $S^{14}\hookrightarrow\Sigma^{-1}(F_2')_{15}^{27} \overset{\gamma}{\to} (C_3)_4^{12}$ has to factor as $S^{14}\overset{h_5}{\to} S^4\hookrightarrow (C_3)_4^{12}$, where this $h_5$ is either $\beta_1$ or 0. Since $\beta_1\in \langle \alpha_1, \alpha_1, \alpha_1\rangle $, by \cref{etaeta^2},we can make a choice of $\gamma$ such that $h_5=0$. In particular, $\gamma$ factor through $\Sigma^{-1}(F_2')_{23}^{27}$. By the previous discussion, $\Sigma^{-1}(F_2')_{23}^{27} \overset{\gamma}{\to} (C_3)_4^{12} \twoheadrightarrow (C_3)_8^{12}$ is nullhomotopic. Therefore, $\Sigma^{-1}(F_2')_{23}^{27} \overset{\gamma}{\to} (C_3)_4^{12}$ factors as $\Sigma^{-1}(F_2')_{23}^{27} \to S^4\hookrightarrow (C_3)_4^{12}$, where the first map is nullhomotopic for degree reasons.

    In conclusion, we have shown that $\gamma\simeq 0$, so in particular, $(C_3)\simeq (C_3)_4^{12}\vee (F_2')_{15}^{27}$. Combining the previous discussion, we have a sequence of projections
    \[F_2'\twoheadrightarrow(F_2')_7^{27} \twoheadrightarrow C_2\twoheadrightarrow \Sigma^{12}T_1,\]
    and let $F_3'$ denote the fiber of the above composite. By the splitting of $C_3$, the 4-cell is a top cell of $F_3'$, so we can further consider the fiber 
    \[F_4'\hookrightarrow F_3' \twoheadrightarrow S^4.\]
    In particular, $F_4' \hookrightarrow (B\Sigma_3)^{27}_3$ is a subcomplex based on the above discussion.

    Now $F_4'$ only consists of cells in dimension 27, 23, 15, 11, and 3. Generalizing \cref{AHdiff}, we have
    \[\begin{split}
        d_{24}(\alpha_2[27])&=\langle \alpha_2, \alpha_1, \alpha_2, \alpha_1, \alpha_2\rangle[3]\\
        &=\beta_1^3[3]. \quad (\text{\cref{beta2todabracket}})
    \end{split}
    \]
    in the Atiyah--Hirzebruch spectral sequence of $F_4'$. Pushing forward along the inclusion map, as the $\alpha_2[27]$ and $\beta_1^3[3]$ survive to the $E_{24}$-page, we must have 
    \[d_{24}(\alpha_2[27])=\beta_1^3[3]\]
    in the Atiyah--Hirzebruch spectral sequence of $(B\Sigma_3)^{27}_3$.
    
    The other two differentials follow from the same argument using the same cell-diagram pattern.
\end{proof}

\begin{rmk}\label{beta1tobeta2}
Within $(BC_3)_{-16}^{60}$, we would like to mention that there are differentials
    \begin{align*}
        d_{17}(\beta_1[16])&=\beta_2[-1],\\
        d_{17}(\beta_1[28])&=\beta_2[11],\\
        d_{17}(\beta_1[34])&=\beta_2[17],\\
        d_{17}(\beta_1[10])&=\beta_2[-7].
    \end{align*}
The relevant cell structures are fairly complicated to analyze in detail, so we only provide some ideas for why they should be true.
\end{rmk}
\begin{proof}[Ideas for the proof]
    We focus on $(B\Sigma_3)_{-1}^{16}$; the other differentials follow from the same argument. The cell structure of $(B\Sigma_3)_{-1}^{16}$ is depicted in \cref{BSigma3{16}}.
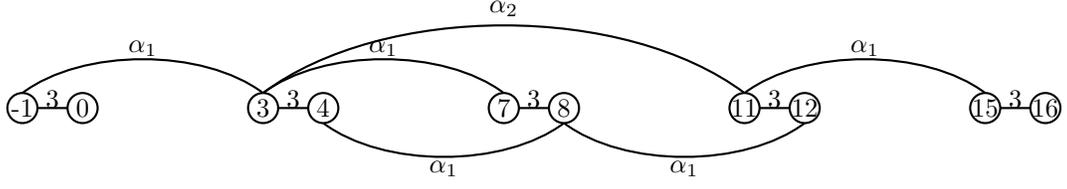
\begin{figure}
    \centering
    \begin{tikzpicture}[scale=0.4]

\draw[thick](16, -1) circle (0.5);
\node at (16, -1) {16};
\draw[thick](14, -1) circle (0.5);
\node at (14, -1) {15};
\draw[thick](8, -1) circle (0.5);
\node at (8, -1) {12};
\draw[thick](6, -1) circle (0.5);
\node at (6, -1) {11};
\draw[thick](0, -1) circle (0.5);
\node at (0, -1) {8};
\draw[thick](-2, -1) circle (0.5);
\node at (-2, -1) {7};
\draw[thick](-8, -1) circle (0.5);
\node at (-8, -1) {4};
\draw[thick](-10, -1) circle (0.5);
\node at (-10, -1) {3};
\draw[thick](-16, -1) circle (0.5);
\node at (-16, -1) {0};
\draw[thick](-18, -1) circle (0.5);
\node at (-18, -1) {-1};

\draw[thick] (-16.5, -1)--(-17.5, -1);
\node at (-17, -0.7){3};
\draw[thick] (-0.5, -1)--(-1.5, -1);
\node at (-1, -0.7){3};
\draw[thick] (-8.5, -1)--(-9.5, -1);
\node at (-9, -0.7){3};
\draw[thick] (7.5, -1)--(6.5, -1);
\node at (7, -0.7){3};
\draw[thick] (15.5, -1)--(14.5, -1);
\node at (15, -0.7){3};

\draw[thick] (0, -1.5) .. controls (-2, -3) and (-6, -3) .. (-8, -1.5);
\node at (-4, -3){$\alpha_1$};
\draw[thick] (8, -1.5) .. controls (6, -3) and (2, -3) .. (0, -1.5);
\node at (4, -3){$\alpha_1$};

\draw[thick] (-2, -0.5) .. controls (-4, 1) and (-8, 1) .. (-10, -0.5);
\node at (-6, 1){$\alpha_1$};

\draw[thick] (-18, -0.5) .. controls (-16, 1) and (-12, 1) .. (-10, -0.5);
\node at (-14, 1){$\alpha_1$};
\draw[thick] (6, -0.5) .. controls (8, 1) and (12, 1) .. (14, -0.5);
\node at (10, 1){$\alpha_1$};

\draw[thick] (6, -0.5) .. controls (2, 2.5) and (-6, 2.5) .. (-10, -0.5);
\node at (-2, 2.3){$\alpha_2$};

\end{tikzpicture}
    \caption{A cell diagram of $(B\Sigma_3)_{-1}^{16}$}
    \label{BSigma3{16}}
\end{figure}

    Since $\langle \alpha_2, \alpha_1, 3\rangle =\bar{\alpha}_3+ \{0, \alpha_3, -\alpha_3\}$, which does not contain 0, there is an obstruction from $3$-cell to 16-cell in the sense of \cite{barrattjonesmahowald1984relations}. As $\langle \alpha_2, \alpha_1, 3\rangle+\langle \alpha_1, \alpha_2, 3\rangle \ni 0$ from the Jacobi identity \cite{toda1962composition}, the only possibility is an $\pm\alpha_2$-attaching map in the indeterminacy either from the 15-cell to the 7-cell or from the 16-cell to the 8-cell.

    In the former case, a generalized version of \cref{AHdiff} implies
    \[d_{17}(\beta_1[16])=\langle\beta_1, 3, (\alpha_1, \pm\alpha_2), \begin{pmatrix}
        \alpha_2\\
        \alpha_1
    \end{pmatrix}, \alpha_1\rangle[-1].\]
    Notice 
    \[
    \begin{split}
        \langle\beta_1, 3, (\alpha_1, \pm\alpha_2), \begin{pmatrix}
        \alpha_2\\
        \alpha_1
    \end{pmatrix}, \alpha_1\rangle \cdot \alpha_1&\subseteq \langle\beta_1, 3, (\alpha_1, \pm\alpha_2), \begin{pmatrix}
        \langle\alpha_2, \alpha_1, \alpha_1\rangle\\
        \langle\alpha_1, \alpha_1, \alpha_1\rangle
    \end{pmatrix}\rangle\\
    &=\langle \beta_1, 3, \pm\alpha_2, \beta_1\rangle\\
    &=\pm\beta_2\alpha_1. \quad \text{(By \cref{beta2todabracket})}
    \end{split}
    \]
    Therefore, $\langle\beta_1, 3, (\alpha_1, \alpha_2), \begin{pmatrix}
        \alpha_2\\
        \alpha_1
    \end{pmatrix}, \alpha_1\rangle=\pm\beta_2$ and 
    \[d_{17}(\beta_1[16])=\beta_2[-1].\]

    In the latter case, a generalized version of \cref{AHdiff} implies
    \[d_{17}(\beta_1[16])=\langle\beta_1, (3, \pm\alpha_2), \begin{pmatrix}
        \alpha_1& 0& 0\\
        0& 3& \alpha_1\\
    \end{pmatrix}, \begin{pmatrix}
        \alpha_2\\
        \alpha_1\\
        3\\
    \end{pmatrix}, \alpha_1 \rangle[-1].\]
    We shuffle by 
    \[\begin{split}
         \langle\beta_1, (3, \pm\alpha_2), \begin{pmatrix}
        \alpha_1& 0& 0\\
        0& 3& \alpha_1\\
    \end{pmatrix}, \begin{pmatrix}
        \alpha_2\\
        \alpha_1\\
        3\\
    \end{pmatrix}, \alpha_1 \rangle\cdot \alpha_1 &\subseteq \langle\beta_1, (3, \pm\alpha_2), \begin{pmatrix}
        \alpha_1& 0& 0\\
        0& 3& \alpha_1\\
    \end{pmatrix}, \begin{pmatrix}
        \langle\alpha_2, \alpha_1, \alpha_1\rangle\\
        \langle\alpha_1, \alpha_1, \alpha_1\rangle\\
        \langle 3, \alpha_1, \alpha_1\rangle\\
    \end{pmatrix}\rangle\\
    &=\langle \beta_1, \pm\alpha_2, (3, \alpha_1), \begin{pmatrix}
        \beta_1\\
        \alpha_2
    \end{pmatrix}\rangle\\
    &=\langle \beta_1, \pm\alpha_2, 3, \beta_1\rangle+\langle \beta_1, \pm\alpha_2, \alpha_1, \alpha_2\rangle. \\
     \end{split}\]
     The sign of $\alpha_2$ in the brackets are the same, so by \cref{beta2todabracket}, the sum is equal to $\pm \beta_2\alpha_1\neq 0$.
     In particular, this also implies the differential
     \[d_{17}(\beta_1[16])=\beta_2[-1].\]
\end{proof} 

\begin{rmk}
    There are potential differentials
    \begin{align*}
        d_{17}(\beta_1[4])&=\beta_2[-13]\\
        d_{17}(\beta_1[22])&=\beta_2[5]
    \end{align*}
     of similar patterns. However, the cell structures are different and we cannot determine whether they indeed happen. This limits our range to $i\leq 25$.
\end{rmk}

\subsection{Some hidden extensions in the $E_\infty$-page}
Within the range $i\leq 25$, by inspection, there are extensions of the following types in the $E_\infty$-page of the Atiyah--Hirzebruch spectral sequence of $\pi_{i-j}(BC_3)_{-j}^\infty$:
 \[\begin{split}
 3\cdot 1[k]&= -\alpha_1[k-3],\\
        3\cdot \alpha_n[k]&=\alpha_{n+1}[k-4].\\     
 \end{split}
\]
The proofs are essentially the same as that of \cref{leibnizrule}. The only exceptional cases are the following.

\begin{prop}\hspace{5pt}

\begin{enumerate}
    \item In $(BC_3)_0^\infty$, there are no multiplicative extensions from $\alpha[m]$ to $\beta[0]$ for any $\alpha, \beta \in \pi_*^s$ and $m > 0$.
    \item In $(BC_3)_{-12}^\infty$ at stem $(12k-1)$, there are elements $\alpha_{1+3k}[-4]$, $\alpha_{2+3k}[-8]$, and $\bar{\alpha}_{3+3k}[-12]$ surviving to the $E_\infty$-page. They generate a subgroup $\bbZ/3\oplus \bbZ/3^{l+2}$ in $\pi_{12k-1}(BC_3)_{-12}^\infty$, where $l = \text{ord}_{3}(3k+3)$.
\end{enumerate}  
\end{prop}
\begin{proof}\hspace{5pt}
\begin{enumerate}
    \item For $(BC_3)_0^\infty$, the 0-cell splits off, inducing a splitting in the homotopy groups. 
    \item Consider the subcomplex $(B\Sigma_3)_{-12}^{-4}$. For degree reasons, the $\Sigma^{-9}C\alpha_1$ is a subcomplex of $(B\Sigma_3)_{-12}^{-4}$, so we may consider its cofiber
    \[\Sigma^{-9}C\alpha_1 \hookrightarrow (B\Sigma_3)_{-12}^{-4} \twoheadrightarrow C'. \]

    Now, $C'$ only has cells in dimension $-4$, $-8$ and $-12$. In the $E_\infty$-page of the Atiyah--Hirzebruch spectral sequence for $C'$, there are $\alpha_{1+3k}[-4]$ and $\alpha_{2+3k}[-8]$ each generating a single copy of $\bbZ /3$, and $\bar{\alpha}_{3+3k}[-12]$ generating a $\bbZ/3^l$, which equals the order of $\bar{\alpha}_{3k+3}$ in $\pi^{cl}_*$.
    
    Since $\Sigma^{-12}C\alpha_2$ and $\Sigma^{-8}C\alpha_1$ are subquotients of $C'$, $C'$ must fits into the cofiber sequence
    \[S^{-5}\overset{\alpha_1\vee\alpha_2}{\to} S^{-8}\vee S^{-12} \hookrightarrow C' \twoheadrightarrow S^{-4}.\]
    By \cref{generalhiddenextension} (b), we have a hidden extension
\[\begin{split}
    3\cdot \alpha_{1+3k}[-4]&=\langle 3, \alpha_{1+3k}, \alpha_1\rangle[-8]+\langle 3, \alpha_{1+3k}, \alpha_2\rangle[-12]\\
      &=\alpha_{2+3k}[-8]+\bar{\alpha}_{3+3k}[-12]
\end{split}\]
in $\pi_{12k-1}C'$. Since $3\cdot \alpha_{2+3k}=0$, $\{\alpha_{1+3k}[-4], \alpha_{2+3k}[-8], \bar{\alpha}_{3+3k}[-12]\}$ generates a $\bbZ/3\oplus \bbZ/3^{l+1}$. Since the corresponding $\alpha_{1+3k}[-4],\alpha_{2+3k}[-8],\bar{\alpha}_{3+3k}[-12]$ survive to the $E_\infty$-page of $(B\Sigma_3)_{-12}^{-4}$, and for degree reasons there are no other extensions in $(B\Sigma_3)_{-12}^{-4}$ related to these elements, $\bbZ/3\oplus \bbZ/3^{l+1}$ must be a summand in $\pi_{12k-1}(B\Sigma_3)_{-12}^{-4}$. Finally, since the corresponding $\alpha_{1+3k}[-4],\alpha_{2+3k}[-8],\bar{\alpha}_{3+3k}[-12]$ survive to the $E_\infty$-page of $(B\Sigma_3)_{-12}^{\infty}$, the inclusion $(B\Sigma_3)_{-12}^{-4}\hookrightarrow(B\Sigma_3)_{-12}^{\infty}$ must take $\bbZ/3\oplus \bbZ/3^{l+1}$ to a subgroup of $\pi_{12k-1}(B\Sigma_3)_{-12}^{\infty}$.
\end{enumerate}
\end{proof}

Within the range $-16\leq j\leq 16$, similar phenomena of extensions also appear for $(BC_3)_{-6}^\infty$, $(BC_3)_{6}^\infty$, $(BC_3)_{12}^\infty$. 

In conclusion, the homotopy groups of $\pi_{i-j}(BC_3)_{-j}^\infty$ are summarized in \cref{Chartlensspace}. We remark that for $0\leq i\leq 11$ one can obtain
\[\pi_{i-j}(BC_3)_{-j}^\infty \cong \pi_{i-j+18}(BC_3)_{-j+18}^\infty \quad \forall j \in \bbZ,\]
due to the James periodicity of \cref{jamesperiodicity}. This corresponds to the $\tau$-periodicity phenomenon in the $C_2$-equivariant case \cite{guillouisaksen2024c2}. In particular, combining \cref{splitofSES} and discussions in \cref{4}, the $\pi_{i, j}^{C_3}$ can be obtained for $i\leq 11$ and all $j\in \bbZ$.

\section{The Mahowald invariants and the \texorpdfstring{$\bbZ[a_\Yright]$-module}{spoke module} structures}\label{4}
As shown in \cref{splitofSES}, the long exact sequence in \cref{isotropyLES} splits if $i<2j-1$ or $i>3j$. In the range we considered, we are left to deal with the boundary homomorphisms \[M: \pi_{i-j+1}^{cl}\to \pi_{i-j}(BC_3)_{-j}^\infty\]
for $-16\leq j\leq 16$ and $2j-1\leq i\leq 3j$. This is equivalent to the range $-1\leq i-j\leq 8$ and $i-j\leq 2j$.

\subsection{Case $i-j=-1$}

\begin{prop}\label{mahinvi-j=-1}
    \[\pi_{j-1, j}^{C_3}=\begin{cases}
        0 & \mathrm{if}\ 1\leq j\leq 10\ \text{or}\ 15\leq j\leq 16\\
        \bbZ/3 & \mathrm{if}\ j=11, 13, \text{or} \ 14\\
        \bbZ/3 \oplus \bbZ/3 &\mathrm{if}\ j=12.
    \end{cases}\]
\end{prop}
\begin{proof}
    \cref{isotropyLES} gives 
\[\pi_0^{cl}\overset{M}{\to} \pi_{-1} (BC_3)_{-j}^\infty \to \pi_{j-1, j}^{C_3} \to \pi_{-1}^s=0\]

According to \cite{mahowaldravenel1993root, iriye1989images}, $M(3^k)=\alpha_k$, so the map $M: \pi_0^{cl}\to \pi_{-1}(BC_3)_{-j}^\infty$ is the projection onto the $\bbZ/3^k$ summand generated by $1[-1]$ in the Atiyah--Hirzebruch spectral sequence. The result then follows from computing the cokernel of this projection map.
\end{proof}

\subsection{Case $i-j=0$}
\begin{prop}
    \[\pi_{j, j}^{C_3}=\begin{cases}
        \bbZ\oplus \bbZ & \mathrm{if}\ j=0\\
        \bbZ & \mathrm{if}\ 1\leq j\leq 9 \ \text{or}\ 14\leq j\leq 16\\
        \bbZ \oplus \bbZ/3 &\mathrm{if}\ 10\leq j\leq 13.
    \end{cases}\]
\end{prop}
\begin{proof}
    \cref{isotropyLES} gives 
\[\pi_1^{cl}=0\to \pi_{0} (BC_3)_{-j}^\infty \to \pi_{j, j}^{C_3} \to \pi_{0}^{cl}\overset{M}{\to}\pi_{-1} (BC_3)_{-j}^\infty\]

\cref{mahinvi-j=-1} implies the kernel of $M$ is $\bbZ$. Thus, we have a short exact sequence
\[0\to \pi_{0} (BC_3)_{-j}^\infty \to \pi_{j, j}^{C_3} \to \bbZ\to 0\]
which is always split since $\bbZ$ is free. The result follows from the computation of $\pi_0(BC_3)_{-j}^\infty$ in \cref{3}.
\end{proof}

\subsection{Case $i-j=1, 4, 5\ \text{or}\ 8$}
In these cases, the long exact sequence in \cref{isotropyLES} becomes
\[\pi_{i-j+1}^{cl} =0\to \pi_{i-j} (BC_3)_{-j}^\infty \to \pi_{i, j}^{C_3} \to \pi_{i-j}^{cl}=0\]
so that we have isomorphisms 
\[\pi_{i-j} (BC_3)_{-j}^\infty \cong \pi_{i, j}^{C_3}.\]

\subsection{Case $i-j=2$}
\begin{prop}\label{i-j=2mahinv}
     \[\pi_{j+2, j}^{C_3}=\begin{cases}
        0 & \mathrm{if}\ 4\leq j\leq 10\\
        \bbZ/3 &\mathrm{if}\ 11\leq j\leq 16.
    \end{cases}\]
\end{prop}
\begin{proof}
    \cref{isotropyLES} gives 
\[\pi_3^{cl}=\bbZ/3\{\alpha_1\}\overset{M}{\to} \pi_{2} (BC_3)_{-j}^\infty \to \pi_{j+2, j}^{C_3} \to \pi_{2}^{cl}=0.\]
For $4\leq j\leq 7$, $\pi_{2} (BC_3)_{-j}^\infty=0$, so that $\pi_{j+2, j}^{C_3}=0$.

For $8\leq j\leq 10$, $\pi_{2} (BC_3)_{-j}^\infty \cong \bbZ/3\{\beta_1[-8]\}$. Since $M(\alpha_1)=\beta_1$ \cite{mahowaldravenel1993root, behrens2006root}, the map $M$ is an isomorphism. Therefore, $\pi_{j+2, j}^{C_3}=0$.

For $10\leq j\leq 16$, $\pi_{2} (BC_3)_{-j}^\infty \cong \bbZ/3\{\beta_1[-8]\}\oplus \bbZ/3\{\beta_1\alpha_1[-11]\}$, and the map $M$ is projection onto the summand $\bbZ/3\{\beta_1[-8]\}$. Thus the cokernel is just $\bbZ/3$.
\end{proof}

\subsection{Case $i-j=3$}
\begin{prop}
     \[\pi_{j+3, j}^{C_3}=\begin{cases}
        \bbZ/3 \oplus \bbZ/3 &\mathrm{if}\ j=6, 7, 12\\
        \bbZ/27 & \mathrm{if}\ j=8\\
        \bbZ/3 & \mathrm{if}\ j=9, 10, 11, 13, 14, 15\\
        \bbZ/9 &\mathrm{if}\ j=16.
    \end{cases}\]
\end{prop}
\begin{proof}
    \cref{isotropyLES} gives 
\[\pi_4^{cl}=0 \to \pi_{3} (BC_3)_{-j}^\infty \to \pi_{j+3, j}^{C_3} \to \pi_{3}^{cl} \to \pi_2(BC_3)_{-j}^\infty.\]

Consider the projection $(BC_3)_{-6}^\infty \to (BC_3)_{-5}^\infty$ and the inclusion $(BC_3)_{-\infty}^{-7}\to (BC_3)_{-\infty}^{-6}$, which induce maps of short exact sequences
\[\begin{tikzcd}
0 \arrow[r]         & \pi_3(BC_3)_{-5}^\infty=\bbZ/3\{\alpha_2[-4]\}  \arrow[r]           & {\pi_{8, 5}^{C_3}} \arrow[r]          & \pi_3^{cl} \arrow[r] \arrow[l, dashed, bend right] & 0                       \\
\pi_4^{cl}=0 \arrow[r] & \pi_3(BC_3)_{-6}^\infty=\bbZ/3\{\alpha_2[-4]\}  \arrow[r] \arrow[u] & {\pi_{9,6}^{C_3}} \arrow[u] \arrow[r] & \pi_3^{cl} \arrow[u] \arrow[r]                & \pi_2(BC_3)_{-6}^\infty=0.
\end{tikzcd}\]

Since  $j=5, i=8$ is in the splitting range of \cref{splitofSES}, the top short exact sequence splits.

By the naturality of the Atiyah--Hirzebruch spectral sequence, the left vertical map induced by the projection is an isomorphism. Thus, the five lemma implies that the middle vertical map is an isomorphism and the bottom short exact sequence splits as well. Similar argument shows that the short exact sequence for $j=7$ also splits.

When $j\geq8$, we have
\[\pi_4^{cl}=0 \to \pi_{3} (BC_3)_{-j}^\infty \to \pi_{j+3, j}^{C_3} \to \pi_{3}^{cl}=\bbZ/3\{\alpha_1\} \overset{M}{\to} \bbZ/3\{\beta_1\}\subset \pi_2(BC_3)_{-j}^\infty.\]
\cref{i-j=2mahinv} implies the map $M$ is an injection. Therefore, $\pi_{j+3, j}^{C_3} \cong \pi_{3} (BC_3)_{-j}^\infty$. The result follows from computations of $\pi_3(BC_3)_{-j}^\infty$ in \cref{3}.
\end{proof}

\subsection{Case $i-j=6, 7$}
\begin{prop}
    \[\pi_{j+6, j}^{C_3}=0\quad \mathrm{for}\ 12\leq j\leq 16,\]
    \[\pi_{j+7, j}^{C_3}=\begin{cases}
        \bbZ/3 \oplus \bbZ/3 &\mathrm{if}\ j=14, 15\\
        \bbZ/27 \oplus \bbZ/3 & \mathrm{if}\ j=16.\\
    \end{cases}\]
\end{prop}
\begin{proof}
\cref{isotropyLES} gives
\[\pi_8^{cl}=0 \to \pi_{7} (BC_3)_{-j}^\infty \to \pi_{j+7, j}^{C_3} \to \pi_{7}^{cl} \to \pi_6(BC_3)_{-j}^\infty \to \pi_{j+6, j}^{C_3}\to \pi_6^{cl}=0 \]

Within the range $j\leq 16$, $\pi_6(BC_3)_{-j}^\infty=0$. Thus, \[\pi_{j+6, j}^{C_3}=0 \quad \mathrm{for}\ 12\leq j\leq 16,\]
and there is a short exact sequence
\[0 \to \pi_{7} (BC_3)_{-j}^\infty \to \pi_{j+7, j}^{C_3} \to \pi_{7}^{cl} \to 0.\]

Since $j=13, i=20$ is in the splitting range of \cref{splitofSES}, the top short exact sequence splits.
\[\begin{tikzcd}
0 \arrow[r]         & \pi_7(BC_3)_{-13}^\infty=\bbZ/3\{\alpha_5[-12]\}\oplus \bbZ/3\{\beta_1\alpha_1[-6]\}  \arrow[r]           & {\pi_{20, 13}^{C_3}} \arrow[r]          & \pi_7^{cl} \arrow[r] \arrow[l, dashed, bend right] & 0                       \\
0 \arrow[r] & \pi_7(BC_3)_{-14}^\infty=\bbZ/3\{\alpha_5[-12]\} \arrow[r] \arrow[u] & {\pi_{21,14}^{C_3}} \arrow[u] \arrow[r] & \pi_7^{cl} \arrow[u] \arrow[r]                & 0.
\end{tikzcd}\]

By the five lemma, $\pi^{C_3}_{21, 14}\to \pi^{C_3}_{20, 13}$ is an injection. The commutative diagram implies $\pi_7^{cl}$ is a direct summand of $\pi^{C_3}_{21, 14}$, so that the bottom row also splits. Similar argument applies for $j=15,16$.
\end{proof}

\subsection{The $\bbZ[a_\Yright]$-module structure}\label{aspokemodule}
Having completely understood the long exact sequence in \cref{isotropyLES}, we also obtained the information of the geometric fixed point map
\[\Phi^{C_3}: \pi^{C_3}_{i, j}\to \pi^{cl}_{i-j}\]
as part of the long exact sequence.

In this subsection, we will describe the $a_\Yright$-action on the $C_3$-equivariant stable stems. This in turn gives the information for the underlying map
\[Res: \pi_{i, j}^{C_3}\to \pi_i^{cl}.\]

By \cref{isotropyLES}, $\pi^{C_3}_{i, j}$ fits into the following short exact sequence
\[0\to \pi_{i-j}^{cl}(BC_{3})^{\infty}_{-j}/\im M \to \pi^{C_3}_{i, j} \overset{\Phi^{C_3}}{\to}\ker M \to 0\]
where $\ker M$ is a subgroup of $\pi^{cl}_{i-j}$. As discussed in the previous subsections, this short exact sequence always splits in the range we consider.

\begin{prop}\label{a^1/2 action}
    The $\bbZ[a_\Yright]$-module structure on $\pi^{C_3}_{i,j}$ are characterized as follows:
\begin{enumerate}
    \item The $\ker M$ components are $a_\Yright$-free.  
    \item The $\pi_{i-j}^{cl}(BC_{3})^{\infty}_{-j}/
    \im M$ components are $a_\Yright$-torsion. If one such element takes the form $\alpha[n]$ as discussed in \cref{AHSSsetup}, then 
    \[(a_\Yright)^{j+n} \cdot \alpha[n]\neq 0, \quad (a_\Yright)^{j+n+1} \cdot \alpha[n] = 0.\]
\end{enumerate}
\end{prop}

\begin{proof} Since $\ker M\subseteq \pi^{cl}_{i-j}$, the first part follows from the observation in \cite{balderrama2025cpnequivariantmahowaldinvariants} that 
\[\Phi^{C_3}(\pi^{C_3}_{i,j}) \simeq (\pi^{C_3}_{i,j})[a_\Yright]^{-1}.\]

For the second part, let $ \alpha[n] \in \pi_{i-j}(BC_{3})^{\infty}_{-j}$ be represented by a $C_3$-equivariant map $S^{i-j+k\lambda+\epsilon\Yright}\to {EC_3}_+$, for some $\epsilon=0, 1$. Consider the following commutative diagram
\begin{center}
\begin{tikzcd}
                                    & {S^{i-j+k\lambda+\epsilon \Yright}} \arrow[d, "{\alpha[n]}"'] \arrow[rd, "{a_\Yright\cdot \alpha[n]}"] &                               \\
{EC_3}_{+} \wedge {C_3}_+ \arrow[r] & {EC_3}_{+} \arrow[r]                                                             & {EC_3}_{+} \wedge S^{\Yright}.
\end{tikzcd}
\end{center}
Note that $a_\Yright \cdot \alpha[n] = 0$ if and only if $\alpha[n]$ factors through ${EC_3}_+\wedge {C_3}_+\to {EC_3}_+$, which is adjoint to the bottom-cell inclusion $S^{-j} \to (BC_3)^{\infty}_{-j}$. Thus, $a_\Yright \cdot \alpha[n] = 0$ if and only if $-j=n$. 

If $a_\Yright \cdot \alpha[n] \neq 0$, upon taking $C_3$-fixed point the right hand triangle reduces to the following non-equivariant commutative diagram
\[\begin{tikzcd}
S^{i-j} \arrow[d, "{\alpha[n]}"'] \arrow[rd, "{a_\Yright\cdot \alpha[n]}"] &                      \\
(BC_3)_{-j}^\infty \arrow[r, two heads]                                    & (BC_3)_{-j+1}^\infty.
\end{tikzcd}\]
where the bottom map is the natural projection. By our naming convention, the $a_\Yright \cdot\alpha[n]$ corresponds to the $\alpha[n] \in \pi_{i-j}(BC_{3})^{\infty}_{-j+1}$. If $\alpha[n]$ is a surviving permanent cycle in $\pi_{i-j}^{cl}(BC_{3})^{\infty}_{-j}$, then it must also be a surviving permanent cycle in $\pi_{i-j}^{cl}(BC_{3})^{\infty}_{-j+1}$ for degree reasons. Therefore, $a_\Yright^{r} \cdot \alpha[n]\neq 0$ until $r= n+j+1$ in which case we are considering $\pi_{i-j}^{cl}(BC_{3})^{\infty}_{n+1}$ and $\alpha[n]$ is not an element even in the $E_1$-page for filtration reasons.
\end{proof}

\begin{cor}\label{charofres}
The underlying map    
\[Res: \pi_{i, j}^{C_3}\to \pi_i^{cl}.\]
can be characterized as follows:
\begin{enumerate}
    \item Suppose $\alpha\in \pi_{i-j}^{cl}$ is in $\ker M$, then viewed as an element in $\pi_{i, j}^{C_3}$, its image under $Res$ is trivial unless $i=|M(\alpha)|$, in which case the image is $M(\alpha)$.
    \item $\alpha[n] \in \pi_{i-j}^{cl}(BC_{3})^{\infty}_{-j}/\im M$ has a non-trivial image under $Res$ if and only if it supports a non-trivial differential in the Atiyah--Hirzebruch spectral sequence of $(BC_{3})^{\infty}_{-j-1}$.
\end{enumerate}
\end{cor}
\begin{proof}
The long exact sequence
    \[\cdots \xrightarrow{tr} \pi^{C_3}_{i,j}(X) \xrightarrow{\cdot a_{\Yright}} \pi^{C_3}_{i-1,j-1}(X) \xrightarrow{Res} \pi^{cl}_{i-1}(X) \xrightarrow{tr} \pi^{C_3}_{i-1,j}(X) \xrightarrow{a_{\Yright}} \cdots\]
    implies that an element has nonzero image under $Res$ if and only if it is not $a_\Yright$-divisible.
\begin{enumerate}  
    \item This is an equivalent definition of the classical $3$-primary Mahowald invariant \cite{brunergreenlees1995bredon}.
    \item  This follows from \cref{a^1/2 action}.
\end{enumerate}
\end{proof}

\section{Other primary and torsion-free information}\label{5}
The $C_3$-equivariant stable homotopy category splits after inverting $3$ \cite{greenleesmay1995generalized, bonventreguilloustapleton2022kug, liu2023computingequivarianthomotopysplitting}. In particular, there is a splitting
\begin{equation}\label{otherprimesplitting}
    S^0[\frac{1}{3}]\simeq {EC_3}_+[\frac{1}{3}]\vee \widetilde{EC_3}[\frac{1}{3}].
\end{equation}
\begin{prop}\label{otherprime}
   Let $p$ be any prime other than 3.
    \begin{enumerate}
        \item If $j$ is odd, $(BC_3)_{j}^\infty$ is $p$-adically contractible.
        \item If $j$ is even, then there is a $p$-adic equivalence
        \[(BC_3)_{j}^\infty \simeq S^{j}.\]
    \end{enumerate}
\end{prop}
\begin{proof}
    Because there is an alternating pattern of multiplication by $3$ in the chain complex of $(BC_3)_j^\infty$ with $\bbZ$ coefficients \cite{hatcher2002algebraictopology}, the $H\mathbb{F}_p$-homology of $(BC_3)_{j}^\infty$ is trivial if $j$ is odd, or $\mathbb{F}_p$ in degree $j$ if $j$ is even by the universal coefficient theorem. 
    
    The former case implies $(BC_3)_{j}^\infty$ is $p$-adically contractible, while in the latter case, the inclusion of the bottom cell map induces an isomorphism in $H\mathbb{F}_p$ homology, which gives us the desired equivalence.
\end{proof}

\begin{theorem}\label{otherpcomplete}
    Let $p$ be any prime other than 3. Then
    \[(\pi_{i, j}^{C_3})_p^\wedge \cong \begin{cases}
        (\pi_{i-j}^{cl})_p^\wedge & \mathrm{if}\ $j$\ \mathrm{is}\ \mathrm{odd}\\
        (\pi_{i}^{cl})_p^\wedge \oplus (\pi_{i-j}^{cl})_p^\wedge & \mathrm{if}\ $j$\ \mathrm{is}\ \mathrm{even}\\
    \end{cases}.\]
\end{theorem}
\begin{proof}
    Using the splitting in \cref{otherprimesplitting}, 
    \[(\pi_{i, j}^{C_3})_p^\wedge \cong \pi_{i, j}^{C_3}({EC_3}_+)_p^\wedge \oplus \pi_{i,j}^{C_3}(\widetilde{EC_3})_p^\wedge.\]
    By \cref{otherprime}, 
    \[\begin{split}
        \pi_{i, j}^{C_3}({EC_3}_+)_p^\wedge &\cong {[S^{i-j+j\Yright}, {EC_3}_+]^{C_3}}_p^\wedge\\
        &=[S^{i-j}, ({EC_3}_+\wedge S^{-j\Yright})^{C_3}]_p^\wedge\\
        &=\pi_{i-j}((BC_3)^\infty_{-j})_p^\wedge\\
        &=\begin{cases}
        0 & \mathrm{if}\ $j$\ \mathrm{is}\ \mathrm{odd}\\
        (\pi_{i}^{cl})_p^\wedge & \mathrm{if}\ $j$\ \mathrm{is}\ \mathrm{even}.\\
    \end{cases}
    \end{split}\]
    Moreover, 
    \[\begin{split}
        \pi_{i,j}^{C_3}(\widetilde{EC_3})&=[S^{i-j+j\lambda}, \widetilde{EC_3}]^{C_3}\\
        &=[S^{i-j}, (\widetilde{EC_3}\wedge S^{-j\lambda})^{C_3}]\\
        &=\pi_{i-j}^{cl}.
    \end{split}\]
    As a result,
    \[(\pi_{i, j}^{C_3})_p^\wedge \cong \begin{cases}
        (\pi_{i-j}^{cl})_p^\wedge & \mathrm{if}\ $j$\ \mathrm{is}\ \mathrm{odd}\\
        (\pi_{i}^{cl})_p^\wedge \oplus (\pi_{i-j}^{cl})_p^\wedge & \mathrm{if}\ $j$\ \mathrm{is}\ \mathrm{even}\\
    \end{cases}\]
\end{proof}

We obtain the last piece of information needed, which agrees with that of \cite{greenleesquigley2023ranks}. 

\begin{cor}\label{integral}
    The group structure of the torsion-free part for $\pi_{i, j}^{C_3}$ is 
\[\pi_{i, j}^{C_3}/\tors=\begin{cases}
    \bbZ^2\quad \text{if}\ i=j=0\\
    \bbZ \quad \text{if}\ i=0, j=2k\\
    \bbZ \quad \text{if}\ i=j\neq 0\\
    0 \quad \text{otherwise}.
\end{cases}\]
\end{cor}
\begin{proof}
\cref{otherpcomplete} shows that there is a summand of $\bbZ_p^\wedge$ if and only if (1) $j$ is odd and $i=j$, or (2) $j$ is even, $i=j$ or $i=0$. The corollary follows by combining the $p$-completed information for all $p$. 
\end{proof}

\section{Charts and tables}\label{6}

\subsection{Classical 3-primary stable stems.}
\hspace{5pt}

\cref{3primaryAdamsSS} records data of the 3-primary Adams spectral sequence up to stem 35, starting from $E_2$-page. Here is a key for reading the charts:
\begin{enumerate}
    \item The first column displays the Adams filtrations;
    \item The second column displays the stems;
    \item The third column displays the elements on the $E_2$-page, named under the Adams convention;
    \item The fourth column displays the names of homotopy elements;
    \item The fifth column displays the Adams differentials.
\end{enumerate}

\subsection{Atiyah--Hirzebruch spectral sequences}
\hspace{5pt}

We only include the charts of the Atiyah--Hirzebruch spectral sequence for $(BC_3)_{-16}^\infty$. See \cite{houzhang2025AHSS} for the data of other stunted lens spaces. \cref{fig:E2AHSSBSigma3,fig:E2AHSSX} depicts the Atiyah--Hirzebruch spectral sequences (completed at 3) that converge to homotopy groups of $(B\Sigma_3)^{\infty}_{-16}$ and $X_{-16}^\infty$, respectively, starting from the $E_2$-pages. The horizontal axis is the stem, and the vertical axis is the cellular filtration (inverted). It is graded in terms of the Adams grading: $d_r$ differentials goes left by $1$ and goes up by $r$.
\begin{enumerate}
    \item Dots indicate copies of $\bbZ/3$;
    \item Filled squares indicate copies of $\bbZ_3$;
    \item Hollow squares with number k amid indicate copies of $\bbZ/3^k$;
    \item Green arrows of slope -4 indicate $d_4$-differentials;
    \item Blue arrows of slope -5 indicate $d_5$-differentials;
    \item Red arrows of slope -8 indicate $d_8$-differentials;
    \item Orange arrows of slope -9 indicate $d_9$-differentials;
    \item Brown arrows of slope -12 indicate $d_{12}$-differentials;
    \item Magenta arrows of slope -13 indicate $d_{13}$-differentials;
    \item Cyan arrows of slope -17 indicate $d_{17}$-differentials;
    \item Cyan arrows of slope -24 indicate $d_{24}$-differentials;
    \item Dashed cyan arrow of slope -17 indicate potential $d_{17}
    $-differentials.
\end{enumerate}

\cref{fig:EinftyBSigma3,fig:EinftyX} depicts the $E_\infty$-pages of Atiyah--Hirzebruch spectral sequences (completed at 3) that converge to homotopy groups of $(B\Sigma_3)^{\infty}_{-16}$ and $X_{-16}^\infty$, respectively.
\begin{enumerate}
    \item Dots indicate copies of $\bbZ/3$.
    \item Filled squares indicate copies of $\bbZ_3$.
    \item Hollow squares with number k amid indicate copies of $\bbZ/3^k$.
    \item Dashed lines indicate 3-extension.
\end{enumerate}

\subsection{Homotopy groups of stunted lens spaces}
\hspace{5pt}

\cref{Chartlensspace} lists the 3-completed homotopy groups of stunted lens spaces $\pi_{i-j}(BC_3)^\infty_{-j}$ for $-16 \leq j \leq 16$ and $i \leq 25$. The table is horizontally indexed by $i$ and vertically by $j$. Each number $n^k$ in the table represents the group $(\bbZ/n)^k$. For every cell, the corresponding homotopy group is the direct sum of all such terms that show up.

\subsection{$C_3$-equivariant stable stems}
\hspace{5pt}

\cref{ChartC_3stablestem} lists the 3-completed $C_3$-equivariant stable stems in the range $-16 \leq w \leq 16$ and $s \leq 25$. The conventions for depicting \cref{ChartC_3stablestem} follows from that of \cref{Chartlensspace}. In particular, the orange cells corresponds to the $C_3$-stable stems which fall in the range discussed in \cref{4}.

\newpage
\begin{table}
\begin{tabular}{@{}|l|l|l|l|l|@{}}
\hline
filtration &  stem  & name & name in $\pi_*^{cl}$ & differentials                                                                     \\\hline 
$n$ ($n\geq 0$)  & 0  & $a_0^n$ & $3^n$ & \\\hline
1  & 3  & $h_0$ & $\alpha_1$ & \\\hline
2  & 7  & $a_1h_0$ & $\alpha_2$ & \\\hline
2  & 10  & $b_{10}$ & $\beta_1$ & \\\hline
3  & 10  & $a_0b_{10}$ &  & $d_2(h_1)=a_0b_{10}$ \\
1  & 11  & $h_1$ &  & \\\hline
2  & 11  & $a_0h_1$ & $\bar{\alpha}_3$ & \\
\hline
\end{tabular}
\caption{Table of the 3-primary Adams spectral sequence of $S^0$}
\label{3primaryAdamsSS}
\end{table}

\begin{figure}[p]
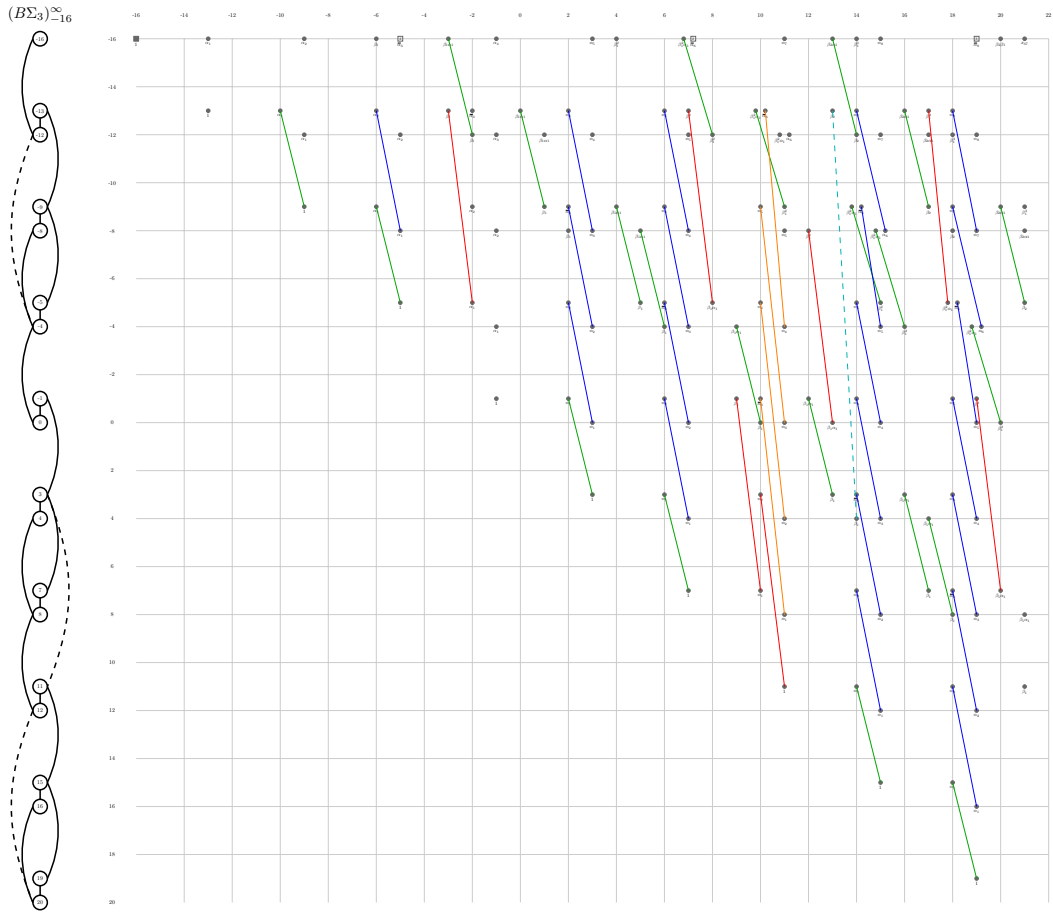

  \centering
  \begin{adjustbox}{max width=\textwidth, max height=\textheight, keepaspectratio}
              %

  \end{adjustbox}
  \caption{$E_2$-AHSS of $(B\Sigma_3)^{\infty}_{-16}$}
  \label{fig:E2AHSSBSigma3}
\end{figure}

\newpage

\begin{figure}[p]
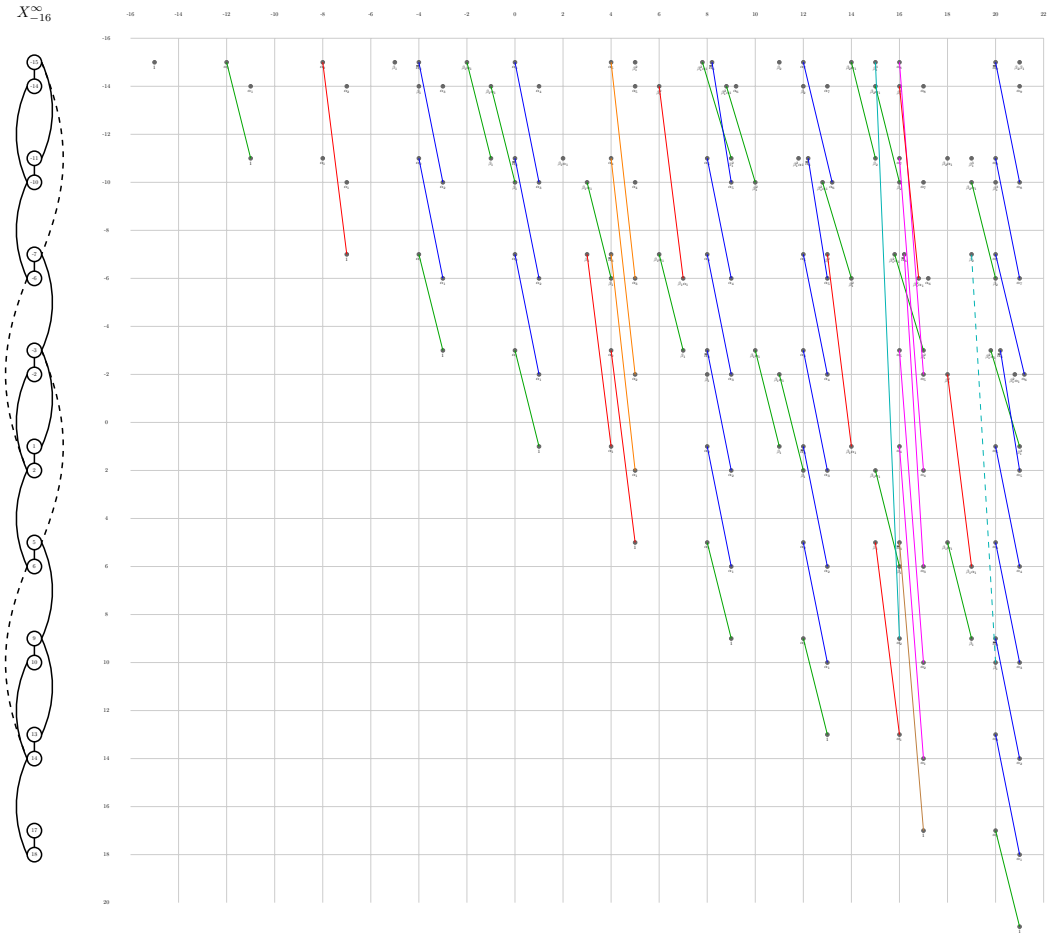

    \centering
    \begin{adjustbox}{max width=\textwidth, max height=\textheight, keepaspectratio}
        %
    \end{adjustbox}
    \caption{$E_2$-AHSS of $X^{\infty}_{-16}$}
    \label{fig:E2AHSSX}
\end{figure}

\newpage

\begin{figure}[p]
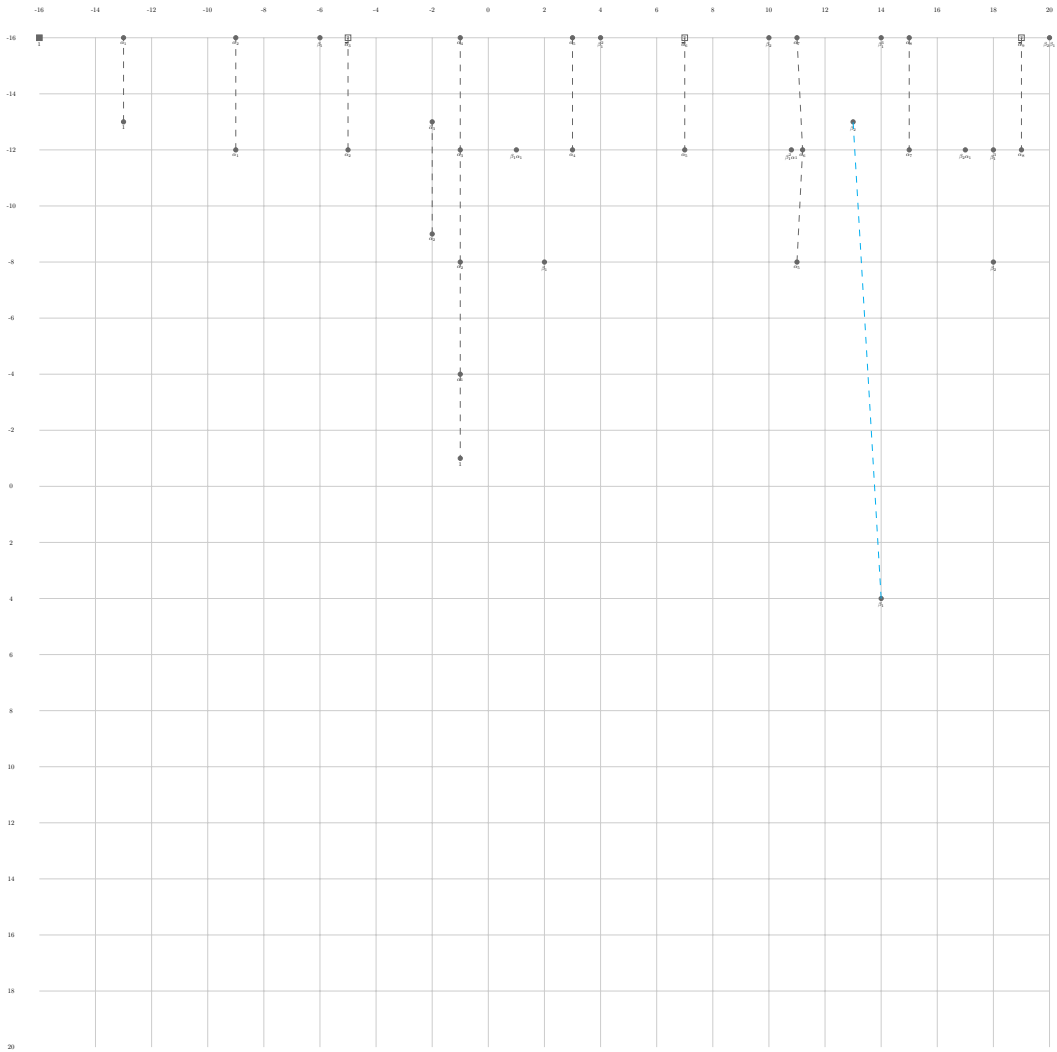

    \centering
    \begin{adjustbox}{max width=\textwidth, max height=\textheight, keepaspectratio}
           %
    
    \end{adjustbox}
    \caption{$E_\infty$-AHSS of $(B\Sigma_3)^{\infty}_{-16}$}
    \label{fig:EinftyBSigma3}
\end{figure}

\newpage

\begin{figure}[p]
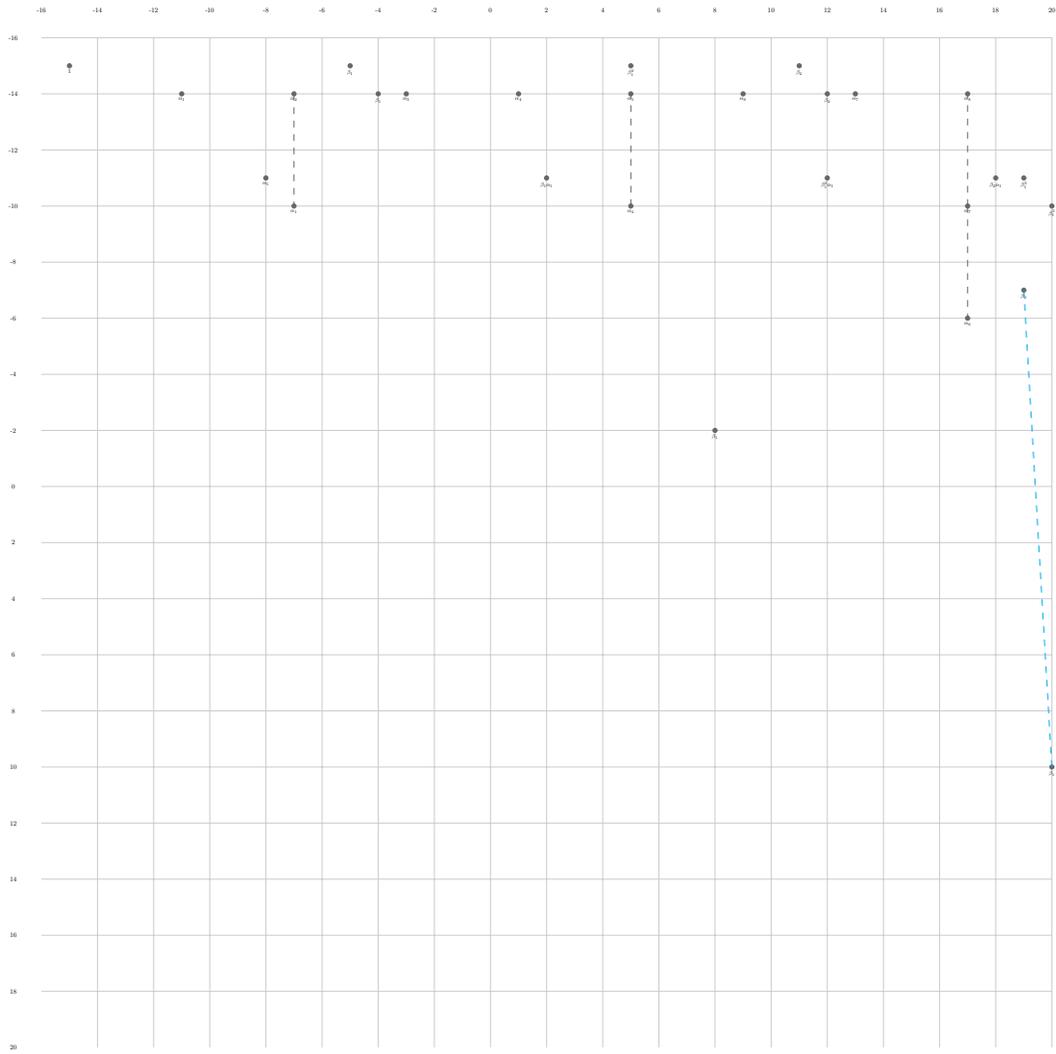

    \centering
    \begin{adjustbox}{max width=\textwidth, max height=\textheight, keepaspectratio}
        %
   
    \end{adjustbox}
    \caption{$E_\infty$-AHSS of $X^{\infty}_{-16}$}
    \label{fig:EinftyX}
\end{figure}

\begin{landscape}

\begin{table}[p]
    \centering
    \footnotesize
    \begin{adjustbox}{width=\linewidth,totalheight=\textheight,keepaspectratio}
      \begin{tabular}{|c|*{26}{>{\centering\arraybackslash}p{.045\linewidth}|}}
        \toprule

    $\pi_{i-j}(BC_3)^\infty_{-j}$ & i=0 & i=1 & i=2 & i=3 & i=4 & i=5 & i=6 & i=7 &  i=8 & i=9 & i=10 & i=11 & i=12 &  i=13 & i=14 & i=15 & i=16 & i=17 &  i=18 & i=19 & i=20 & i=21 & i=22 & i=23 & i=24 & i=25 \\
  
  \midrule
    j=16 & $\bbZ_{3}$ & 3 & 0 & 9 & 0 & 3 & 0 & 9 & 3 & 9 & 3 & 3,27 & 3 & 3 & 9 & 243 & 0 & $3^2$ & $3^2$ & 9 & 3 & 3,9 & 0 & 27 & 3 & 3 \\
  \midrule
    j=15 & 3 & 0 & 3 & 0 & 3 & 0 & 3 & 3 & 9 & 0 & $3^2$ & 3 & 3 & 3,9 & 81 & 0 & $3^2$ & $3^2$ & 3 & 0 & 3,9 & 0 & 3 & $3^2$ & 3 & 0 \\
  \midrule
    j=14 & $\bbZ_{3}$ & 3 & 0 & 9 & 0 & 3 & 3 & 27 & 0 & 3 & 3 & 27 & 3,9 & 3,81 & 0 & 3,9 & $3^2$ & 3 & 0 & 27 & 0 & 3 & $3^2$ & 3,27 & 0 & 3,9 \\
  \midrule
    j=13 & 3 & 0 & 3 & 0 & 3 & 3 & 9 & 0 & 3 & 0 & 3 & 3,9 & 3,81 & 3 & $3^2$ & $3^2$ & 3 & 0 & 9 & 0 & $3^2$ & $3^2$ & $3^2$ & 3 & 3,9 & 3 \\
  \midrule
    j=12 & $\bbZ_{3}$ & 3 & 0 & $3^{2}$ & 3 & 9 & 0 & $3^2$ & 0 & 3 & 3,9 & $3^2$,81 & 3 & $3^3$ & $3^2$ & 3 & 0 & 9 & 0 & $3^3$ & $3^3$ & $3^2$ & 3 & $3^3,27$ & 3 & 3\\
  \midrule
    j=11 & 3 & 0 & 3 & 3 & 9 & 0 & 3 & 0 & 3 & 9 & 3,27 & 3 & $3^2$ & $3^2$ & 3 & 0 & 9 & 0 & $3^2$ & $3^2$ & $3^2$ & 3 & 3,9, & 3 & 3 & 3 \\
  \midrule
    j=10 & $\bbZ_{3}$ & 3 & 0 & 9 & 0 & 3 & 0 & 9 & 9 & 27 & 3 & 3,27 & 3 & 3 & 0 & 27 & 0 & $3^2$ & $3^2$ & 9 & 3 & 3,9 & 0 & 27 & 3 & 3 \\
\bottomrule

\end{tabular}
\end{adjustbox}
\caption{Homotopy groups of stunted lens space}
\label{Chartlensspace}
\end{table}
\end{landscape}
\restoregeometry

\begin{landscape}

\begin{table}[p]
    \centering
    \tiny
    \begin{adjustbox}{width=\linewidth,totalheight=\textheight,keepaspectratio}
      \begin{tabular}{|c|*{26}{>{\centering\arraybackslash}p{.045\linewidth}|}}
        \toprule

    $\pi_{i,j}^{C_3}$ & i=0 & i=1 & i=2 & i=3 & i=4 & i=5 & i=6 & i=7 &  i=8 & i=9 & i=10 & i=11 & i=12 &  i=13 & i=14 & i=15 & i=16 & i=17 &  i=18 & i=19 & i=20 & i=21 & i=22 & i=23 & i=24 & i=25 \\
  \midrule
    j=16 & $\bbZ_{3}$ & 3 & 0 & 9 & 0 & 3 & 0 & 9 & 3 & 9 & 3 & 3,27 & 3 & 3 & 9 & \textcolor{orange}{0} & \textcolor{orange}{$\bbZ_{3}$} & \textcolor{orange}{$3^2$}  & \textcolor{orange}{3}  & \textcolor{orange}{9} & \textcolor{orange}{3} & \textcolor{orange}{3,9} & \textcolor{orange}{0} & \textcolor{orange}{3,27} & \textcolor{orange}{3} & 3 \\
  \midrule
    j=15 & 3 & 0 & 3 & 0 & 3 & 0 & 3 & 3 & 9 & 0 & $3^2$ & 3 & 3 & 3,9 & \textcolor{orange}{0} & \textcolor{orange}{$\bbZ_{3}$} & \textcolor{orange}{$3^2$}  & \textcolor{orange}{3}  & \textcolor{orange}{3} & \textcolor{orange}{0} & \textcolor{orange}{3,9} & \textcolor{orange}{0} & \textcolor{orange}{$3^2$} & $3^2$ & 3 & $\pi_{10}$ \\
  \midrule
    j=14 & $\bbZ_{3}$  & 3 & 0 & 9 & 0 & 3 & 3 & 27 & 0 & 3 & 3 & 27 & 3,9 & \textcolor{orange}{3} & \textcolor{orange}{$\bbZ_{3}$} & \textcolor{orange}{3,9}  & \textcolor{orange}{3}  & \textcolor{orange}{3} & \textcolor{orange}{0} & \textcolor{orange}{27} & \textcolor{orange}{0} & \textcolor{orange}{$3^2$} & $3^2$ & 3,27 & $\pi_{10}$ & 3,9,$\pi_{11}$ \\
  \midrule
    j=13 & 3 & 0 & 3 & 0 & 3 & 3 & 9 & 0 & 3 & 0 & 3 & 3,9 & \textcolor{orange}{3} & \textcolor{orange}{3,$\bbZ_{3}$} & \textcolor{orange}{$3^2$}  & \textcolor{orange}{3}  & \textcolor{orange}{3} & \textcolor{orange}{0} & \textcolor{orange}{9} & \textcolor{orange}{0} & $3^2$ & $3^2$ & $3^2$ & 3,$\pi_{10}$ & 3,9,$\pi_{11}$ & 3 \\
\bottomrule

\end{tabular}
\end{adjustbox}
\caption{spoke-graded $C_3$-equivariant stable stem}
\label{ChartC_3stablestem}

\end{table}
\end{landscape}

\begingroup
\raggedright
\bibliography{refs}
\bibliographystyle{alpha}
\endgroup

\end{document}